\documentclass[twoside,12pt]{article}
\DeclareMathSizes{20}{30}{16}{12}
\usepackage{float}
\usepackage{amsmath}
\usepackage{mathtools}
\usepackage{mathrsfs}
\usepackage{stmaryrd}
\usepackage{bbm}
\usepackage{yfonts}
\usepackage{amsfonts}
\usepackage{tcolorbox}
\usepackage{tikz-cd}
\usepackage[utf8]{inputenc}

\usepackage{caption}
\usepackage{indentfirst}
\usepackage[Symbol]{upgreek}
\usepackage{enumerate}   
\usepackage{upref}

\usepackage{graphicx}
\usepackage[margin=0.8 in]{geometry}
\usepackage{fancyhdr}
\usepackage{hyperref}
\hypersetup{
    colorlinks,
    citecolor=black,
    filecolor=black,
    linkcolor=black,
    urlcolor=black
}
\newcommand{\N}{\mathbb{N}}
\newcommand{\Z}{\mathbb{Z}}

\newcommand{\ra}{\rightarrow}

\newcommand{\calA}{\mathcal{A}}
\newcommand{\calI}{\mathcal{I}}

\newcommand{\calE}{\mathcal{E}}

\newcommand{\calC}{\mathcal{C}}
\newcommand{\calT}{\mathcal{T}}
\newcommand{\calD}{\mathcal{D}}

\newcommand{\calR}{\mathcal{R}}

\newcommand{\calM}{\mathcal{M}}
\newcommand{\calN}{\mathcal{N}}
\newcommand{\calP}{\mathcal{P}}
\newcommand{\calQ}{\mathcal{Q}}
\newcommand{\calV}{\mathcal{V}}

\newcommand{\op}{\operatorname}
\newcommand{\w}{\widehat}

\newcommand{\Ox}{\mathcal{O}}
\newcommand{\F}{\mathcal{F}}
\newcommand{\calG}{\mathcal{G}}

\newcommand{\ov}{\overline}

\newcommand{\frakX}{\mathfrak{X}}
\newcommand{\frakE}{\mathfrak{E}}

\newcommand{\frakU}{\mathfrak{U}}
\newcommand{\frakV}{\mathfrak{V}}

\newcommand{\frakS}{\mathcal{S}}
\newcommand{\frakY}{\mathfrak{Y}}
\newcommand{\frakT}{\mathfrak{T}}
\newcommand{\frakZ}{\mathfrak{Z}}

\newcommand{\frakG}{\mathfrak{G}}

\newcommand{\Sp}{\operatorname{Spec}}

\long\def\/*#1*/{}

\usepackage{unicode-math}
\usepackage{fontspec}
\usepackage{sectsty}
\sectionfont{\centering}
\usepackage[toc,page]{appendix}
\usepackage{amsthm}

\newtheorem*{proof*}{Proof}

\newtheorem{theorem}[subsection]{Theorem}
\newtheorem{theorem*}{Theorem}
\newtheorem{proposition}[subsection]{Proposition}
\newtheorem{corollaire}[subsection]{Corollary}
\newtheorem{lemma}[subsection]{Lemma}
\newtheorem{lemmadef}[subsection]{Lemma and definition}

\theoremstyle{definition}

\newtheorem{definition}[subsection]{Definition}
\newtheorem{remark}[subsection]{Remarks}

\theoremstyle{definition}
\newtheorem{parag}[subsection]{}

\numberwithin{equation}{section} % default: (section.equation)

\makeatletter
\renewcommand{\theequation}{%
  \ifnum\value{subsubsection}>0
    \thesubsubsection.\arabic{equation}%
  \else
    \thesubsection.\arabic{equation}%
  \fi
}
\makeatother

\newtheorem{proposition1}[subsubsection]{Proposition}

\title{A Topos-Theoretic Approach to the Logarithmic Cartier Transform}
\author{Sami Fersi}
\date{}

\AtEndDocument{\bigskip{\footnotesize%
  \textsc{Sami Fersi, Laboratoire Alexander Grothendieck, Institut des Hautes Études Scientifiques, 35 Route de Chartres, 91440 Bures-sur-Yvette, France} \par
  \textit{E-mail address}: \texttt{fersi@ihes.fr} \par
}}
\bibliography{./references/ref}
\begin{document}
\maketitle

\begin{abstract}
This article is the second of three articles whose goal is to generalize the Cartier transform of Ogus and Vologodsky to the logarithmic setting. 
We generalize a topos-theoretic version of this transform, due to Oyama. Let $k$ be a perfect field of positive characteristic $p$ and equip $S=\op{Spec}k$ with the trivial log structure. For a log smooth morphism of logarithmic schemes $X \ra S,$
we construct crystalline-like ringed topoi $\calE'$ and $\underline{\calE}$ and subcategories of crystals of quasi-coherent modules $\calC'$ and $\underline{\calC},$ equivalent respectively, under some lifting assumption, to modules with Higgs fields and integrable connections, both satisfying certain nilpotence conditions, and a morphism of topoi $\underline{\calE} \ra \calE'.$ We then prove that the pullback functor of this morphism of topoi preserves quasi-coherent crystals and hence induces a functor $\calC' \ra \underline{\calC},$ generalizing the Cartier transform. We finally use a log flat descent theorem for morphisms, that we proved in \cite{SF1}, to prove that this functor is fully faithful.
\end{abstract}

\tableofcontents

\section{Introduction}

\begin{parag}
This article is the second in a series of three articles whose goal is to generalize the Cartier transform of Ogus and Vologodsky to the logarithmic case. In the first one \cite{SF1}, we generalized a local version of the Cartier transform, due to Shiho (\cite{Shiho}), to the log smooth setting. In this article, our goal is to generalize a topos-theoretic version of the \emph{Cartier tranform}, due to Oyama. Unlike Shiho's local construction, Oyama's crystalline-like approach has the advantage of being global and not requiring any lifting assumption on the relative Frobenius or the scheme itself. More precisely, in the classical smooth case, for a perfect field $k$ of positive characteristic, let $W(k)$ be the ring of Witt vectors of $k,$ $\frakS=\op{Spf}W(k)$ and $S=\Sp k.$ For a smooth morphism $X\ra S,$ Oyama constructed two ringed topoi $\calE'$ and $\underline{\calE}$ as well as subcategories of crystals of quasi-coherent modules $\calC'$ and $\underline{\calC}$ of these respective topoi. He proved that, if $X$ and $X'=X\times_{S,F_S}S$ lift to smooth formal schemes over $\frakS,$ then $\calC'$ and $\underline{\calC}$ are equivalent respectively to modules equipped with Higgs fields and modules equipped with integrable connections, both satisfying certain nilpotence conditions. He also constructed, without assuming any lifting conditions, a morphism of ringed topoi $\varphi:\underline{\calE} \ra \calE'.$ He proved that the inverse image functor $\varphi^{-1}$ preserves crystals and hence induces a functor $\calC' \ra \underline{\calC}$ extending the usual Cartier transform. He then proved, using a faithful flat descent argument, that it is an equivalence of categories. This is the topos-theoretic Cartier transform that we aim to generalize to log smooth schemes in this article. Namely, we construct two crystalline-like topoi, a morphism between them and subcategories of crystals $\underline{\calC}$ and $\calC'.$ We prove that the pullback functor preserves quasi-coherent crystals and so it induces a functor $\calC' \ra \underline{\calC}$ that gives the Cartier transform. We also prove that it is fully faithful. Since the relative Frobenius is not necessarily flat in the log smooth case, the faithful flat descent argument used by Oyama fails in our case and we hence encounter a problem with essential surjectivity. We solve this in the third article by using Lorenzon's indexed algebras and modules.
\end{parag}

\begin{parag}
We now describe the content of this article in more details. We fix a prime number $p.$ For a morphism $X\ra S$ of fs logarithmic schemes of characteristic $p,$ recall (\cite{SF1} 3.1) the diagram of the exact relative Frobenius :
\begin{equation}
\begin{tikzcd}
X\ar{r}{F}\ar[swap]{dr}{F_{X/S}}\ar[bend right=40]{rdd} & X'\ar{d}{G}\ar{dr}{\pi} & \\
 & X''\ar{d}\ar{r} & X\ar{d} \\
 & S \ar{r}{F_S} & S,
\end{tikzcd}
\end{equation}
where $F_S$ is the absolute Frobenius of $S,$ the square is cartesian in the category of fine logarithmic schemes, $F_{X/S}$ is the relative Frobenius, $F$ is the exact relative Frobenius and $G$ is log étale (\cite{Ogus2018} IV 3.3.9).
In section 5, we define, for an fs logarithmic scheme $T$ of characteristic $p,$ the logarithmic scheme theoretic image of the Frobenius $F_T$ \eqref{era3logimagedef}, denoted by $\underline{T}.$ We then prove, in \ref{era4prop15}, that the absolute Frobenius $F_T$ factors through $\underline{T}:$
\begin{equation}\label{IranG}
F_T:T\xrightarrow{G_T} \underline{T} \hookrightarrow T.
\end{equation}
We also prove, in \ref{propfT/S} and under the assumption that $S=\Sp k$ equipped with the trivial logarithmic structure, that the relative Frobenius $F_{T/S}:T \ra T'$ factors through $\underline{T}':$
\begin{equation}\label{Iranf}
T\xrightarrow{f_{T/S}} \underline{T}' \hookrightarrow T'.
\end{equation}
\end{parag}

\begin{parag}
Let $X \ra S$ be a log smooth morphism of fs logarithmic schemes of characteristic $p$ and denote by $\calD_{X/S}$ the algebra of differential operators (\cite{SF1} 4.1) and by $\calT_{X/S}=\mathscr{Hom}_{\Ox_X} \left ( \omega^1_{X/S}, \Ox_X \right )$ the dual of the module of logarithmic differentials $\omega^1_{X/S}.$ As shown in (\cite{Ogus94} 1.2.1), the Lie algebra of derivations $\calT_{X/S}$ can be equipped with a restricted Lie algebra structure via a $p$-operation
$$
\partial \mapsto \partial^{(p)}.
$$
Seeing derivations as differential operators of $\calD_{X/S},$ we obtain a morphism
$$
F_X^*\calT_{X/S} \ra \calD_{X/S},\ \partial \mapsto \partial^p-\partial^{(p)},
$$
where $\partial^p$ is the composition of differential operators and $\partial^{(p)}$ is the $p$-operation.
By adjunction, this induces a morphism
\begin{equation*}
\psi:\calT_{X'/S} \ra F_{*}\calD_{X/S}, 
\end{equation*}
where $F$ is the exact relative Frobenius. For any derivation $\partial,$ $\psi(\partial)$ is in the center of $F_*\calD_{X/S}$ and so we obtain a morphism of $\Ox_X$-algebras
\begin{equation}\label{intropsi}
\psi:S^{\bullet}\calT_{X'/S} \ra F_{*}\calD_{X/S}. 
\end{equation}
This morphism $\psi$ is called the \emph{$p$-curvature map}.
\end{parag}

\begin{parag}\label{intro16}
We fix a perfect field $k$ of characteristic $p$ and we denote its ring of Witt vectors by $W(k).$ We equip $\op{Spf}W(k)$ with the trivial logarithmic structure. If a gothic letter $\frakX$ denotes a logarithmic $p$-adic formal scheme over $W(k),$ then the corresponding roman letter $X=\frakX \times_{\op{Spf}W(k)} \op{Spec}k$ will denote its special fiber.
We consider a log smooth morphism $f:\frakX \ra \frakS$ of fs $p$-adic logarithmic formal schemes, flat over $\op{Spf}W(k).$ Denote by $f_1:X \ra S$ its special fiber.
We say that a morphism of logarithmic $p$-adic formal schemes over $\op{Spf}W(k)$ is log flat if it is so modulo $p^l$ for all $l\ge 1.$ We consider an fs logarithmic $p$-adic formal scheme $\frakS$ log flat (\cite{SF1} 7.18) and locally of finite type over $\op{Spf}W(k).$ Note that this implies that $\frakS$ is flat over $\op{Spf}W(k)$ \eqref{Wflat}. We also consider a log smooth morphism $f:(\frakX,Q) \ra (\frakS,P)$ of framed fs logarithmic $p$-adic formal schemes (\cite{SF1} 8). Let $\frakX \ra \frakU \ra  \frakY=\frakX \times_{\frakS,[Q]}^{\op{log}}\frakX$ be a factorization of the exact diagonal immersion (\cite{SF1} 8.10) into a closed immersion followed by an open one. Let $\calI$ be the ideal of $\frakX \ra \frakU.$
For a positive integer $n,$ we consider the formal groupoid $R_{\frakX,n}$ (\cite{SF1} 11.2) and define a similar formal groupoid $Q_{\frakX}$ (6.4). More precisely, $R_{\frakX,n}$ (resp. $Q_{\frakX}$) is the dilatation of $\calI+(p)$ with respect to $p^n$ (resp. the dilatation of the ideal locally generated by $\{p,a^p,a\in \calI \}$ with respect to $p$). We prove that $Q_{\frakX}$ has a natural structure of formal groupoid.
\begin{theorem}[\ref{thm1237}]
Let $f:(\frakX,Q) \ra (\frakS,P)$ be a log smooth morphism of framed fs logarithmic $p$-adic formal schemes such that $\frakS$ is log flat and locally of finite type over $\op{Spf}W(k).$ Denote by $X\ra S$ its special fiber. Let $\calD_{X/S}$ be the sheaf of logarithmic differential operators, $\w{\Gamma}^{\bullet}\calT_{X'/S}$ the completed PD-algebra of the tangent sheaf $\calT_{X'/S}.$ We consider the $p$-curvature map $S^{\bullet}\calT_{X'/S} \ra \calD_{X/S}$ \eqref{intropsi} and 
$$\calD^{\gamma}_{X/S}=\calD_{X/S}\otimes_{S^{\bullet} \calT_{X'/S}} \widehat{\Gamma}^{\bullet}\calT_{X'/S}.$$
The following tensor categories are canonically equivalent:
\begin{enumerate}
\item The category of $\Ox_X$-modules equipped with a stratification relative to $R_{\frakX,1}$ (resp. $Q_{\frakX}$) (\cite{DXU19} 5.4).
\item The category of locally PD-nilpotent $\widehat{\Gamma}^{\bullet}\calT_{X/S}$-modules (resp. locally PD-nilpotent $\calD^{\gamma}_{X/S}$-modules).
\end{enumerate}
\end{theorem}
\end{parag}

\begin{parag}\label{110}
In the classical smooth case, Oyama provided a crystalline-like interpretation of the Cartier transform. We follow his work and provide a similar interpretation in the log smooth case. Let $W(k)$ be the ring of Witt vectors of a perfect field $k$ of positive characteristic $p.$ We equip $\op{Spf}W(k)$ with the trivial logarithmic structure. Let $\frakS$ be a log flat and locally of finite type fs logarithmic $p$-adic formal scheme over $\op{Spf}W(k).$ Denote by $S$ its special fiber and consider a log smooth morphism $X\ra S$ of fs logarithmic schemes.
We define categories $\calE(X/\frakS)$ and $\underline{\calE}(X/\frakS)$ \eqref{defforintro} as follows: an object of $\calE(X/\frakS)$ (resp. $\underline{\calE}(X/\frakS)$) is a triple $(U,\frakT,u)$ consisting of an étale strict morphism of logarithmic schemes $U \ra X,$ a log flat fs logarithmic $p$-adic formal $\frakS$-scheme $\frakT$ and an $S$-morphism of logarithmic schemes $u:T \ra U$ (resp. $u:\underline{T}\ra U$) which is affine as a morphism of schemes, where $\underline{T}$ is the logarithmic scheme-theoretic image of the absolute Frobenius $F_T$ of $T.$
A morphism $(U_1,\frakT_1,u_1) \ra (U_2,\frakT_2,u_2)$ in $\calE(X/\frakS)$ (resp. $\underline{\calE}(X/\frakS)$) is a pair $(f,g)$ consisting of an $\frakS$-morphism $f:\frakT_1 \ra \frakT_2$ and an $X$-morphism $g:U_1 \ra U_2$ such that $u_2\circ f_1=g\circ u_1$ (resp. $u_2 \circ \underline{f_1}=g\circ u_2$), where $f_1:T_1 \ra T_2$ is the morphism induced by $f$ by reduction modulo $p$ and $\underline{f_1}:\underline{T_1} \ra \underline{T_2}$ is induced by $f_1.$
We define, in \ref{parettop} and \ref{era4logflattop}, two topologies on each one of these two categories: the étale topology and the log flat topology and we denote by $\widetilde{\calE}(X/\frakS),$ $\widetilde{\underline{\calE}}(X/\frakS),$ $\widetilde{\calE}_{lf}(X/\frakS)$ and $\widetilde{\underline{\calE}}_{lf}(X/\frakS)$ the corresponding topoi. We equip them with the rings
$$
\Ox_{\calE(X/\frakS)}:(U,\frakT,u) \mapsto \Gamma\left (T,\Ox_T \right ),\ \Ox_{\underline{\calE}(X/\frakS)}:(U,\frakT,u) \mapsto \Gamma\left (T,\Ox_T \right ).
$$
A morphism $(f,g):(U_1,\frakT_1,u_1) \ra (U_2,\frakT_2,u_2)$ of $\widetilde{\calE}(X/\frakS)$ (resp. $\widetilde{\underline{\calE}}(X/\frakS)$) induces a morphism of ringed topoi
$$
\widetilde{f}:\left (U_{1,\text{ét}},u_{1*}\Ox_{T_1} \right ) \ra \left (U_{2,\text{ét}},u_{2*}\Ox_{T_2} \right ).
$$
We prove that the data of a module $\F$ of $\widetilde{\calE}(X/\frakS)$ (resp. $\widetilde{\underline{\calE}}(X/\frakS)$) is equivalent to the data, for every object $(U,\frakT,u),$ of a module $\F_{(U,\frakT,u)}$ of $\left (U_{\text{ét}},u_*\Ox_T \right ),$ and, for every morphism $(f,g):(U_1,\frakT_1,u_1) \ra (U_2,\frakT_2,u_2),$ a morphism
$$
c_{\F,(f,g)}:\widetilde{f}^*\F_{(U_2,\frakT_2,u_2)} \ra \F_{(U_1,\frakT_1,u_1)},
$$
satisfying certain conditions \eqref{lindescentdata}. A module $\F$ is said to be a crystal if $c_{\F,(f,g)}$ is an isomorphism for every morphism $(f,g).$ The category of crystals of $\widetilde{\calE}(X/\frakS)$ (resp. $\widetilde{\underline{\calE}}(X/\frakS)$) will be denoted $\calC(X/\frakS)$ (resp. $\underline{\calC}(X/\frakS)$). Consider the following assumption:
\begin{itemize}
\item[$(H)$] The logarithmic formal scheme $\frakS$ is equipped with a frame $\frakS \ra [P]$ and the morphism $X\ra S$ lifts to a log smooth morphism $f:(\frakX,Q) \ra (\frakS,P)$ of framed fs logarithmic $p$-adic formal schemes.
\end{itemize}
Under this assumption, we prove, in \ref{equivRQ}, that we have equivalences of categories
$$
\calC(X/\frakS) \xrightarrow{\sim} \begin{Bmatrix} \Ox_{X}\text{-modules\ with\ an}\\ \calR_1\text{-stratification} \end{Bmatrix},
$$
$$
\underline{\calC}(X/\frakS) \xrightarrow{\sim} \begin{Bmatrix} \Ox_{X}\text{-modules\ with\ a}\\ \calQ_1\text{-stratification}. \end{Bmatrix},
$$
where $\calQ_1$ and $\calR_1$ are the hopf algebras corresponding to the special fibers of $Q_{\frakX}$ and $R_{\frakX,1}$ \eqref{intro16} respectively.

For an object $(U,\frakT,u)$ of $\underline{\calE}(X/\frakS),$ we prove in \ref{Shihoparag1} that there exists a unique morphism $T \ra U''$ fitting into the commutative diagram
$$
\begin{tikzcd}
T\ar{r}{G_T} \ar{dr}\ar[bend right=30]{rdd} & \underline{T} \ar{dr}{u} \ar[hook]{r} & T\ar[bend right=-30]{dd} \\
 & U''\ar{d}\ar{r} & U\ar{d} \\
 & S \ar{r}{F_S} & S.
\end{tikzcd}
$$
where $U''$ is the base change of $U\ra S$ by $F_S$ in the category of fine logarithmic schemes and $G_T$ is given in \eqref{IranG}. We denote by $U\ra U'$ the exact relative Frobenius of $U$ with respect to $S$ and we then construct, in \ref{propShiho185}, a morphism $\Delta_u:T \ra U'$ fitting into the commutative diagram
$$
\begin{tikzcd}
U \ar{rr} & & U' \ar{d} \\
\underline{T} \ar[swap]{u}{u} \ar[hook]{r} & T \ar{ur}{\Delta_{u}} \ar{r} & U'',
\end{tikzcd}
$$
We use it to define a functor
$$
\rho:\begin{array}[t]{clc}
\underline{\calE}(X/\frakS) & \ra & \calE(X'/\frakS) \\
(U,\frakT,u) & \mapsto & \left ( U',\frakT,\Delta_u \right ).
\end{array}
$$
\begin{proposition}[\ref{thmproof2}]
If $X\ra S$ is of Cartier type (\cite{Kat89} 4.8), then $\rho$ is continuous and cocontinuous for both the étale and log flat topologies and hence induces morphisms of topoi
$$
C_{X/\frakS}:\widetilde{\underline{\calE}}(X/\frakS) \ra \widetilde{\calE}(X'/\frakS),
$$
$$
C_{X/\frakS,lf}:\widetilde{\underline{\calE}}_{lf}(X/\frakS) \ra \widetilde{\calE}_{lf}(X'/\frakS).
$$
such that $C_{X/\frakS}^{-1}$ and $C_{X/\frakS,lf}^{-1}$ are composition by $\rho.$
\end{proposition}
\end{parag}

\begin{parag}
Keep the assumptions of \ref{110} and suppose furthermore that $\frakS=\op{Spf}W(k)$ equipped with the trivial logarithmic structure. For an object $(U,\frakT,u)$ of $\underline{\calE}(X/\frakS),$ we have a commutative diagram
$$
\begin{tikzcd}
U \ar[swap]{d}{F_{U/S}} & & \underline{T}\ar[swap]{ll}{u} \ar[swap]{d}{F_{\underline{T}/S}} \ar[hook]{rr} & & T\ar{d}{F_{T/S}} \ar{dll}{f_{T/S}} \\
U' & & \underline{T}' \ar{ll}{u'} \ar[hook]{rr} & & T',
\end{tikzcd}
$$
where $f_{T/S}$ is defined in \eqref{Iranf}.
The functor $\rho$ is then given by
$$
\rho:\begin{array}[t]{clc}
\underline{\calE}(X/\frakS) & \ra & \calE(X'/\frakS) \\
(U,\frakT,u) & \mapsto & \left ( U',\frakT,u'\circ f_{T/S} \right ).
\end{array}
$$
We prove in this case that $C_{X/\frakS,lf}$ is an equivalence of topoi \eqref{era3equivlfcrystals} and deduce the following theorem:
\begin{theorem}[\ref{lemdirectimage}]
If $\frakS=\op{Spf}W(k)$ equipped with the trivial logarithmic structure, $S$ is its special fiber and $X$ is a log smooth fs logarithmic $S$-scheme, then the functor $C_{X/\frakS}^*$ induces a fully faithful functor
$$\calC^{\text{qcoh}}(X'/\frakS) \ra \underline{\calC}^{\text{qcoh}}(X/\frakS).$$
\end{theorem}
\end{parag}

\textbf{Acknowledgement.} This article is the second part of my thesis prepared at Université Paris-Saclay and IHES. I express my greatest gratitude to my PhD advisor Ahmed Abbes for introducing me to this topic, his guidance and his patience. I also thank Atsushi Shiho and Daxin Xu for their helpful comments and suggestions.

\section{Notations and conventions}

\begin{parag}\label{Wflat}
In this article, $k$ denotes a perfect field of positive characteristic $p$ and $W$ denotes its ring of Witt vectors. We equip the formal scheme $\op{Spf}W$ with the trivial logarithmic structure. If a gothic letter $\frakT$ denotes a logarithmic $p$-adic formal scheme (\cite{SF1} 7), we denote by $\frakT_n,$ for any positive integer $n,$ the logarithmic scheme obtained from $\frakT$ by reduction modulo $p^n.$ The corresponding roman letter $T$ will denote $\frakT_1.$

We say that a $p$-adic formal $\op{Spf}W$-scheme $\frakX$ is \emph{flat over $\op{Spf}W$} if multiplication by $p$ on $\Ox_{\frakX}$ is injective. This is equivalent to the fact that, for every affine open formal subscheme $U$ of $\frakX,$ the algebra $\Gamma(U,\Ox_{\frakX})$ is flat over $W.$ In this sense, a $p$-adic formal $\op{Spf}W$-scheme $\frakX$ is flat over $\op{Spf}W$ if and only if $\frakX_n$ is flat over $\Sp W/(p^n)$ for all positive integers $n$ (\cite{SF1} 2.16).
We recall that, if $\frakT$ is an fs logarithmic $p$-adic formal scheme which is log flat over $\op{Spf}W,$ then $\frakT$ is flat over $\op{Spf}W$ (\cite{SF1} 7.19).

If $R$ is a ring and $P$ is a monoid, we denote by $R[P]$ the free $R$-algebra on $P$ and by $A_R[P]$ the scheme $\Sp R[P]$ equipped with the logarithmic structure associated with the canonical chart $P\ra R[P].$ If $n\ge 1$ is an integer, we denote by $A_n[P]$ (resp. $A[P]$) the logarithmic scheme $A_{\Z/p^n\Z}[P]$ (resp. $A_{\Z}[P]$).

For morphisms of fs logarithmic schemes $X\ra S$ and $Y\ra S,$ we denote by $X\times_S^{\op{log}}Y$ the fiber product of $X\ra S$ and $Y\ra S$ in the category of fs logarithmic schemes. We keep the notation $X\times_SY$ for the fiber product of $X\ra S$ and $Y\ra S$ in the category of logarithmic schemes.
\end{parag}

\section{Preliminaries}

\begin{parag}\label{PFrob}
We start by recalling the exact relative Frobenius (\cite{SF1} 3). Let $f:X\ra S$ be a morphism of fine logarithmic schemes of characteristic $p.$ We have the following diagram :
\begin{equation}\label{diag51}
\begin{tikzcd}
X\ar{r}{F}\ar[swap]{dr}{F_{X/S}}\ar[swap,bend right=40]{rdd}{f} \ar[bend right=-60]{drr}{F_X} & X'\ar{d}{G}\ar{dr}{\pi} & \\
 & X''\ar{d}\ar{r} & X\ar{d}{f} \\
 & S \ar{r}{F_S} & S,
\end{tikzcd}
\end{equation}
where $F_S$ and $F_X$ are the absolute Frobenius morphisms of $S$ and $X$ respectively, the square is cartesian in the category of fine logarithmic schemes, $G$ is log étale, $F_{X/S}$ is the relative Frobenius and $F$ is the exact relative Frobenius.
Recall that, if $f$ is log smooth, then $F$ is log flat (\cite{SF1} 3.5). If $f$ is of Cartier type (\cite{Kat89} 4.8) then $X'=X''$ and $F=F_{X/S}.$

Suppose that $f$ underlies a morphism of framed fs logarithmic schemes $(f,\theta):(X,Q) \ra (S,P)$ (\cite{SF1} 6).
Consider the Frobenius morphism $F_P:P \ra P,\ x\mapsto px$ of $P,$ the monoid $Q''=(Q\oplus_{P,F_P}P)^{int}$ and denote by $Q'$ the inverse image of $Q$ by the homomorphism
\begin{equation}\label{eqfrobmon1}
Q''^{gp} \ra Q^{gp},\ (x,y) \mapsto px+\theta^{gp}(y).
\end{equation}
Denote by $F_{Q/P}:Q' \ra Q$ the morphism induced by \eqref{eqfrobmon1} and $\pi_{Q/P}:Q \ra Q',\ x\mapsto (x,0).$
By (\cite{SF1} 6.6), we have a canonical frame $X' \ra [Q']$ and $(F,F_{Q/P}):(X,Q) \ra (X',Q')$ and $(\pi,\pi_{Q/P}):(X',Q') \ra (X,Q)$ are morphisms of framed logarithmic schemes .
\end{parag}

\begin{proposition1}\label{piunivhomeo}
The underlying morphism of schemes of $\pi:X' \ra X$ \eqref{diag51} is a universal homeomorphism.
\end{proposition1}

\begin{proof}
The Frobenius $F_S$ is weakly inseparable (\cite{Ogus2018} III 2.4) hence so is the canonical projection $X'' \ra X$ (\cite{Ogus2018} 2.4.8). The relative Frobenius $X\ra X''$ and the exact relative Frobenius $X\ra X'$ are also weakly inseparable so they are universal homeomorphisms. It follows by (\cite{SP} \href{https://stacks.math.columbia.edu/tag/0H2M}{Lemma 0H2M}) that $\pi:X' \ra X$ is a universal homeomorphism.
\end{proof}

\begin{parag}\label{parag77}
Let $f:(\frakX,Q) \ra (\frakS,P)$ be a morphism of framed locally Noetherian fs logarithmic $p$-adic formal schemes (\cite{SF1} 8). By (\cite{SF1} 8.10), the diagonal immersion $\frakX \ra \frakX \times_{\frakS}^{\op{log}}\frakX$ factors as follows :
$$
\begin{tikzcd}
 & \frakY = \frakX\times_{\frakS,[Q]}^{\op{log}}\frakX \ar{d}{g} \\
\frakX  \ar{r} \ar{ur} & \frakX \times_{\frakS}^{\op{log}} \frakX,
\end{tikzcd}
$$
where $\frakX\ra \frakY$ is a strict immersion and $g$ is log étale and affine. In addition, the projections $\frakY \ra \frakX$ are strict (\cite{SF1} 8.9). Let $\frakX \xrightarrow{\Delta} \frakZ\ra \frakY$ be a factorization of $\frakX\ra \frakY$ into a strict closed immersion $\Delta$ followed by an open one and let $p_1,p_2:\frakZ\ra \frakX$ be the canonical projections. Denote by $\calI$ the ideal of $\Delta.$ By (\cite{SF1} 8.12.1), we have a canonical isomorphism
\begin{equation}\label{isoomega1}
\calI/\calI^2 \xrightarrow{\sim} \omega^1_{\frakX/\frakS}.
\end{equation}
Suppose that $f$ is log smooth and $m_1,\hdots,m_d \in \Gamma(U,\calM_X)$ are local coordinates (\cite{SF1} 4.4) that lift to local sections $\widetilde{m}_1,\hdots,\widetilde{m}_d\in \Gamma(\frakU,\calM_{\frakX}),$ where $\frakU$ is the unique étale formal $\frakX$-scheme satisfying $U=\frakU\times_{\frakX}X.$ By (\cite{SF1} 7.15), we have an exact sequence
\begin{equation}\label{era2exactseq}
0 \ra \Delta^{-1}(1+\calI) \xrightarrow{\lambda} \Delta^{-1}\calM_{\frakZ} \xrightarrow{\Delta^{\flat}} \calM_{\frakX} \ra 0,
\end{equation}
where $\lambda$ is induced by $\alpha_{\frakZ}^{-1}:\Ox_{\frakZ}^{\times} \ra \calM_{\frakZ}.$ It follows that there exist $\widetilde{\eta}_1,\hdots,\widetilde{\eta}_d\in \Gamma \left ( \frakU,\Delta^{-1} \calI \right )$ such that $\left (\Delta^{-1}p_1^{\flat} \right )(\widetilde{m}_i)+\lambda(1+\widetilde{\eta}_i)=\left ( \Delta^{-1} p_2^{\flat}\right )(\widetilde{m}_i).$ The local sections $\widetilde{\eta}_1,\hdots,\widetilde{\eta}_d$ locally generate $\Delta^{-1}\calI$ (\cite{SF1} 8.12).
\end{parag}

\section{The $p$-curvature map}

\begin{parag}\label{Mojtaba}
Let $X\ra S$ be a log smooth morphism of fine logarithmic schemes of characteristic $p.$ Denote by $P_0$ the log PD-envelope of $X\ra X\times_SX$ (\cite{Kat89} 5.4), by $\calP_0$ its structural ring and by $\calI_0$ its PD-ideal. We consider $\calP_0$ as a module over $\Ox_X$ via the first projection $P_0\ra X.$ Let $\calD_{X/S}=\bigcup_{n\ge 1}\mathscr{Hom}_{\Ox_X}\left (\calP_0/\calI_0^{[n]},\Ox_X\right )$ be the $\Ox_X$-algebra of differential operators. If we identify the étale sites of $X$ and $P_0$ via the nilimmersion $\iota : X\ra P_0,$ then we have an exact sequence of monoids (\cite{SF1} 4.3)
$$
0 \ra 1+\calI_0 \xrightarrow{\lambda} \calM_{P_0} \xrightarrow{\iota^{\flat}} \calM_X \ra 0.
$$
If $m_1,\hdots,m_d$ are local sections of $\calM_X$ which form a local system of coordinates and $p_1,p_2:P_0 \ra X$ are the canonical projections, then, for all $1\le i \le d,$ there exists a unique local section $\eta_i$ of $\calI_0$ such that
$$
p_2^{\flat}m_i=p_1^{\flat}m_i+\lambda(1+\eta_i).
$$
Then we have basis $\left (\eta^{[I]}=\prod_{i=1}^d\eta_i^{[I_i]}\right )_{I=(I_1,\hdots,I_d)\in \N^d}$ of $\calP_0$ (\cite{Kat89} 6.5). Denote by $\left (\partial_I\right )$ its dual basis. Note that $\partial_I:\calP_0 \ra \Ox_X$ factors through $\calP_0/\calI_0^{[|I|+1]}$ and is hence naturally a differential operator of order $\le |I|.$

Recall that the $p$-curvature map is a morphism of $\Ox_{X'}$-algebras (\cite{Ohkawa} 3.10)
\begin{equation}\label{pcurv}
\psi:S^{\bullet}\calT_{X'/S} \ra F_*\calD_{X/S}
\end{equation}
where $F:X\ra X'$ is the exact relative Frobenius and $\calT_{X'/S}=\mathscr{Hom}_{\Ox_{X'}}\left (\omega^1_{X'/S},\Ox_{X'} \right )$ is the dual of the module of logarithmic differentials $\omega^1_{X'/S}.$ Recall that the image of $\psi$ lies in the center of $F_*\calD_{X/S}$ (\cite{Ohkawa} 4.15) and that this morphism is defined locally as follows : For $1\le i\le d,$ let $m_i'=\pi^{\flat}m_i.$ The local sections $\op{dlog}m_1',\hdots,\op{dlog}m_d'$ form a basis of $\omega^1_{X'/S}.$ Let $(\partial_i')_{1\le i \le d}$ be its dual basis. By (\cite{Ohkawa} 3.11) and (\cite{SF1} (4.5.8)), for all $1\le i\le d$ and $I\in \N^d,$ we have
\begin{equation}\label{eqDoma1}
\psi\left (\partial'_i\right )=\partial_{p\epsilon_i},\quad \psi\left (\partial'^I\right )=\partial_{pI},
\end{equation}
where $\partial'^I=\prod_{i=1}^d\partial_i'^{I_i}.$
\end{parag}

\begin{comment}
It follows from that and (\cite{SF1} 2.2.4) that, for all $I\in \N^d,$ we have
\begin{equation}\label{era2Koko11}
\psi \left ( \prod_{i=1}^d\partial_i'^{I_i} \right )=\partial_{pI}.
\end{equation}
\end{comment}

\begin{lemma}\label{era4tak2}
Denote by $\widehat{S}^{\bullet}\calT_{X'/S}$ the completion of $S^{\bullet}\calT_{X'/S}$ with respect to the ideal $\bigoplus_{n\ge 1}S^n\calT_{X'/S}.$
The composition
$v:S^{\bullet}\calT_{X'/S} \xrightarrow{\psi} \calD_{X/S} \ra \widehat{\calD}_{X/S}=\mathscr{Hom}_{\Ox_X}(\calP_0,\Ox_X),$
where the first arrow $\psi$ is the $p$-curvature map \eqref{pcurv} and the second arrow is induced by the canonical projections $\calP_0 \ra \calP_0/\ov{\calI}_0^{[n]}$ for $n\ge 1,$ induces a homomorphism
\begin{equation}\label{eq13191}
\widehat{\psi}:\widehat{S}^{\bullet}\calT_{X'/S} \ra \widehat{\calD}_{X/S}=\mathscr{Hom}_{\Ox_X}(\calP_0,\Ox_X).
\end{equation}
\end{lemma}

\begin{proof}
Suppose that we have local coordinates $m_1,\hdots,m_d \in \Gamma(X,\calM_X)$ and $\left (\eta^{[I]}:=\prod_{i=1}^d\eta_i^{[I_i]} \right )_{I\in \N^d}$
the corresponding basis for the $\Ox_X$-module $\calP_0,$ as in \ref{Mojtaba}.
First, note that
$$
\w{S}^{\bullet}\calT_{X'/S} = \lim\limits_{\substack{\longleftarrow \\ n\ge 1}} S^{\bullet}\calT_{X'/S} / \bigoplus_{k\ge n}S^k\calT_{X'/S} 
= \lim\limits_{\substack{\longleftarrow \\ n\ge 0}} \bigoplus_{k\le n}S^k\calT_{X'/S} 
= \prod_{k\in \N}S^k\calT_{X'/S}.
$$
It is sufficient to define $\w{\psi}:\calP_0\ra \Ox_X$ by
$\w{\psi}(t)=\sum_{k=0}^{\infty}v(t_k).$
for every
$t=(t_0,t_1,\hdots )\in \prod_{k\in \N}S^k\calT_{X'/S}.$
Note that this infinite sum is well-defined since, for every $I\in \N^d,$ $v(t_k)\left (\eta^{[I]} \right )=0$ for sufficiently large $k.$
\end{proof}

\section{Logarithmic scheme-theoretic image}

\begin{lemmadef}\label{Mfs}
For a monoid $M,$ let $\sim$ be the equivalence relation defined, for any $x,y\in M,$ by $x\sim y$ if $px=py.$ We equip the quotient $\underline{M}:=M_{/\sim}$ with the natural monoid structure inherited from $M$ and call it the monoidal Frobenius image of $M.$
\begin{enumerate}
\item For any monoid $M,$ there exists a canonical isomorphism
$\underline{M^{gp}} \xrightarrow{\sim} \underline{M}^{gp}.$
\item If $M$ is an integral (resp. fine, resp. saturated, resp. fs) monoid then so is $\underline{M}.$
\item If $M$ is a saturated monoid then the canonical morphism $M \ra \underline{M}$ is exact and strict.
\end{enumerate}
\end{lemmadef}

\begin{proof}
The composition
$M \ra \underline{M} \ra \underline{M}^{gp}$
factors through $\underline{M^{gp}}$ yielding a morphism
$$u:\underline{M^{gp}} \ra \underline{M}^{gp}.$$
The composition
$M \ra M^{gp} \ra \underline{M^{gp}}$
factors through $\underline{M}^{gp}$ yielding a morphism
$$v:\underline{M}^{gp} \ra \underline{M^{gp}}.$$
The morphisms $u$ and $v$ are inverse to each other.

It is clear that, if $M$ is finitely generated, then so is $\underline{M}.$ Suppose that $M$ is integral and let $x,y,t\in M$ such that
$\ov{x}+\ov{t}=\ov{y}+\ov{t}\in \underline{M}.$
Then
$px+pt=py+pt \in M$
and by integrality of $M,$ $px=py$ so $\ov{x}=\ov{y} \in \underline{M}.$
Now suppose that $M$ is saturated and let $x\in M^{gp}$ and $\ov{x}$ its image in $\underline{M}^{gp}.$ If there exists a positive integer $n$ such that $n\ov{x}\in \underline{M}$ then there exists $y\in M$ such that $pnx=py\in M^{gp}.$ Since $M$ is saturated and $y\in M,$ $x\in M$ and so $\ov{x} \in \underline{M}.$ In addition, if the image of an element $x\in M^{gp}$ in $\underline{M}^{gp}$ belongs to $\underline{M},$ then $px\in M$ and so $x\in M.$ This proves that the square
$$
\begin{tikzcd}
M \ar{r} \ar{d} & \underline{M} \ar{d} \\
M^{gp} \ar{r} & \underline{M}^{gp}
\end{tikzcd}
$$ 
is cartesian, hence the exactness of $M \ra \underline{M}.$

Finally, suppose that $M$ is saturated and let us prove that the canonical morphism $M\ra \underline{M}$ is strict. It is surjective so it is sufficient to prove that, if $m,m'\in M$ are such that there exists a $u\in M$ satisfying $pm=pm'+pu$ and $\ov{u}$ is invertible in $\underline{M},$ then there exists $v\in M^{\times}$ such that $m=m'+v.$ Since $\ov{u}$ is invertible in $\underline{M},$ there exists $u'\in M$ such that $pu+pu'=0.$ So $u\in M^{\times}.$ In $M^{gp},$ $p(m-m')=pu\in M$ and $M$ is saturated so $m-m'\in M.$ Since $u$ is invertible and $p(m-m')=pu,$ we have $m-m'\in M^{\times}.$ We conclude by taking $v=m-m'.$
\end{proof}

\begin{lemma}\label{Mfs2}
Let $u:M\ra N$ be a morphism of saturated monoids, $\underline{M}$ and $\underline{N}$ the monoidal Frobenius images of $M$ and $N$ respectively \eqref{Mfs} and $\underline{u}:\underline{M}\ra \underline{N}$ the morphism induced by $u.$ If $u$ is strict, then so is $\underline{u}.$
\end{lemma}

\begin{proof}
We have to prove that the morphism
$\underline{M}/\underline{M}^{\times} \ra \underline{N}/\underline{N}^{\times},$
induced by $\underline{u},$ is an isomorphism.
First, note that an element $x\in M$ is invertible if and only if its image $\ov{x}$ in $\underline{M}$ is invertible. This, along with the surjectivity of the projection $N\ra \underline{N},$ yields the surjectivity of $\underline{u}$ modulo the invertibles. Let $x,y\in M$ and $t\in N$ such that $\ov{t}$ is invertible and
$\underline{u}(\ov{x})=\underline{u}(\ov{y})+\ov{t}.$
Then
$\ov{u(x)}=\ov{u(y)}+\ov{t}.$
So
$u(px)=u(py)+pt.$
Since $u$ is strict and $t$ is invertible, there exists $z\in M^{\times}$ such that $px=py+z.$ Set $w=x-y\in M^{gp}.$ Since $pw=z\in M^{\times}$ and $M$ is saturated, $w\in M^{\times}.$ We have $px=p(y+w)$ so $\ov{x}=\ov{y}+\ov{w}$ and $\underline{u}$ is injective modulo the invertibles.
\end{proof}

\begin{parag}\label{era3logimagedef}
Let $X$ be a logarithmic scheme of characteristic $p$ and $\underline{X}$ the scheme theoretic image of the absolute Frobenius $F_X$ i.e. the closed subscheme of $X$ defined by the ideal consisting of local sections of $\Ox_X$ whose $p^{th}$ power vanishes. The ideal of the closed immersion $\underline{X} \ra X$ is, by definition, a nilideal and $\underline{X} \ra X$ is hence a universal homeomorphism (\href{https://stacks.math.columbia.edu/tag/054M}{Lemma 054M}). We identify the small étale sites of $\underline{X}$ and $X$ via this immersion (\href{https://stacks.math.columbia.edu/tag/04DZ}{Theorem 04DZ}). For any étale $X$-scheme $U,$ let $\underline{\calM_X(U)}$ be the monoidal Frobenius image of the monoid $\calM_X(U)$ (\ref{Mfs}).
We denote by $\calM_{\underline{X}}$ the sheaf of monoids associated to the presheaf
$U \mapsto \underline{\calM_X(U)}$
and call it \emph{the monoidal Frobenius image of $\calM_X$}.
The canonical morphism $\calM_X \ra \calM_{\underline{X}}$ is an epimorphism of sheaves of sets.
For any local sections $m$ and $m'$ of $\calM_X,$ we have
$$\left (\alpha_X(m)-\alpha_X(m') \right )^p=\alpha_X(pm)-\alpha_X(pm').$$
This proves that the composition
$\calM_X\xrightarrow{\alpha_X} \Ox_X \ra \Ox_{\underline{X}},$
where the second arrow is the canonical projection, induces a morphism of monoids
$\alpha_{\underline{X}} :\calM_{\underline{X}} \ra \Ox_{\underline{X}}.$
We claim that the morphism of monoids
\begin{equation}\label{era4eq1}
\alpha_{\underline{X}}^{-1}\left (\Ox_{\underline{X}}^* \right ) \ra \Ox_{\underline{X}}^*,
\end{equation}
induced by $\alpha_{\underline{X}},$ is an isomorphism and hence that $(\calM_{\underline{X}},\alpha_{\underline{X}})$ is a logarithmic structure on $\underline{X}.$ Indeed, let $m$ and $m'$ be local sections of $\calM_X$ and $\ov{m}$ and $\ov{m'}$ their images in $\calM_{\underline{X}}.$ If $\alpha_{\underline{X}}(\ov{m})=\alpha_{\underline{X}}(\ov{m'})$ then $(\alpha_X(m)-\alpha_X(m'))^p=0$ and hence $\alpha_X(pm)=\alpha_X(pm').$ If, in addition, $\alpha_{\underline{X}}(\ov{m}),\alpha_{\underline{X}}(\ov{m'})\in \Ox_{\underline{X}}^*$ and since a local section of $\Ox_X$ is invertible if and only if its image in $\Ox_{\underline{X}}$ is invertible, then $\alpha_X(m),\alpha_X(m') \in \Ox_X^*$ and so $pm=pm'.$ This proves the injectivity of (\ref{era4eq1}). The surjectivity follows from the fact that a local section of $\Ox_X$ is invertible if and only if its image in $\Ox_{\underline{X}}$ is invertible.

We call the logarithmic scheme $(\underline{X},\calM_{\underline{X}}),$ the \emph{logarithmic scheme theoretic image of $F_X$}.
\end{parag}

\begin{remark}\label{remstalk}
For a logarithmic scheme $X$ of characteristic $p$ and a geometric point $\ov{x} \ra X,$ the stalk $\calM_{\underline{X},\ov{x}}$ is by construction canonically isomorphic to the monoidal Frobenius image $\underline{\calM_{X,\ov{x}}}.$
\end{remark}

\begin{proposition}\label{era4prop15}
Let $X$ be a logarithmic scheme of characteristic $p.$
\begin{enumerate}
\item The absolute Frobenius morphism $F_X:X\ra X$ factors uniquely through the logarithmic scheme $\underline{X}$ defined in \ref{era3logimagedef}. If we denote by $G_X:X \ra \underline{X}$ the resulting morphism, then the composition
$
\underline{X} \ra X \xrightarrow{G_X} \underline{X},
$
of the canonical immersion $\underline{X} \ra X$ with $G_X,$ is equal to the absolute Frobenius $F_{\underline{X}}$ of $\underline{X}.$
\item If $X$ is fs then the canonical immersion $\underline{X} \ra X$ is strict and inseparable (\cite{Ogus2018} III 2.4.2).
\end{enumerate}
\end{proposition}

\begin{proof}
The first assertion follows from the fact that $F_X^{\#}:\Ox_X \ra \Ox_X,\ x\mapsto x^p$ factors through $\Ox_{\underline{X}}$ and $F_X^{\flat}:\calM_X \ra \calM_X,\ m\mapsto pm$ factors through $\calM_{\underline{X}}$ and the commutativity of the diagram
$$
\begin{tikzcd}
\calM_X \ar{r}\ar[bend right=-30]{rrr}{m \mapsto pm} \ar[swap]{d}{\alpha_X} & \calM_{\underline{X}} \ar{rr}{\ov{m}\mapsto pm} \ar[swap]{d}{\alpha_{\underline{X}}} & & \calM_X \ar{d}{\alpha_X} \\
\Ox_X \ar{r} \ar[swap, bend right=30]{rrr}{x \mapsto x^p} & \Ox_{\underline{X}} \ar{rr}{\ov{x}\mapsto x^p} & & \Ox_X 
\end{tikzcd}
$$
The fact that the composition
$
\underline{X} \ra X \xrightarrow{G_X} \underline{X}
$
is equal to $F_{\underline{X}}$ is clear. 
Now suppose that $X$ is fs. Recall the remark \ref{remstalk}. The canonical morphism $\ov{\calM}_{X,\ov{x}} \ra \ov{\calM}_{\underline{X},\ov{x}}$ is an isomorphism for every geometric point $\ov{x}$ of $X,$ by \ref{Mfs}. It follows that the canonical immersion $\underline{X} \ra X$ is strict and hence exact. Since the canonical immersion $\underline{X} \ra X$ is clearly weakly inseparable (\cite{Ogus2018} III 2.4.2), this finishes the proof. 
\end{proof}

\begin{proposition}\label{era3functoriality}~ %%%
\begin{enumerate}
\item The correspondance $X\mapsto \underline{X}$ is functorial on logarithmic schemes of characteristic $p.$
\item Let $X\ra S$ and $Y\ra S$ be morphisms of logarithmic schemes of characteristic $p.$ There exists a canonical morphism
$\underline{X\times_SY}\ra \underline{X}\times_{\underline{S}}\underline{Y}.$
In addition, this morphism is a closed nilimmersion i.e. a closed immersion whose defining ideal is a nilideal.
\item Let $X\ra S$ and $Y\ra S$ be morphisms of fs logarithmic schemes of characteristic $p.$ There exists a canonical morphism
$\underline{X\times_S^{\op{log}}Y}\ra \underline{X}\times_{\underline{S}}^{\op{log}}\underline{Y}.$
In addition, this morphism is a closed nilimmersion.
\item Let $f:X\ra Y$ be a morphism of fs logarithmic schemes of characteristic $p$ and $\underline{f}:\underline{X}\ra \underline{Y}$ the morphism induced by $f.$ If $f$ is strict, then so is $\underline{f}.$
\end{enumerate}
\end{proposition}

\begin{proof}
Let $f:X\ra Y$ be a morphism of logarithmic schemes of characteristic $p.$ By \ref{era3logimagedef}, $\underline{X}$ (resp. $\underline{Y}$) is the closed subscheme of $X$ (resp. $Y$) defined by the ideal $\calI_X=\{x\in \Ox_X, x^p=0\}$ (resp. $\calI_Y=\{ y\in \Ox_Y, y^p=0\}$). Let $i_X:\underline{X}\ra X$ and $i_Y:\underline{Y}\ra Y$ be the canonical closed immersions. For any $y\in \calI_Y,$ $f^{\#}(y)^p=0,$ so there exists a morphism of schemes $\underline{f}:\underline{X}\ra \underline{Y}$ making the following diagram commutative in the category of schemes
$$
\begin{tikzcd}
\underline{X}\ar{r}{\underline{f}} \ar{d}{i_X} & \underline{Y}\ar{d}{i_Y} \\
X \ar{r}{f} & Y 
\end{tikzcd}
$$
Note that, as maps of topological spaces, $f=\underline{f}.$
By definition of $\calM_{\underline{X}}$ and $\calM_{\underline{Y}},$ the morphism $f^{\flat}:\calM_Y \ra f_*\calM_X$ induces a morphism
$\underline{f}^{\flat}:\calM_{\underline{Y}} \ra \underline{f}_*\calM_{\underline{X}}.$
This turns $\underline{f}$ into a morphism of logarithmic schemes. The correspondances $X\mapsto \underline{X}$ and $f\mapsto \underline{f}$ clearly define a functor. By functoriality, we get the canonical morphisms $\underline{X\times_SY}\ra \underline{X}\times_{\underline{S}}\underline{Y}$ and $\underline{X\times_S^{\op{log}}Y}\ra \underline{X}\times_{\underline{S}}^{\op{log}}\underline{Y}.$ The fact that they are closed nilimmersions follows from the commutativity of the diagrams
$$
\begin{tikzcd}
\underline{X\times_SY} \ar{r} \ar{d} & \underline{X}\times_{\underline{S}}\underline{Y} \ar{dl} & \underline{X\times_S^{\op{log}}Y} \ar{r} \ar{d} & \underline{X}\times_{\underline{S}}^{\op{log}}\underline{Y} \ar{dl} \\
X\times_SY & & X\times^{\op{log}}_SY, & 
\end{tikzcd}
$$
the fact that $\underline{X\times_SY} \ra X\times_SY$ and $\underline{X\times_S^{\op{log}}Y}\ra X\times_S^{\op{log}}Y$ are closed nilimmersions and $\underline{X}\times_{\underline{S}}\underline{Y} \ra X\times_SY,$ $\underline{X}\times_{\underline{S}}^{\op{log}}\underline{Y} \ra X\times^{\op{log}}_SY$ are affine morphisms, and (\cite{EGA1} 4.3.6). The assertion on strictness follows from \ref{Mfs2}.
\end{proof}

\begin{remark}\label{rem178}
Let $k$ be a perfect field of characteristic $p,$ $S=\Sp k$ equipped with the trivial logarithmic structure and $X$ a fine logarithmic scheme over $S.$ We have $\underline{S}=S.$ The canonical morphism
$$\underline{X}'=\underline{X}\times_{S,F_S}S \ra X\times_{S,F_S}S=X'$$
factors through $\underline{X'}.$ The resulting morphism $\underline{X}' \ra \underline{X'}$ is inverse to the canonical morphism $\underline{X'}=\underline{X\times_{S,F_S}S} \ra \underline{X}\times_{\underline{S},F_{\underline{S}}}\underline{S}= \underline{X}'$ (\ref{era3functoriality}).
From now on, we will identify $\underline{X}'$ and $\underline{X'}$ via these isomorphisms.
\end{remark}

\begin{proposition}\label{propfT/S}
Let $k$ be a perfect field of characteristic $p,$ $S=\Sp k$ equipped with the trivial logarithmic structure and $g:X\ra S$ a morphism of fs logarithmic schemes. Then the exact relative Frobenius $F:X\ra X'$ factors uniquely through the canonical closed immersion $i_X':\underline{X}' \ra X'$ \eqref{rem178}. We denote by $f_{X/S}:X\ra \underline{X}'$ the unique morphism satisfying $F=i_{X}'\circ f_{X/S}.$ This morphism $f_{X/S}$ is inseparable (\cite{Ogus2018} III 2.4.2) and an epimorphism in the category of logarithmic schemes.
\end{proposition}

\begin{proof}
Let $\ov{x}$ be a geometric point of $X$ and $\ov{s}=g(\ov{x}).$ Note that since $S$ is equipped with the trivial logarithmic structure, $\calM_{S,\ov{s}}=\Ox_{S,\ov{s}}^*.$ The morphism
$
g_{\ov{x}}^{\flat}:\calM_{S,\ov{s}} \ra \calM_{X,\ov{x}}
$
induces, modulo the invertibles, the morphism
$
\ov{g}^{\flat}_{\ov{x}}: \{1\} \ra \ov{\calM}_{X,\ov{x}}.
$
Since the direct sum of saturated morphisms is saturated, the morphism $\ov{g}^{\flat}_{\ov{x}}$ is saturated. It follows that $g$ is saturated (\cite{Ogus2018} III 2.5.1) and then that $F$ is equal to the relative Frobenius morphism $F_{X/S}$ (\cite{Ogus2018} III 2.5.4).
In addition, the logarithmic structure on $S$ is trivial so $F_S$ is strict and hence $X'=X\times_{S,F_S}S.$ Let $\underline{X}'=\underline{X}\times_{S,F_S}S.$
By \ref{era4prop15}, the absolute Frobenius $F_X$ factors uniquely through the canonical closed immersion $i_X:\underline{X} \ra X.$ Hence the existence of a morphism of logarithmic schemes $f_{X/S}:X\ra \underline{X}'$ fitting into the following commutative diagram
$$
\begin{tikzcd}
X \ar[bend right=-90]{rrdd}{F_X} \ar[swap, bend right=-30]{rrd} \ar{dr}{f_{X/S}} \ar[bend right =30]{ddr}{F} \ar[bend right=30]{dddr} &  & \\
 & \underline{X}' \ar{r} \ar{d}{i_X'} & \underline{X} \ar{d}{i_X} \\
 & X' \ar{r} \ar{d} & X\ar{d}{g} \\
 & S\ar{r}{F_S} & S
\end{tikzcd}
$$
The uniqueness follows from the fact that $i_X':\underline{X}' \ra X'$ is a closed immersion and closed immersions of logarithmic schemes are monomorphisms. The (exact) relative Frobenius morphism $F_{X/S}$ and $i_{X}'$ are inseparable (\ref{era4prop15}) so, by (\cite{Ogus2018} III 2.4.8 (3)), $f_{X/S}$ is inseparable.
We now prove that $f_{X/S}$ is an epimorphism. If $X$ is a affine, say $X=\Sp A$ for a $k$-algebra $A,$ then $\underline{X}=\Sp (A/I),$ where $I$ is the ideal of $A$ consisting of elements of $A$ whose $p^{th}$ power vanishes. Then $\underline{X}'=\Sp \left ( A/I)\otimes_{k,F_k}k\right ),$ where $F_k$ is the absolute Frobenius morphism of $k,$ and $f_{X/S}$ corresponds to
$$
h:\begin{array}[t]{clc}
(A/I)\otimes_{k,F_k}k & \ra & A \\
\ov{x} \otimes a & \mapsto & ax^p,
\end{array}
$$
for $x\in A$ and $a\in k.$ Since $k$ is perfect, $h$ is injective and hence the underlying morphism of schemes of $f_{X/S}$ is an epimorphism in the category of schemes. It remains to prove that
$f_{X/S}^{\flat}:f_{X/S}^{-1}\calM_{\underline{X}'} \ra \calM_X$
is injective. We identify the small étale sites of $X,$ $X'$ and $\underline{X}'$ via the universal homeomorphisms $X' \ra X$ and $\underline{X}'\ra X'.$ Let $\ov{x}$ be a geometric point of $X$ and $\ov{s}=g(\ov{x}).$
Then $\calM_{\underline{X}',\ov{x}}=\calM_{\underline{X},\ov{x}}\oplus_{\calM_{S,\ov{s}},F_{S,\ov{s}}^{\flat}}\calM_{S,\ov{s}}$ and $f_{X/S,\ov{x}}$ is given by
$$
f_{X/S,\ov{x}}:\begin{array}[t]{clc}
\calM_{\underline{X},\ov{x}}\oplus_{\calM_{S,\ov{s}},F_{S,\ov{s}}^{\flat}}\calM_{S,\ov{s}} & \ra & \calM_{X,\ov{x}} \\
(\ov{m},t) & \mapsto & pm+g_{\ov{x}}^{\flat}(t),
\end{array}
$$
for $m\in \calM_{X,\ov{x}}$ and $t\in\calM_{S,\ov{s}}.$ Let $m,m' \in \calM_{X,\ov{x}}$ and $t,t\in \calM_{S,\ov{s}}$ such that $f_{X/S,\ov{x}}^{\flat}(\ov{m},t)=f_{X/S,\ov{x}}^{\flat}(\ov{m'},t').$ Then
$pm+g_{\ov{x}}^{\flat}(t)=pm'+g_{\ov{x}}^{\flat}(t').$
By definition of $\calM_{\underline{X}}$ and amalgamated sums in the category of monoids, to prove that $(\ov{m},t)=(\ov{m'},t')$ in $\calM_{\underline{X},\ov{x}}\oplus_{\calM_{S,\ov{s}},F_{S,\ov{s}}^{\flat}}\calM_{S,\ov{s}},$ it is sufficient to prove that there exist $u,u'\in \calM_{S,\ov{s}}$ such that
\begin{equation}\label{era4eq4}
\begin{cases}
pm+g_{\ov{x}}^{\flat}(pu)=pm'+g_{\ov{x}}^{\flat}(pu') \\
t+pu'=t'+pu.
\end{cases}
\end{equation}
Since $S=\Sp k$ is equipped with the trivial logarithmic structure, $\calM_S=\Ox_S^*.$ Since $k$ is perfect, there exist $u,u'\in \calM_{S,\ov{s}}$ such that $t=pu$ and $t'=pu'.$ The condition (\ref{era4eq4}) is then automatically satisfied.
\end{proof}

\begin{proposition}\label{Xreduced}
Let $k$ be a perfect field of characteristic $p,$ $S=\Sp k$ equipped with the trivial logarithmic structure and $f:X\ra S$ a smooth morphism of fs logarithmic schemes. Then $X$ is reduced, hence $\underline{X}=X.$
\end{proposition}

\begin{proof}
Since $S$ is equipped with the trivial logarithmic structure, $f$ is saturated (see the proof of \ref{propfT/S} for more details). The result then follows from (\cite{Tsuji} II.4.2).
\end{proof}

\begin{parag}\label{Shihoparag1}
Let $S$ be an fs logarithmic scheme of characteristic $p,$ $U$ and $T$ two fs logarithmic schemes over $S$ and $u:\underline{T} \ra U$ an $S$-morphism. Recall the diagram defining the exact relative Frobenius
\begin{equation}\label{diag51U}
\begin{tikzcd}
U\ar{r}\ar{dr}\ar[bend right=30]{rdd} & U'\ar{d}\ar{dr} & \\
 & U''\ar{d}\ar{r} & U\ar{d} \\
 & S \ar{r}{F_S} & S,
\end{tikzcd}
\end{equation}
where the square is cartesian in the category of fine logarithmic schemes and $U' \ra U''$ is log étale. By \ref{era4prop15}, the absolute Frobenius $F_T:T \ra T$ is equal to the composition $T \xrightarrow{G_T} \underline{T} \xrightarrow{i} T,$
where $i:\underline{T} \ra T$ is the canonical immersion, which is strict by \ref{era4prop15}. We then have a unique morphism $T \ra U''$ fitting into the commutative diagram
\begin{equation}\label{morphismShiho2}
\begin{tikzcd}
T\ar{r}{G_T} \ar{dr}\ar[bend right=30]{rdd} & \underline{T} \ar{dr}{u} \ar{r}{i} & T\ar[bend right=-30]{dd} \\
 & U''\ar{d}\ar{r} & U\ar{d} \\
 & S \ar{r}{F_S} & S.
\end{tikzcd}
\end{equation}
\end{parag}

\begin{proposition}\label{propShiho185}
Keep the notation of \ref{Shihoparag1}. The diagram
\begin{equation}\label{morphismShiho1}
\begin{tikzcd}
U \ar{rr} & & U' \ar{d} \\
\underline{T} \ar[swap]{u}{u} \ar{r}{i} & T \ar[dashed]{ur}{\Delta_{u}} \ar{r} & U'',
\end{tikzcd}
\end{equation}
where $T \ra U''$ is defined in \eqref{morphismShiho2}, is commutative and there exists a unique morphism $\Delta_u:T\ra U'$ fitting into it.
In addition, we have the following properties:
\begin{enumerate}
\item If $i_X:\underline{X} \ra X$ is the canonical immersion then $\Delta_{i_X}:X\ra X'$ is the exact relative Frobenius.
\item If $u$ is affine, then so is $\Delta_u.$
\item If $S=\op{Spec}k$ equipped with the trivial logarithmic structure and $f_{T/S}$ is the morphism defined in \ref{propfT/S}, then $U'=U''$ and $\Delta_u=u' \circ f_{T/S}.$
\end{enumerate}
\end{proposition}

\begin{proof}
The composition
$
\underline{T} \xrightarrow{i} T \xrightarrow{G_T} \underline{T}
$
is equal to $F_{\underline{T}}.$
On one hand, by \eqref{diag51U}, the composition
$
\underline{T} \xrightarrow{u} U \ra U' \ra U'' \ra U,
$
where $U'' \ra U$ is the canonical projection, is equal to
$
\underline{T} \xrightarrow{u} U\xrightarrow{F_U} U.
$
This is equal to
$
\underline{T} \xrightarrow{F_{\underline{T}}} \underline{T} \xrightarrow{u} U.
$
This is then equal to
$
\underline{T} \xrightarrow{i} T \xrightarrow{G_T} \underline{T} \xrightarrow{u} U.
$
By \eqref{morphismShiho2}, this is equal to
$
\underline{T} \xrightarrow{i} T \ra U'' \ra U.
$
On the other hand, by \eqref{diag51U}, the composition
$
\underline{T} \xrightarrow{u} U \ra U' \ra U'' \ra S,
$
where $U'' \ra S$ is the canonical projection, is equal to
$
\underline{T} \xrightarrow{u} U\ra S.
$
Since $u$ is an $S$-morphism, this is equal to
$
\underline{T} \ra S.
$
This is, by definition, equal to
$
\underline{T} \xrightarrow{i} T \ra S.
$
By definition of $T \ra U'',$ this is equal to
$
\underline{T} \xrightarrow{i} T \ra U'' \ra S.
$
Since $U''$ is, by definition, the fiber product of $U\ra S$ and $F_S$ in the category of fine logarithmic schemes, we get the commutativity of the outer diagram of
\begin{equation}
\begin{tikzcd}
U \ar{rr} & & U' \ar{d} \\
\underline{T} \ar[swap]{u}{u} \ar{r}{i} & T \ar[dashed]{ur}{\Delta_{u}} \ar{r} & U''.
\end{tikzcd}
\end{equation}
Then, the existence and uniqueness of $\Delta_{u}$ follows from the fact that $U' \ra U''$ is log étale and $\underline{T} \ra T$ is a strict nilpotent closed immersion (i.e. its ideal is nilpotent). The fact that $\Delta_{u}$ is affine if so is $u,$ is clear. It is also clear that $\Delta_{i_X}$ is the exact relative Frobenius of $X$ with respect to $S.$

Suppose now that $S=\op{Spec}k$ equipped with the trivial logarithmic structure. It follows, since $U$ is fs, that the morphism $U\ra S$ is saturated and so $U'=U''$ (\cite{Ogus2018} III 2.5.4). We have the diagram
$$
\begin{tikzcd}
T \ar{dr}{G_T} \ar[swap,bend right=70]{dd}{F_{T/S}} \ar[bend right=-70]{ddr}{F_T} \ar[swap]{d}{f_{T/S}} & \\
\underline{T}' \ar{d}{i'} \ar{r} & \underline{T} \ar{d}{i} \\
T' \ar{r} & T,
\end{tikzcd}
$$
where $T' \ra T$ and $\underline{T}' \ra \underline{T}$ are the canonical projections.
The square and the outer diagram are commutative. Since $i$ is a strict immersion hence a monomorphism, it follows that the diagram
$$
\begin{tikzcd}
T \ar{dr}{G_T} \ar[swap]{d}{f_{T/S}} & \\
\underline{T}' \ar{r} & \underline{T}
\end{tikzcd}
$$
is commutative. It follows that the diagram
$$
\begin{tikzcd}
T \ar{dr}{f_{T/S}} \ar[bend right=-30]{drr}{G_T} \ar[bend right=30]{dddrr} & & & \\
 & \underline{T}' \ar{r} \ar{dr}{u'} & \underline{T} \ar{dr}{u} & \\
 & & U' \ar{r} \ar{d} & U \ar{d} \\
 & & S \ar{r}{F_S} & S
\end{tikzcd}
$$
is commutative.
The morphism $\Delta_{u}$ is then equal to the composition
$
T \xrightarrow{f_{T/S}} \underline{T}' \xrightarrow{u'} U'.
$
\end{proof}

\begin{proposition}
Let $S$ be an fs logarithmic scheme of characteristic $p$ and $T_1\ra T_2$ a morphism of fs logarithmic schemes over $S$ and $u_1:\underline{T_1} \ra U_1$ and $u_2:\underline{T_2} \ra U_2$ two $S$-morphisms fitting into a commutative diagram
$$
\begin{tikzcd}
\underline{T_1} \ar{r} \ar[swap]{d}{u_1} & \underline{T_2} \ar{d}{u_2} \\
U_1 \ar{r} & U_2.
\end{tikzcd}
$$
The diagram
$$
\begin{tikzcd}
T_1 \ar{r} \ar[swap]{d}{\Delta_{u_1}} & T_2\ar{d}{\Delta_{u_2}} \\
U_1' \ar{r} & U_2'
\end{tikzcd}
$$
is then commutative.
\end{proposition}

\begin{proof}
Consider the notation of \ref{Shihoparag1}. We have the commutative diagram
\begin{equation}\label{diagShiho8}
\begin{tikzcd}
T_1 \ar{rr}{G_{T_1}} \ar{dd} \ar{dr} & & \underline{T_1} \ar[dashed]{dd} \ar{dr}{u_1} & \\
 & U_1'' \ar{rr} \ar{dd} & & U_1 \ar{dd} \\
 T_2 \ar[bend right=30]{ddr} \ar[dashed,pos=0.20]{rr}{G_{T_2}} \ar{dr} & & \underline{T_2} \ar[dashed]{dr}{u_2} & \\
& U_2'' \ar{rr} \ar{d} & & U_2 \ar{d} \\
 & S \ar{rr}{F_S} & & S.
\end{tikzcd}
\end{equation}
We also have the commutative diagrams
\begin{equation}\label{diagShiho6}
\begin{tikzcd}
 & U_2 \ar[bend right =-30]{rrd} &  & \\
U_1 \ar{rr} \ar{ur} & & U_1' \ar{d} \ar{r} & U_2' \ar{d} \\
\underline{T_1} \ar{r} \ar{u}{u_1} & T_1 \ar{ur}{\Delta_{u_1}} \ar{r} & U_1'' \ar{r} & U_2''
\end{tikzcd}
\end{equation}
and
\begin{equation}\label{diagShiho7}
\begin{tikzcd}
U_1 \ar{r}  & U_2 \ar{rr} &  & U_2' \ar{d} \\
\underline{T_1} \ar{r} \ar{u}{u_1} & \underline{T_2} \ar{u}{u_2} \ar{r} & T_2 \ar{ur}{\Delta_{u_2}} \ar{r}  & U_2''.
\end{tikzcd}
\end{equation}
The lower composition
$$
\underline{T_1} \ra \underline{T_2} \ra T_2 \ra U_2''
$$
of \eqref{diagShiho7}, is equal to
$$
\underline{T_1} \ra T_1 \ra T_2 \ra U_2''.
$$
This is equal, by \eqref{diagShiho8}, to
$$
\underline{T_1} \ra T_1 \ra U_1'' \ra U_2'',
$$
which is the lower composition of \eqref{diagShiho6}.
We conclude by the log étaleness of $U_2' \ra U_2''.$
\end{proof}

\section{The formal schemes $Q_{\frakX}$ and $R_{\frakX,n}$}

\begin{proposition}\label{erafdil2}
Let $\frakS=\op{Spf}W$ \eqref{Wflat} equipped with the trivial logarithmic structure and $\frakY$ a logarithmic $p$-adic formal scheme flat and locally of finite type over $\frakS$ and $i:T \ra \frakY_1$ a strict immersion. There exists a strict morphism of logarithmic $p$-adic formal schemes flat over $\frakS,$ $g:\frakY_{(T/p)}^{\#} \ra \frakY,$ unique up to a canonical isomorphism, satisfying the following conditions:
\begin{enumerate}
\item The morphism $\underline{\left (\frakY_{(T/p)}^{\#}\right )_1}\ra \left (\frakY_{(T/p)}^{\#}\right )_1 \ra \frakY_1$ factors uniquely through $T \ra \frakY_1$ and the resulting morphism $\underline{\left (\frakY_{(T/p)}^{\#}\right )_1} \ra T$ is affine.
\item If $\frakZ$ is a logarithmic $p$-adic formal scheme flat over $\frakS$ and $f:\frakZ \ra \frakY$ is an $\frakS$-morphism such that $\underline{\frakZ_1} \ra \frakY_1$ factors through $T\ra \frakY_1$ then there exists a unique $\frakS$-morphism $f':\frakZ \ra \frakY_{(T/p)}^{\#}$ such that $f=g\circ f'.$ In addition, if $T \ra \frakY$ and $f$ are closed immersions, then so is $f'.$
\end{enumerate}
\end{proposition}

\begin{proof}
The proof is similar to (\cite{SF1} 10.2). The only difference is that, if $T \ra \frakY_1$ is a closed immersion, we take $\frakY_{(T/p)}^{\#}$ to be the dilatation of the ideal $\calI^{\#}$ locally generated by $p$ and the $p$th powers of the sections of the ideal $\calI$ of $T \ra \frakY.$
\end{proof}

\begin{proposition}\label{propdiletale}
Let $\frakS=\op{Spf}W$ \eqref{Wflat} equipped with the trivial logarithmic structure, $\frakX$ and $\frakY$ two logarithmic $p$-adic formal schemes flat and locally of finite type over $\frakS$ and $f:\frakX \ra \frakY$ a log étale $\frakS$-morphism. Suppose that there exists two strict $\frakS_1$-immersions $i:T \ra \frakX_1$ and $j:T \ra \frakY_1$ such that $f_1\circ i=j.$ Then $f$ induces a canonical isomorphism of logarithmic formal schemes
$$\frakX_{(T/p)}^{\#} \xrightarrow{\sim} \frakY_{(T/p)}^{\#}.$$
\end{proposition}

\begin{proof}
The proof is similar to (\cite{SF1} 10.4)
\end{proof}

\begin{proposition}\label{propdilflat}
Let $\frakS=\op{Spf}W$ \eqref{Wflat} equipped with the trivial logarithmic structure, $\frakX$ and $\frakY$ two logarithmic $p$-adic formal schemes flat and locally of finite type over $\frakS,$ $f:\frakX \ra \frakY$ an $\frakS$-morphism which is flat on the underlying formal schemes, $T \ra \frakY_1$ a strict immersion and $S=\frakX\times_{\frakY}T.$ Then $f$ induces a canonical isomorphism
$$\frakY_{(T/p)}^{\#}\times_{\frakY}\frakX \xrightarrow{\sim} \frakX_{(S/p)}^{\#}.$$
\end{proposition}

\begin{proof}
The proof is similar to (\cite{SF1} 10.5).
\end{proof}

\begin{parag}\label{parag86}
In the rest of this article, we let $f:(\frakX,Q)\ra (\frakS,P)$ be a log smooth morphism of framed fs logarithmic $p$-adic formal schemes (\cite{SF1} 8.3), log flat and locally of finite type over $\op{Spf}W.$ Note that this implies, by \ref{Wflat}, that $\frakX$ and $\frakS$ are flat over $\op{Spf}W$ \eqref{Wflat}. 
We now recall the formal groupoid $R_{\frakX,n}$ defined in (\cite{SF1} 11.2) and introduce a new groupoid that we denote by $Q_{\frakX}.$
Let $r$ be a nonnegative integer. By (\cite{SF1} 8.12), the diagonal immersion $\frakX \ra \frakX^{r+1}_{\frakS}$ factors as
$$\begin{tikzcd}
 & \frakX_{\frakS,\left [Q\right ]}^{r+1}\ar{d} \\
\frakX \ar{r} \ar{ur}{\Delta(r)}&  \frakX^{r+1}_{\frakS},
\end{tikzcd}$$
where $\frakX^{r+1}_{\frakS}$ denotes the fiber product, in the category of fs logarithmic $p$-adic formal schemes, of $\frakX$ over $\frakS$ with itself $r+1$ times, $\Delta(r):\frakX\ra \frakX_{\frakS,[Q]}^{r+1}$ is a strict immersion and $\frakX_{\frakS,[Q]}^{r+1}\ra \frakX_{\frakS}^{r+1}$ is log étale and affine. Set $\frakY(r)=\frakX_{\frakS,[Q]}^{r+1},$ let $\frakX \ra \frakU(r) \ra \frakY(r)$ be a factorization of $\Delta(r)$ into a closed strict immersion followed by an open one and denote by $\calI_{\frakX/\frakS}(r)$ the ideal of $\frakX \ra \frakU(r).$ The canonical morphism $\frakY(r) \ra \frakS$ is log smooth hence locally of finite type, so $\frakY(r)$ is locally of finite type over $\op{Spf}W.$ The projections $\frakY(r) \ra \frakX$ are strict (\cite{SF1} 8.9) and log smooth hence smooth. It follows that $\frakY(r)$ is flat over $\op{Spf}W.$ 

For any integer $n\ge 1,$ denote by $R_{\frakX,n}(r)$ (resp. $Q_{\frakX}(r)$) the dilatation defined in (\cite{SF1} 8.2) (resp. $\frakY(r)_{(X/p)}^{\#}$ defined in \ref{erafdil2}). We also denote, for a positive integer $k,$ by $P_{\frakX/\frakS}(r)_k$ the PD-envelope of $\frakX_k \ra \frakY(r)_k.$ Just as in (\cite{SF1} 11.2), we set $P_{\frakX/\frakS}(r)=\lim\limits_{\longrightarrow}P_{\frakX/\frakS}(r)_k.$ By (\cite{SF1} 11.2) (resp. the universal property of $Q_{\frakX}$), there exists a unique strict closed immersion of logarithmic formal schemes $\frakX\ra R_{\frakX,n}(r)$ (resp. $\frakX\ra Q_{\frakX}(r)$) making the following diagrams commutative
$$\begin{tikzcd}
 & R_{\frakX,n}(r)\ar{d} \\
 & \frakY(r)\ar{d} \\
\frakX\ar{uur}\ar{ur}\ar{r} & \frakX^{r+1}_{\frakS}
\end{tikzcd} 
\begin{tikzcd}
 & Q_{\frakX}(r)\ar{d} \\
 & \frakY(r)\ar{d} \\
\frakX\ar{uur}\ar{ur}\ar{r} & \frakX^{r+1}_{\frakS}
\end{tikzcd}
$$
Following (\cite{SF1} 11.2), for $r=1,$ we will drop $(r)$ from the notation we just introduced.
\end{parag}

\begin{proposition}\label{erafprop13}
Let $Q_{1}=\left (Q_{\frakX}\right )_1.$ Denote by $\underline{Q_1}$ the logarithmic scheme-theoretic image of the Frobenius of $Q_1$ \eqref{era3logimagedef}. The composition $\underline{Q_1} \ra Q_1 \ra Y,$ of the canonical morphism $Q_1 \ra Y$ and the canonical closed immersion $\underline{Q_1} \ra Q_1$ factors uniquely through the strict diagonal immersion $X\ra Y$. In addition, $\underline{Q_1} \ra X$ is affine.
\end{proposition}

\begin{proof}
This follows from the definition of $Q_{\frakX}$ and \ref{erafdil2}.
\end{proof}

\begin{proposition}[\cite{SF1} 11.6]
For all positive integers $n,$ $R_{\frakX,n}$ has a natural formal groupoid structure.
\end{proposition}

\begin{proposition}
Consider the morphism
$\alpha :Q_{\frakX}\times_{\frakX}Q_{\frakX} \xrightarrow{\sim} Q_{\frakX}(2) \ra Q_{\frakX}$
induced by the projection $\frakY(2)=\frakX \times_{\frakS,[Q]}^{\op{log}}\frakX \times_{\frakS,[Q]}^{\op{log}} \frakX\ra \frakX \times_{\frakS,[Q]}^{\op{log}} \frakX=\frakY$ on the first and third factors,
the morphism $\iota:\frakX\ra Q_{\frakX}$
and $\eta :Q_{\frakX} \ra Q_{\frakX}$
induced by the morphism $\frakY\ra \frakY$ that exchanges factors.
Then the morphisms $\alpha,\ \iota$ and $\eta$ define a $p$-adic $\frakX$-groupoid structure on $Q_{\frakX}$ (\cite{DXU19} 4.7).
\end{proposition}

\begin{proof}
The fact that the canonical morphism $Q_{\frakX} \ra \frakY$ factors, as a map of topological spaces, through $\Delta:\frakX \ra \frakY,$ follows from \ref{erafprop13}. We omit the proof of the commutativity of the usual diagrams defining the groupoid structure.
\end{proof}

\begin{parag}\label{paragomega}
Let $\frakG$ be a $p$-adic $\frakX$-groupoid (\cite{DXU19} 4.7) and $\omega:\frakG_{\text{zar}} \ra \frakX_{\text{zar}}$ the morphism of topoi induced by the factorization of the morphism of underlying topological spaces of $\frakG \ra \frakX^2_{\frakS}$ through the diagonal. Let $q_1,q_2:\frakG \ra \frakX$ be the projections. Then
$$\omega_*\Ox_{\frakG}=q_{1*}\Ox_{\frakG}=q_{2*}\Ox_{\frakG}.$$
In this way, we see $\Ox_{\frakG}$ as a bialgebra of $\frakX_{\text{zar}}.$ The groupoid structure on $\frakG$ induces a formal Hopf algebra structure on $\Ox_{\frakG}.$ For any positive integers $r$ and $n,$ we denote by $\calR_{\frakX,n}(r)$ and $\calQ_{\frakX}(r)$ the formal Hopf algebras $\Ox_{R_{\frakX,n}(r)}$ and $\Ox_{Q_{\frakX}(r)}$ respectively. 
\end{parag}

\begin{proposition}\label{locdescRQ}
Consider $Q_{\frakX}$ as a logarithmic formal scheme over $\frakX$ via the first projection $q_1: Q_{\frakX} \ra \frakX.$ Suppose that there exists a chart $\alpha:P \ra M$ of $f:\frakX \ra \frakS,$ fitting into a commutative diagram
$$
\begin{tikzcd}
\frakX \ar{r} \ar{d} & \frakS \ar{d} \\
B \langle M \rangle \ar{r} \ar{d} & B\langle P \rangle \ar{d} \\
\left [Q \right ] \ar{r} & \left [P \right ],
\end{tikzcd}
$$ 
satisfying the conditions :
\begin{enumerate}
\item $\alpha^{gp}$ is injective and the torsion subgroup of $\op{coker}\alpha^{gp}$ is finite and of order coprime with $p.$
\item The morphism $g:\frakX \ra \frakS \times_{B\langle P\rangle}B\langle M\rangle,$ induced by $f$ and the chart $\frakX \ra B\langle M\rangle,$ is étale and strict.
\end{enumerate}
Suppose that there exist $m_1,\hdots,m_d \in \Gamma\left (\frakX,\calM_{\frakX}\right )$ lifting local coordinates of $X\ra S.$ Let $\eta_1,\hdots,\eta_d \in \Gamma \left (\frakX,\Delta^{-1}\calI_{\frakX/\frakS} \right )$ be as defined in \eqref{parag77} and consider them as sections of $\Ox_{Q_{\frakX,n}}$. Then, there exists an isomorphism of $\Ox_{\frakX}$-algebras:
\begin{equation}\label{isoQ1}
\Ox_{\frakX}\{x_1,\hdots,x_d,y_1,\hdots,y_d\}/(y_1^p-px_1,\hdots,y_d^p-px_d) \xrightarrow{\sim} q_{1*}\Ox_{Q_{\frakX}},
\end{equation}
sending $x_i$ to $\frac{\eta_i^p}{p}$ and $y_i$ to $\eta_i$.
\end{proposition}

\begin{proof}
The proof is similar to (\cite{SF1} 11.8) so we just outline it. Set $\frakT=\frakS \times_{B\langle P\rangle}B\langle M\rangle.$
Just as in the proof of (\cite{SF1} 11.8), we have
$$\frakY = \left (\frakX \times_{\frakS}^{\op{log}}\frakX \right )\times _{B\langle \left (M\oplus_PM\right )^{sat}\rangle }B\langle N \rangle,$$
$$\frakT \times_{\frakS,[M]}^{\op{log}}\frakT = \left (\frakT \times_{\frakS}^{\op{log}}\frakT \right )\times _{B\langle \left (M\oplus_PM\right )^{sat}\rangle }B\langle N \rangle=\frakS \times_{B\langle P\rangle} B\langle N\rangle,$$
where $N$ is the inverse image of $M$ by
$\left (M\oplus_PM\right )^{gp} \ra M^{gp},\ (x,y)\mapsto x+y.$
Using \ref{propdiletale} and \ref{propdilflat}, we prove that we have isomorphisms
\begin{alignat*}{2}
Q_{\frakX}=\frakY^{\#}_{(X/p)} & \xrightarrow{\sim} \left (\frakX \times_{\frakS,[M]}^{\op{log}}\frakT\right )^{\#}_{(X/p)} \\
  & \xrightarrow{\sim} \left (\frakT \times_{\frakS,[M]}^{\op{log}}\frakT\right )_{(T/p)}^{\#} \times_{\frakT \times_{\frakS,[M]}^{\op{log}}\frakT}\left (\frakX \times_{\frakS,[M]}^{\op{log}}\frakT\right ) \\
& \xrightarrow{\sim} \frakX \times_{\frakT} \left (\frakT \times_{\frakS,[M]}^{\op{log}}\frakT\right )_{(T/p)}^{\#},
\end{alignat*}
where $\left (\frakT \times_{\frakS,[M]}^{\op{log}}\frakT\right )_{(T/p)}^{\#}$ is considered as a logarithmic scheme over $\frakT$ via the first projection.
We can hence reduce to the case $\frakX=\frakT=\frakS \times_{B\langle P\rangle}B\langle M\rangle$ and we can also suppose that $\frakS=\op{Spf}A$ for a $p$-adic algebra $A.$ In this case,
\begin{alignat*}{2}
\frakX &= \frakS \times_{B\langle P\rangle}B\langle M\rangle = \op{Spf} \left ( A \widehat{\otimes}_{\Z_p\langle P \rangle} \Z_p \langle M\rangle \right ) = \op{Spf} C, \\
\frakY &= \frakS \times_{B\langle P \rangle} B\langle N \rangle = \op{Spf} \left ( A\widehat{\otimes}_{\Z_p\langle P\rangle } \Z_p\langle N\rangle \right ) = \op{Spf} D,
\end{alignat*}
where $\widehat{\otimes}$ is the $p$-adically completed tensor product. Let $I$ be the kernel of the exact diagonal $D\ra C$ and $a_1,\hdots,a_d\in I$ the generators corresponding to $\eta_1,\hdots,\eta_d.$ Then, by definition of $Q_{\frakX}$ and (\cite{SF1} 10.1.1), we have
\begin{alignat*}{2}
Q_{\frakX} &= \op{Spf}\left ( \left ( \frac{D\{ x_1,\hdots,x_d\}}{\left (px_1-a_1^p,\hdots,px_d-a_d^p \right )} \right )_{/p\text{-tor}} \right ),
\end{alignat*}
where $p$-tor denotes the $p$-torsion ideal.
The result then follows from the lemma \ref{lemQ} below.
\end{proof}

\begin{lemma}\label{lemQ}
Suppose that $\frakS=\op{Spf}A$ for a $p$-adic algebra $A$ and $\frakX=\frakS \times_{B\langle P\rangle}B\langle M\rangle=\op{Spf} C.$ The morphism $\frakX \ra \frakS$ corresponds then to a continuous morphism $\xi:A \ra C.$ In this case, if $\frakX \ra \frakY$ denotes the exact diagonal immersion, then $\frakY=\frakS \times_{B\langle P \rangle} B\langle N \rangle=\op{Spf} D,$ where $N$ is the inverse image of $M$ by
$$\left (M\oplus_PM\right )^{gp} \ra M^{gp},\ (x,y)\mapsto x+y,$$
and
$$
D=A\widehat{\otimes}_{\Z_p\langle P\rangle } \Z_p\langle N\rangle.
$$
Let $I$ be the ideal of the exact diagonal $D \ra C,$ $I_k$ its reduction modulo $p^k$ for $k\ge 1$ and $a_1,\hdots,a_d\in I$ whose images in $I_1/I_1^2$ are a basis of the $\left ( C/pC \right )$-module $I_1/I_1^2.$ The morphism of $C$-modules
$$\varphi: \frac{C\{x_1,\hdots,x_d,y_1,\hdots,y_d\}}{(px_1-y_1^p,\hdots,px_r-y_r^p)} \ra \left ( \frac{D\{ x_1,\hdots,x_d\}}{\left (px_1-a_1^p,\hdots,px_d-a_d^p \right )} \right )_{/p\text{-tor}},$$
sending $x_i$ to $x_i,$ $y_i$ to $a_i$ and $C$ to $D$ via the first projection, is an isomorphism.
\end{lemma}

\begin{proof}
Denote by $\pi:C \ra D$ the morphism corresponding to the first projection $\frakY \ra \frakX$ and $\Delta:D \ra C$ the morphism corresponding to the exact diagonal $\frakX \ra \frakY.$ Note that $\Delta \circ \pi=\op{Id}_{C}.$
For $g\in D \{x_1,\hdots,x_d\}$ and $J=(J_1,\hdots,J_d)\in \N^d,$ set
$x^J=\prod_{j=1}^dx_j^{J_j}$ and $a^J=\prod_{j=1}^da_j^{J_j}$
and denote by $g_J$ the coefficient of $x^J$ in $g.$
Let us prove that $\varphi$ is injective. Let
$$f=\sum_{\substack{K\in \N^d \\ J\in \llbracket 0,p-1\rrbracket^d}}\alpha_{K,J}x^Ky^J\in C\{x_1,\hdots,x_d,y_1,\hdots,y_d\}$$
such that $\varphi(f)=0.$
There exists a positive integer $l$ and $g_1,\hdots,g_d\in D \{x_1,\hdots,x_d\}$ such that
\begin{equation}\label{eq1291}
p^l\sum_{\substack{K\in \N^d \\ J\in \llbracket 0,p-1\rrbracket^d}}\pi(\alpha_{K,J})a^Jx^K=\sum_{i=1}^dg_i(px_i-a_i^p).
\end{equation}
From here on, we work modulo $p^k$ for a fixed positive integer $k.$
We prove by induction on $n$ that $p^l\alpha_{K,J}=0$ for all $(K,J)\in \N^d \times \llbracket 0,p-1\rrbracket^d$ such that $|K|=n$ and that 
$
\sum_{\substack{1\le i\le d \\ |K|=n}}g_{i,K}a^{pK+p\epsilon_i}=0.
$
Taking the constant coefficients of both sides in \eqref{eq1291}, we get
\begin{equation}\label{eq1294}
p^l\sum_{J\in \llbracket 0,p-1\rrbracket^d}\pi(\alpha_{0,J})a^J=-\sum_{i=1}^dg_{i,0}a_i^p.
\end{equation}
We rewrite this as
\begin{equation}\label{eq1297}
\sum_{J\in \N^d}\beta_Ja^J=0,
\end{equation}
where
$$\beta_J=\begin{cases} p^l\pi(\alpha_{0,J})\ \text{if}\ J\in \llbracket 0,p-1\rrbracket^d,\\ g_{i,0}\ \text{if}\ J=p\epsilon_i,\\  0\ \text{otherwise.}\end{cases}$$ 
Proceeding by induction on $|J|,$ reducing modulo $I_k,$ $I_k^2,\hdots$ and using (\cite{SF1} (11.3.1)), we prove that $\Delta(\beta_J)=0$ for all $J.$ Indeed, applying $\Delta$ to \eqref{eq1297}, we get
$\Delta(\beta_0)=0.$
Let $m\ge 0$ be an integer and suppose that $\Delta(\beta_J)=0$ for $|J|\le m.$ In the $\left (C/p^kC\right )$-module $I_k^{m+1}/I_k^{m+2},$ we have
$\sum_{|J|=m+1}\Delta (\beta_J)a^J=0.$
By (\cite{SF1} 11.9), the family $(a_1,\hdots,a_d)$ is a basis of the $\left (C/p^kC\right )$-module $I_k/I_k^2.$ Then, by (\cite{SF1} (11.3.1)), $(a^J)_{|J|=m+1}$ is free in the $\left (C/p^kC\right )$-module $I_k^{m+1}/I_k^{m+2}$ so $\Delta(\beta_J)=0$ for $|J|=m+1.$
In particular, if $J\in \llbracket 0,p-1\rrbracket^d,$ then
$$\Delta(\beta_J)=p^l\Delta\circ \pi(\alpha_{0,J})=p^l\alpha_{0,J}=0.$$
By \eqref{eq1294}, we deduce that
$\sum_{i=1}^dg_{i,0}a_i^p=0.$
This concludes the first step of the induction.

Now taking the monomials of degree $1$ in \eqref{eq1291}, we get
$$p^l\sum_{\substack{1\le i\le d \\ J\in \llbracket 0,p-1 \rrbracket^d}} \pi(\alpha_{\epsilon_i,J})a^Jx_i=\sum_{i=1}^dpg_{i,0}x_i-\sum_{1\le i,j\le d}g_{i,\epsilon_j}a_i^px_j.$$
Taking $x_i=a_i^p,$ we get
$$p^l\sum_{\substack{1\le i\le d \\ J\in \llbracket 0,p-1 \rrbracket^d}} \pi(\alpha_{\epsilon_i,J})a^{J+p\epsilon_i}=-\sum_{1\le i,j\le d}g_{i,\epsilon_j}a^{p\epsilon_i+p\epsilon_j}.$$
Just as in the previous case, we prove that $p^l\alpha_{\epsilon_i,J}=0$ for all $1\le i\le d$ and $J.$
Let $n$ be a positive integer and suppose that $p^l\alpha_{K,J}=0$ for all $K\in \N^d$ and $J\in \llbracket 1,p-1\rrbracket^d$ such that $|K|\le n$ and that
\begin{equation}\label{eq1293}
\sum_{\substack{|K|= n \\ 1\le i\le d}}g_{i,I}a^{pK+p\epsilon_i}=0.
\end{equation}
Taking the monomials of degree $n+1$ in \eqref{eq1291}, we get
$$p^l\sum_{\substack{|K|=n+1 \\ J\in \llbracket 0,p-1 \rrbracket^d}} \pi(\alpha_{K,J})a^Jx^K=\sum_{\substack{|K|=n \\ 1\le i\le d}}pg_{i,K}x^{K+\epsilon_i}-\sum_{\substack{|K|=n+1 \\ 1\le i\le d}}g_{i,K}a_i^px^K.$$
Taking $x_i=a_i^p$ for all $1\le i\le d,$ we get
$$p^l\sum_{\substack{|K|=n+1 \\ J\in \llbracket 0,p-1 \rrbracket^d}} \pi(\alpha_{K,J})a^{J+pK}=\sum_{\substack{|K|=n \\ 1\le i\le d}}pg_{i,K}a^{pI+p\epsilon_i}-\sum_{\substack{|K|=n+1 \\ 1\le i\le d}}g_{i,K}a^{pK+p\epsilon_i}.$$
By \eqref{eq1293}, we get
$$ p^l\sum_{\substack{|K|=n+1 \\ J\in \llbracket 0,p-1 \rrbracket^d}} \pi(\alpha_{K,J})a^{J+pK}=-\sum_{\substack{|K|=n+1 \\ 1\le i\le d}}g_{i,K}a^{pK+p\epsilon_i}.$$
Again, using the same argument as in the case $n=0,$ we prove that $p^l\alpha_{K,J}=0$ for $|K|=n+1$ and $J\in \llbracket 0,p-1 \rrbracket^d$ and then
$$\sum_{\substack{|K|= n+1 \\ 1\le i\le d }}g_{i,K}a^{pK+p\epsilon_i}=0.
$$
This concludes the induction, and since this is true modulo $p^k$ for all $k,$ this concludes the injectivity of $\varphi.$ The surjectivity follows from the fact that, for $y\in D,$
$y=\pi \circ \Delta(y)+y-\pi \circ \Delta(y)$
and $y-\pi \circ \Delta(y)\in I.$
\end{proof}

\begin{remark}\label{par1212}
By (\cite{SF1} 11.11), the isomorphism \eqref{isoQ1} exists étale locally on $\frakX$ and $\frakS.$
\end{remark}

\begin{parag}\label{loccoord}
We will often make the following hypothesis:
suppose that we have local coordinates $m_1,\hdots,m_d\in \Gamma(X,\calM_X)$ for $X\ra S$ that lift to local sections $\widetilde{m}_1,\hdots,\widetilde{m}_d\in \Gamma(\frakX,\calM_{\frakX}).$ For every $1\le i\le d,$ let $m_i'=\pi^{\flat}m_i$ \eqref{diag51} and $\widetilde{\eta}_i\in \Gamma(\frakX,\Delta^{-1}\calI_{\frakX/\frakS})$ be as defined in \ref{parag77} and $\eta_i$ the reduction of $\widetilde{\eta}_i$ modulo $p.$ Suppose also that we have a chart of $\frakX \ra \frakS$ satisfying the conditions of \ref{locdescRQ}. Let $n$ be a positive integer. By \ref{locdescRQ} and  (\cite{SF1} 11.15), we have isomorphisms:
\begin{alignat}{2}
\calR_{\frakX,n} & \xrightarrow{\sim} \Ox_{\frakX}\left \{ \frac{\widetilde{\eta}_1}{p^n},\hdots,\frac{\widetilde{\eta}_d}{p^n} \right \}, \label{isoR} \\
\calQ_{\frakX} & \xrightarrow{\sim} \Ox_{\frakX}\{ \widetilde{\eta}_1,\hdots,\widetilde{\eta}_d,y_1,\hdots,y_d\}/(py_1-\widetilde{\eta}_1^p,\hdots,py_d-\widetilde{\eta}_d^p), \label{isoQ} \\
\calP_{\frakX/\frakS} & \xrightarrow{\sim} \Ox_{\frakX}\left \langle \left \langle \widetilde{\eta}_1,\hdots,\widetilde{\eta}_d \right \rangle \right \rangle. \label{eqPkraz6}
\end{alignat}
Let $\eta_{i(r),n}$ be the image of $\frac{\widetilde{\eta}_i}{p^n}$ in $\calR_n:=\calR_{\frakX,n}/p\calR_{\frakX,n}$ and $\eta_{i(q)}$ the image of $\frac{\widetilde{\eta}_i^p}{p}$ in $\calQ_1:=\calQ_{\frakX}/p\calQ_{\frakX}.$ For $n=1,$ we denote $ \eta_{i(r),1}$ simply by $\eta_{i(r)}.$ For any $I\in \N^d,$ we set
$$
\eta^I=\prod_{i=1}^d\eta_i^{I_i},\ \eta_{(q)}^I=\prod_{i=1}^d\eta_{i(q)}^{I_i},\ \eta_{(r)}^I=\prod_{i=1}^d\eta_{i(r)}^{I_i}.
$$
Then $\left (\eta^I\eta_{(q)}^J \right )_{\substack{I\in \llbracket 0,p-1 \rrbracket^d \\ J\in \N^d}}$ is a basis of the $\Ox_X$-module $\calQ_1.$ We denote by $\left ( \varphi_{I,J} \right )$ its dual basis. Also, $\left (\eta^{[I]} \right )_{I\in \N^d}$ is a basis for the $\Ox_X$-module $\calP_0:=\calP_{\frakX/\frakS}/p\calP_{\frakX/\frakS}.$
If $X'$ lifts to a log smooth framed fs logarithmic $p$-adic formal $(\frakS,P)$-scheme $(\frakX',Q'),$ we define $\eta_i',$ $\widetilde{\eta}_i',$ $\eta_{i(r)}'$ and $\eta_{i(q)}'$ in a similar way from $m_i'.$ We also denote by $(\partial_1',\hdots,\partial_d')$ the dual basis of $(\op{dlog}m_1',\hdots,\op{dlog}m_d').$
\end{parag}

\begin{corollaire}\label{Qlogflat}
For all positive integers $n,$ the logarithmic formal schemes $Q_{\frakX},$ $P_{\frakX/\frakS}$ and $R_{\frakX,n}$ are both log flat and flat over $\frakX$ and $\op{Spf}W.$
\end{corollaire}

\begin{proof}
By \ref{loccoord} and \ref{par1212}, $Q_{\frakX}$ and $R_{\frakX,n}$ are flat over $\frakX.$ Since the projections $Q_{\frakX}\ra \frakX$ and $R_{\frakX,n}\ra \frakX$ are also strict, they are log flat. We conclude by the log flatness of $\frakX$ over $\op{Spf}W$ and \ref{Wflat}.
\end{proof}

\begin{proposition}[\cite{SF1} 11.17]\label{HopfR}
Let $n$ be a positive integer, $\calR_n=\calR_{\frakX,n}/p\calR_{\frakX,n}$ and
$
\delta:\calR_n \ra \calR_n\otimes_{\Ox_X}\calR_n,
\pi:\calR_n \ra \Ox_X,
\sigma:\calR_n \ra \calR_n
$
the morphisms defining the Hopf algebra structure on $\calR_n.$ Under the hypothesis \ref{loccoord},
we have
\begin{alignat*}{2}
\delta\left (\eta_{i(r),n} \right ) &= 1\otimes \eta_{i(r),n}+\eta_{i(r),n}\otimes 1, \\
\pi\left (\eta_{i(r),n} \right ) &= 0,\\
\sigma\left (\eta_{i(r),n} \right ) &= -\eta_{i(r),n}.
\end{alignat*}
\end{proposition}

\begin{proposition}\label{HopffrakQ}
Let
$
\widetilde{\delta}:\calQ_{\frakX} \ra \calQ_{\frakX}\otimes_{\Ox_{\frakX}}\calQ_{\frakX},
\widetilde{\pi}:\calQ_{\frakX} \ra \Ox_{\frakX},
\widetilde{\sigma}:\calQ_{\frakX} \ra \calQ_{\frakX}
$
the morphisms defining the Hopf algebra structure on $\calQ_{\frakX}.$ Under the hypothesis \ref{loccoord},
we have
\begin{alignat*}{2}
\widetilde{\delta}(\widetilde{\eta}_i) &= 1\otimes \widetilde{\eta}_i+\widetilde{\eta}_i\otimes 1+\widetilde{\eta}_i\otimes \widetilde{\eta}_i, \\
\widetilde{\pi}(\widetilde{\eta}_i) &= 0,\\
\widetilde{\sigma}(\widetilde{\eta}_i) &= (1+\widetilde{\eta}_i)^{-1}-1.
\end{alignat*}
\end{proposition}

\begin{proof}
This is a result of (\cite{SF1} 8.13) and (\cite{SF1} 8.14).
\end{proof}

\begin{proposition}\label{HopfQ}
Let $\calQ_1=\calQ_{\frakX}/p\calQ_{\frakX}$ and
$
\delta:\calQ_1 \ra \calQ_1\otimes_{\Ox_X}\calQ_1,
\pi:\calQ_1 \ra \Ox_X,
\sigma:\calQ_1 \ra \calQ_1
$
the morphisms defining the Hopf algebra structure on $\calQ_1.$ Under the hypothesis \ref{loccoord},
we have
\begin{alignat*}{2}
\delta\left (\eta_{i(q)} \right ) &= 1\otimes \eta_{i(q)}+\sum_{0<b+c<p} \frac{(-1)^{b+c}}{b+c}\begin{pmatrix}b+c \\ b \end{pmatrix} \eta_i^{b+c}\otimes \eta_i^{p-b} +\eta_{i(q)}\otimes 1, \\
\delta(\eta_i) &= 1\otimes \eta_i+\eta_i\otimes 1+\eta_i\otimes \eta_i, \\
\pi\left (\eta_{i(q)} \right ) &= \pi(\eta_i)= 0,\\
\sigma\left (\eta_{i(q)} \right ) &= \frac{-\eta_{i(q)}}{(1+\eta_i)^p}.
\end{alignat*}
\end{proposition}

\begin{proof}
Let
$
\widetilde{\delta}:\calQ_1 \ra \calQ_1\otimes_{\Ox_X}\calQ_1,
\widetilde{\pi}:\calQ_1 \ra \Ox_X,
\widetilde{\sigma}:\calQ_1 \ra \calQ_1
$
be the morphisms defining the Hopf algebra structure on $\calQ_{\frakX}.$ By \ref{HopffrakQ}, we have
\begin{alignat*}{2}
p\widetilde{\delta}\left (\frac{\widetilde{\eta}_i^p}{p}\right ) =& \widetilde{\delta}(\widetilde{\eta}_i)^p 
= \left (1\otimes \widetilde{\eta}_i + \widetilde{\eta}_i\otimes 1 + \widetilde{\eta}_i\otimes \widetilde{\eta}_i\right )^p
= \sum_{a+b+c=p} \begin{pmatrix}p \\ b+c \end{pmatrix} \begin{pmatrix} b+c \\ b\end{pmatrix} \widetilde{\eta}_i^{b+c} \otimes \widetilde{\eta}_i^{a+c} \\
=& 1\otimes \widetilde{\eta}_i^p + \sum_{0<b+c<p} \begin{pmatrix}p \\ b+c \end{pmatrix} \begin{pmatrix} b+c \\ b\end{pmatrix} \widetilde{\eta}_i^{b+c} \otimes \widetilde{\eta}_i^{p-b} + \sum_{b+c=p}\begin{pmatrix}p \\ b \end{pmatrix}\widetilde{\eta}_i^p\otimes \widetilde{\eta}_i^c \\
=& 1\otimes \widetilde{\eta}_i^p + \sum_{0<b+c<p} \begin{pmatrix}p \\ b+c \end{pmatrix} \begin{pmatrix} b+c \\ b\end{pmatrix} \widetilde{\eta}_i^{b+c} \otimes \widetilde{\eta}_i^{p-b} + \sum_{b=1}^{p-1}\begin{pmatrix}p \\ b \end{pmatrix}\widetilde{\eta}_i^p\otimes \widetilde{\eta}_i^{p-b} + \widetilde{\eta}_i^p\otimes 1 + \widetilde{\eta}_i^p\otimes \widetilde{\eta}_i^p \\
=& p \left (1\otimes \frac{\widetilde{\eta}_i^p}{p}\right )  + \sum_{0<b+c<p} \begin{pmatrix}p \\ b+c \end{pmatrix} \begin{pmatrix} b+c \\ b\end{pmatrix} \widetilde{\eta}_i^{b+c} \otimes \widetilde{\eta}_i^{p-b} + \sum_{b=1}^{p-1}\begin{pmatrix}p \\ b \end{pmatrix}\widetilde{\eta}_i^p\otimes \widetilde{\eta}_i^{p-b} + p\left (\frac{\widetilde{\eta}_i^p}{p}\otimes 1 \right )\\
& + p^2\left (\frac{\widetilde{\eta}_i^p}{p}\otimes \frac{\widetilde{\eta}_i^p}{p} \right ).
\end{alignat*}
By the flatness of $Q_{\frakX}$ over $\op{Spf}W$ (\ref{Qlogflat} and \ref{Wflat}), we get
\begin{alignat*}{2}
\widetilde{\delta}\left (\frac{\widetilde{\eta}_i^p}{p} \right )=&  1\otimes \frac{\widetilde{\eta}_i^p}{p}  + \sum_{0<b+c<p} \frac{(p-1)!}{(b+c)!(p-b-c)!} \begin{pmatrix} b+c \\ b\end{pmatrix} \widetilde{\eta}_i^{b+c} \otimes \widetilde{\eta}_i^{p-b} \\
& + \sum_{b=1}^{p-1}\frac{(p-1)!}{b!(p-b)!}\widetilde{\eta}_i^p\otimes \widetilde{\eta}_i^{p-b} + \frac{\widetilde{\eta}_i^p}{p}\otimes 1  + p\left (\frac{\widetilde{\eta}_i^p}{p}\otimes \frac{\widetilde{\eta}_i^p}{p} \right ).
\end{alignat*}
Since
$$
\sum_{b=1}^{p-1}\frac{(p-1)!}{b!(p-b)!}\widetilde{\eta}_i^p\otimes \widetilde{\eta}_i^{p-b}=\sum_{b=1}^{p-1}\begin{pmatrix}p\\ b \end{pmatrix} \frac{\widetilde{\eta}_i^p}{p}\otimes \widetilde{\eta}_i^{p-b},
$$
this sum vanishes modulo $p.$
By \ref{HopffrakQ},
$
(1+\widetilde{\eta}_i)\widetilde{\sigma}(\widetilde{\eta}_i)=-\widetilde{\eta}_i
$
so
$$
p(1+\widetilde{\eta}_i)^p\widetilde{\sigma}\left (\frac{\widetilde{\eta}_i^p}{p}\right )=(1+\widetilde{\eta}_i)^p\widetilde{\sigma}(\widetilde{\eta}_i^p)=-\widetilde{\eta}_i^p=-p\frac{\widetilde{\eta}_i^p}{p}.
$$
It follows, by the flatness of $Q_{\frakX}$ over $\op{Spf}W$ (\ref{Qlogflat} and \ref{Wflat}), that
$
(1+\widetilde{\eta}_i)^p\widetilde{\sigma} \left (\frac{\widetilde{\eta}_i^p}{p} \right )=-\frac{\widetilde{\eta}_i^p}{p}.
$
The result then follows from \ref{HopffrakQ} and the fact that, for any $1\le k\le p-1,$ we have the equality $\frac{(p-1)!}{k!(p-k)!}=\frac{(-1)^k}{k}$ in $\mathbb{F}_p.$
\end{proof}

\begin{proposition}\label{lem92}
Suppose that $\frakX$ and $\frakS$ are equipped with frames $\frakX \ra [Q]$ and $\frakS\ra [P],$ where $P$ and $Q$ are fs monoids and that $f$ underlies a morphism of framed logarithmic formal schemes. Let $\Delta:\frakX \ra \frakY=\frakX\times_{\frakS,[Q]}^{\op{log}}\frakX$ and $\Delta':\frakX' \ra \frakY'= \frakX' \times_{\frakS,[Q']}^{\op{log}}\frakX'$ be the strict diagonal immersions. Suppose that there exists a morphism $F:(\frakX,Q) \ra (\frakX',Q')$ of framed logarithmic formal schemes, over $(\frakS,P),$ lifting the exact relative Frobenius $F_1:X \ra X'.$ Denote by $G:\frakY\ra \frakY'$ the morphism induced by $F:(\frakX,Q)\ra (\frakX',Q').$ Then, there exists a morphism of formal groupoids $\nu:Q_{\frakX} \ra R_{\frakX',1}$ fitting into the commutative diagram
$$
\begin{tikzcd}
Q_{\frakX} \ar{r}{\nu} \ar{d} & R_{\frakX',1} \ar{d} \\
\frakY \ar{r}{G} & \frakY'.
\end{tikzcd}
$$
\end{proposition}

\begin{proof}
Denote by $Q_1$ the special fiber of $Q_{\frakX}.$ Let $\frakX' \ra \frakU' \ra \frakY'$ be a factorization of $\Delta'$ into a closed immersion followed by an open one and denote by $\calI'$ the ideal of $\frakX' \ra \frakU'.$ By definition of $R_{\frakX',1}$ (\cite{SF1} 11.2) and (\cite{SF1} 10.2), it is sufficient to prove that the morphism $Q_1 \ra Y \xrightarrow{G_1} Y'$ factors through the strict diagonal immersion $X' \ra Y'.$ By (\cite{SF1} (9.4.1) and (9.4.2)) and the definition of $Q_{\frakX},$ if $a$ is a local section of $\calI',$ then the image of $G_1^{\#}(a)$ in $Q_1$ vanishes. This proves the result.
\end{proof}

\begin{remark}
Keep the assumption of \ref{lem92} and let $Q_1 \ra R'_1$ be the special fiber of $\nu:Q_{\frakX} \ra R_{\frakX',1}$ and let $\calV:\calR_1' \ra \calQ_1$ be the morphism of Hopf algebras induced by it. Let $q_1,q_2:Q_1 \ra X$ be the canonical projections. Suppose that the hypothesis \ref{loccoord} is satisfied. By (\cite{SF1} 9.3), there exist local sections $b_i$ of $\Ox_{\frakX}$ such that $F^{\flat}(\widetilde{m}_i')=p\widetilde{m}_i+\alpha_{\frakX}^{-1}(1+pb_i).$ If we denote by $c_i$ the image of $b_i$ in $\Ox_X,$ then, by (\cite{SF1} 9.4), we have
\begin{equation}\label{eqKoko1323}
\calV \left (\eta_{i(r)} \right ) = \eta_{i(q)} + \sum_{k=1}^{p-1}\frac{(-1)^{k+1}}{k}\eta_i^k + q_2^{\#}c_i-q_1^{\#}c_i.
\end{equation}
\end{remark}

\section{\texorpdfstring{$Q_{\frakX}$}{Q} %
     and \texorpdfstring{$R_{\frakX,n}$}%
     {R}-stratifications}

\begin{parag}
In this section, we keep the set-up of the previous one i.e. we keep the notation and assumption of \ref{parag86}. In other words, we consider a log smooth morphism of framed fs logarithmic $p$-adic formal schemes $f:(\frakX,Q)\ra (\frakS,P),$ such that $\frakS$ is log flat and locally of finite type over $\op{Spf}W.$ Note that this implies, by \ref{Wflat}, that the formal schemes $\frakX$ and $\frakS$ are flat over $\op{Spf}W$ \eqref{Wflat}. We also consider the formal groupoids $Q_{\frakX}$ and $R_{\frakX,n}$ defined in \ref{parag86}. We denote by $P_{\frakX/\frakS}$ the PD-envelope of the diagonal immersion $\frakX \ra \frakX\times_{\frakS}^{\op{log}}\frakX.$
\end{parag}

\begin{proposition}\label{wiw1}
There exists a unique morphism of logarithmic formal schemes
\begin{equation}\label{wiwkraz}
P_{\frakX/\frakS} \ra Q_{\frakX},
\end{equation}
such that the diagram
$$
\begin{tikzcd}
 & Q_{\frakX} \ar{d} \\
P_{\frakX/\frakS} \ar{r} \ar{ur} & \frakY
\end{tikzcd}
$$
is commutative. In addition, this morphism is compatible with the groupoid structures of $P_{\frakX/\frakS}$ and $Q_{\frakX}.$
\end{proposition}

\begin{proof}
If $X=\frakX_1 \ra U \ra \frakY_1=Y$ is a factorization of the exact diagonal immersion into a closed immersion followed by an open one and if $x$ is a local section of the ideal of $X\ra U$ then its image in $P_0= \left (P_{\frakX/\frakS}\right )_1$ satisfies
$x^p=p! x^{[p]}=0.$
It follows that the composition
$\underline{P_0} \ra P_0 \ra Y$
factors through $X.$
Since $P_{\frakX/\frakS}$ is flat over $\op{Spf}W$ \eqref{Qlogflat}, we conclude by the universal property of $Q_{\frakX}$ (\ref{erafdil2}).
\end{proof} 

\begin{parag}
Let $P_0$ and $Q_1$ be the special fibers of $P_{\frakX/\frakS}$ and $Q_{\frakX},$ and let $\calP_0$ and $\calQ_1$ be the Hopf algebras defined by $P_0$ and $Q_1$ respectively.
The morphism $P_{\frakX/\frakS} \ra Q_{\frakX}$ \eqref{wiw1} induces a morphism of groupoids $P_0 \ra Q_1$ and then a morphism of Hopf algebras
\begin{equation}\label{Muzanu1}
u:\calQ_1 \ra \calP_0.
\end{equation}
By taking the dual, we obtain a morphism of $\Ox_X$-algebras
\begin{equation}\label{checku}
\check{u}:\widehat{\calD}_{X/S} \ra \calQ_1^{\vee},
\end{equation}
where $\widehat{\calD}_{X/S}=\mathscr{Hom}_{\Ox_X}(\calP_0,\Ox_X)$ is the sheaf of hyper PD-differential operators.
Under the hypothesis \ref{loccoord}, since $(p-1)!\equiv -1\ (\text{mod}\ p),$ we clearly have, for all $1\le i\le d$ and with the notation \ref{loccoord},
\begin{equation}\label{ueta}
u(\eta_i)=\eta_i,\ u(\eta_{i(q)})=-\eta_i^{[p]}.
\end{equation}
For an upcoming proposition, we recall the canonical Hopf algebra structure on $S^{\bullet}\omega^1_{X/S}:$
\begin{equation}\label{HopfalgstrS}
\begin{alignedat}{2}
\delta : &\begin{array}[t]{clc}
S^{\bullet}\omega^1_{X/S} & \ra & S^{\bullet}\omega^1_{X/S} \otimes_{\Ox_X} S^{\bullet}\omega^1_{X/S} \\
x & \mapsto & 1\otimes x+x\otimes 1,
\end{array},\quad
\sigma : &\begin{array}[t]{clc}
S^{\bullet}\omega^1_{X/S} & \ra & S^{\bullet}\omega^1_{X/S} \\
x & \mapsto & -x,
\end{array} \\
\pi : &\begin{array}[t]{clc}
S^{\bullet}\omega^1_{X/S} & \ra & \Ox_X \\
x=(x_0,x_1,\hdots ) & \mapsto & x_0,
\end{array}
\end{alignedat}
\end{equation}
where $x_i \in S^{i}\omega^1_{X/S}.$
\end{parag}

\begin{proposition}\label{dualAkaza}
For a Hopf $\Ox_X$-algebra $\calA,$ let $\mathscr{Hom}_{\Ox_X}(\calA,\Ox_X)$ be the sheaf of $\Ox_X$-linear morphisms of $\Ox_X$-modules, equipped with the ring structure induced by the Hopf algebra structure on $\calA.$
Denote by $\widehat{S}^{\bullet}\calT_{X/S}$ the completion of $S^{\bullet}\calT_{X/S}$ with respect to the ideal
$$S^{\ge 1}\calT_{X/S}:=\bigoplus_{n\ge 1}S^n\calT_{X/S}.$$
Also denote by and $\widehat{\Gamma}^{\bullet}\calT_{X/S}$ the completed PD-algebra of $\calT_{X/S}.$
There exist canonical isomorphisms of $\Ox_X$-algebras
\begin{equation}\label{isodualAkaza}
s:\widehat{S}^{\bullet}\calT_{X/S}\xrightarrow{\sim} \mathscr{Hom}_{\Ox_X}(\Gamma^{\bullet}\omega^1_{X/S},\Ox_X),
\end{equation}
\begin{equation}\label{isodualAkaza2}
s':\widehat{\Gamma}^{\bullet}\calT_{X/S}\xrightarrow{\sim} \mathscr{Hom}_{\Ox_X}(S^{\bullet}\omega^1_{X/S},\Ox_X),
\end{equation}
such that the following diagram is commutative
$$
\begin{tikzcd}
\calT_{X/S} \ar[equal]{r} \ar{d} & \mathscr{Hom}_{\Ox_X}(\omega^1_{X/S},\Ox_X) \ar{d} \\
\widehat{S}^{\bullet}\calT_{X/S} \ar{r}{s} & \mathscr{Hom}_{\Ox_X}(\Gamma^{\bullet}\omega^1_{X/S},\Ox_X)
\end{tikzcd}
$$
where the left vertical arrow is the canonical one and the right vertical arrow sends $f:\omega^1_{X/S} \ra \Ox_X$ to the morphism induced by $f$ and the zero maps $0:\Gamma^{\ge 2}\omega^1_{X/S} \ra \Ox_X$ and $0:\Ox_X \ra \Ox_X.$
\end{proposition}

\begin{proof}
The $\Ox_X$-module $\omega^1_{X/S}$ is locally free of finite rank so, by (\cite{Ogus78} A10), there exists, for any integer $n\ge 0,$ canonical isomorphisms
$$s_n:S^{n}\calT_{X/S} \xrightarrow{\sim}\mathscr{Hom}_{\Ox_X}(\Gamma^n\omega^1_{X/S},\Ox_X),$$
$$S^{n}\omega^1_{X/S} \xrightarrow{\sim}\mathscr{Hom}_{\Ox_X}(\Gamma^{n}\calT_{X/S},\Ox_X).$$
By taking the dual of the second arrow, we obtain an isomorphism
$$s_n':\Gamma^{n}\calT_{X/S} \xrightarrow{\sim}\mathscr{Hom}_{\Ox_X}(S^{n}\omega^1_{X/S},\Ox_X).$$
Then
\begin{alignat*}{2}
\mathscr{Hom}_{\Ox_X}(\Gamma^{\bullet}\omega^1_{X/S},\Ox_X) & \xrightarrow{\sim} \prod_{n\ge 0}\mathscr{Hom}_{\Ox_X}(\Gamma^n\omega^1_{X/S},\Ox_X)
\xrightarrow{\sim} \prod_{n\ge 0} S^n\calT_{X/S} \\
&\xrightarrow{\sim} \lim_{\substack{\longleftarrow \\ n\ge 1}} \bigoplus_{k<n}S^k\calT_{X/S}
\xrightarrow{\sim} \lim_{\substack{\longleftarrow \\ n\ge 0}} \left (S^{\bullet}\calT_{X/S} / \bigoplus_{k\ge n}S^k\calT_{X/S} \right ) \\
&\xrightarrow{\sim} \w{S}^{\bullet}\calT_{X/S},
\end{alignat*}
where the first arrow is induced by the canonical embeddings 
$\Gamma^n\omega^1_{X/S} \hookrightarrow \Gamma^{\bullet}\omega^1_{X/S},$
the second arrow is induced by $(s_n^{-1})_{n\ge 0},$ the third arrow sends $(t_n)\in \prod_{n\ge 0}S^n\calT_{X/S}$ to the section $(y_n)$ defined, for any $n\ge 1,$ by
$y_n=(t_0,t_1,\hdots,t_{n-1}),$
and the two last arrows are the canonical ones. The fact that this isomorphism is an isomorphism of algebras follows from (\cite{OgusVol} 5.19). The isomorphism $s'$ is constructed in a similar way and the remaining assertion follows from the construction of $s.$
\end{proof}

\begin{proposition}\label{erafprop115}
Let $\calR=\calR_{\frakX,1}/p\calR_{\frakX,1}.$ There exists a canonical isomorphism of $\Ox_X$-algebras
\begin{equation}\label{eq1181}
S^{\bullet}\omega^1_{X/S} \xrightarrow{\sim} \calR,
\end{equation}
compatible with the Hopf algebra structures. In particular, there exists a canonical isomorphism of $\Ox_X$-algebras
\begin{equation}\label{eq1182}
\calR^{\vee} \xrightarrow{\sim} \widehat{\Gamma}^{\bullet}\calT_{X/S}.
\end{equation}
In addition, under the hypothesis \ref{loccoord}, \eqref{eq1181} sends $\op{dlog}m_i$ to $\eta_{i(r)}.$
\end{proposition}

\begin{proof}
Let $\frakX \xrightarrow{\Delta} \frakU \ra \frakY$ be a factorization of the exact diagonal immersion into a closed immersion followed by an open one, $\calI$ the ideal of $\Delta$ and $g:R_{\frakX,1} \ra \frakU$ the canonical morphism. The morphism of topoi $g:\left (R_{\frakX,1} \right )_{\text{zar}} \ra \frakU_{\text{zar}}$ factors as
$$
\begin{tikzcd}
\left (R_{\frakX,1} \right )_{\text{zar}} \ar{r}{\omega} \ar{d}{g} & \frakX_{\text{zar}} \ar{dl}{\Delta} \\
\frakU_{\text{zar}} &
\end{tikzcd}
$$
The morphism
$
g^{-1}\calI \ra \Ox_{R_{\frakX,1}},\ x\mapsto \frac{x}{p}
$
induces
$\Delta^{-1}\calI \ra \omega_*\Ox_{R_{\frakX,1}} = \calR_{\frakX,1}.$
Since $\calR$ is of characteristic $p,$ the composition $\Delta^{-1}\calI \ra \calR_{\frakX,1} \ra \calR,$ where $\calR_{\frakX,1}\ra \calR$ is the canonical projection,
vanishes on $\Delta^{-1}\calI^2.$ By \eqref{isoomega1}, we get an $\Ox_X$-linear morphism
$\omega^1_{X/S} \ra \calR$
and then a morphism of $\Ox_X$-algebras
$S^{\bullet}\omega^1_{X/S} \ra \calR.$
It is an isomorphism by \eqref{isoR}. The compatibility with the Hopf algebra structures comes from \ref{HopfR}. By taking the dual and \ref{dualAkaza}, we get an isomorphism
$$\calR^{\vee} \xrightarrow{\sim} \widehat{\Gamma}^{\bullet}\calT_{X/S}.$$
\end{proof}

\begin{parag}\label{Q'recall}
Recall the monoid $Q'$ and the morphisms $F_{Q/P}:Q'\ra Q,\ (x,t)\mapsto px+\theta^{gp}(t)$ and $\pi_{Q/P}:Q\ra Q',\ x\mapsto (x,0)$ defined in \ref{PFrob}.
\end{parag}

\begin{lemma}\label{Glogflat}
Keep the notation \ref{Q'recall}. Suppose that there exists a morphism $F:(\frakX,Q) \ra (\frakX',Q')$ of framed logarithmic formal schemes, over $(\frakS,P),$ lifting the exact relative Frobenius $F_1:X \ra X'$ and such that $\frakX'\ra \frakS$ is log smooth. Then $F$ and the morphism
$$G:\frakY=\frakX\times_{\frakS,[Q]}^{\op{log}}\frakX \ra \frakX'\times_{\frakS,[Q']}^{\op{log}}\frakX'=\frakY',$$
induced by $F,$ are log flat.
\end{lemma}

\begin{proof}
Since $F_1$ is log flat \eqref{PFrob} and by (\cite{Ogus2018} IV 4.2.2), the lifting $F$ is log flat. Note that $\frakY$ and $\frakY'$ are log étale over $\frakX\times_{\frakS}^{\op{log}}\frakX$ and $\frakX'\times_{\frakS}^{\op{log}}\frakX'$ respectively \eqref{parag77}. The morphisms $\frakX \ra \frakS$ and $\frakX' \ra \frakS$ are log smooth hence so are the projections $\frakX\times_{\frakS}^{\op{log}}\frakX \ra \frakX$ and $\frakX'\times_{\frakS}^{\op{log}}\frakX' \ra \frakX'.$ Since the projections $\frakY \ra \frakX$ and $\frakY' \ra \frakX'$ are strict \eqref{parag77}, we get that the canonical morphisms $\frakY \ra \frakS$ and $\frakY' \ra \frakS$ are log smooth, hence $\frakY$ and $\frakY'$ are locally of finite presentation and log flat over $\op{Spf}W.$ Consider $X$ as a logarithmic scheme over $[Q']$ by the composition
$$X \ra [Q] \xrightarrow{[F_{Q/P}]} [Q'].$$
We have the commutative diagram
$$
\begin{tikzcd}
X \times_{S,[Q]}^{\op{log}}X \ar{r} \ar{dr} & X \times_{S,[Q']}^{\op{log}}X\ar{d} \\
 & X \times_{S}^{\op{log}}X.
\end{tikzcd}
$$
By \ref{parag77}, the morphisms $X \times_{S,[Q]}^{\op{log}}X \ra X \times_{S}^{\op{log}}X$ and $X \times_{S,[Q']}^{\op{log}}X \ra X \times_{S}^{\op{log}}X$ are log étale so $X \times_{S,[Q]}^{\op{log}}X \ra X \times_{S,[Q']}^{\op{log}}X$ is log étale. We prove that
$$H=F_1\times_{S,[Q']} F_1:X \times_{S,[Q']}^{\op{log}}X \ra X' \times_{S,[Q']}^{\op{log}}X'$$
is log flat. This morphism $F_1\times_{S,[Q']} F_1$ decomposes as
$$X\times_{S,[Q']}^{\op{log}}X \xrightarrow{\op{Id} \times F_1} X\times_{S,[Q']}^{\op{log}}X' \xrightarrow{F_1\times \op{Id}} X'\times_{S,[Q']}^{\op{log}}X'.$$
By the log flatness of $F_1$ \eqref{PFrob} and the cartesian squares
$$
\begin{tikzcd}
X\times_{S,[Q']}^{\op{log}}X \ar{r} \ar{d} & X\ar{d}{F_1} & X\times_{S,[Q']}^{\op{log}}X' \ar{r} \ar{d} & X\ar{d}{F_1} \\
X\times_{S,[Q']}^{\op{log}}X' \ar{r}{q_2} & X' & X'\times_{S,[Q']}^{\op{log}}X' \ar{r}{p_1} & X',
\end{tikzcd}
$$
where $p_1$ is the first projection and $q_2$ is the second projection, we deduce that $H:X \times_{S,[Q']}^{\op{log}}X \ra X' \times_{S,[Q']}^{\op{log}}X'$ is log flat. It follows that the special fiber $G_1$ is also log flat. Since $\frakY$ is log flat over $\op{Spf}W,$ we conclude by (\cite{Ogus2018} IV 4.2.2) that $G$ is log flat.
\end{proof}

\begin{proposition}\label{groupprop17}
Keep the notation \ref{Q'recall}. Suppose that there exists a morphism $F:(\frakX,Q) \ra (\frakX',Q')$ of framed logarithmic formal schemes, over $(\frakS,P),$ lifting the exact relative Frobenius $F_1:X \ra X'$ and such that $\frakX'\ra \frakS$ is log smooth. Let $G:\frakY \ra \frakY':=\frakX'\times_{\frakS,[Q']}^{\op{log}}\frakX'$ be the morphism induced by $F.$ Then the projection
$$\frakY\times_{\frakY'}R_{\frakX',1} \ra \frakY$$
factors through $Q_{\frakX}.$
\end{proposition}

\begin{proof}
Let $R'_1$ and $Y'$ be the special fibers of $R_{\frakX',1}$ and $\frakY'$ respectively\footnote{Not to confuse $Y'$ with the logarithmic scheme appearing in the exact relative Frobenius of $Y/S.$}.
Since $G$ is log flat \eqref{Glogflat}, the projection $\frakY\times_{\frakY'}R_{\frakX',1} \ra R_{\frakX',1}$ is also log flat. By (\cite{SF1} 11.12), $R_{\frakX',1} \ra \frakX'$ is log flat. We deduce that $\frakY\times_{\frakY'}R_{\frakX',1}$ is log flat over $\frakX'$ and then over $\op{Spf}W.$ Then, by \ref{Wflat}, $\frakY\times_{\frakY'}R_{\frakX',1}$ is flat over $\op{Spf}W.$ To apply the universal property of $Q_{\frakX}$ (\ref{erafdil2}), it remains to prove that
$$\underline{Y\times_{Y'}R'_1} \ra Y\times_{Y'}R'_1 \ra Y$$
factors through the exact diagonal immersion $X \ra Y.$ For that it is sufficient to prove that
$$Y\times_{Y'}R'_1 \ra Y \xrightarrow{F_Y} Y$$
factors through $X \ra Y.$ We may work étale locally on $X$ and suppose that the frames $X\ra [Q]$ and $X' \ra [Q']$ lift to charts $X\ra A_1[Q]$ and $X' \ra A_1[Q'].$ Let $\widetilde{Q}$ and $\widetilde{Q}'$ be the inverse images of $Q$ and $Q'$ by the addition maps
$$Q^{gp} \oplus Q^{gp} \ra Q^{gp},\ Q'^{gp} \oplus Q'^{gp} \ra Q'^{gp}$$
respectively. By (\cite{Saito04} 4.2.3 and 4.2.1), we have
$$
Y=X\times_{S,[Q]}^{\op{log}}X=\left (X\times_S^{\op{log}}X\right )\times_{A_1[Q\oplus Q]}^{\op{log}}A_1[\widetilde{Q}],
$$
$$
Y'=X'\times_{S,[Q']}^{\op{log}}X'=\left (X'\times_S^{\op{log}}X'\right )\times_{A_1[Q'\oplus Q']}^{\op{log}}A_1[\widetilde{Q}'].
$$
The absolute Frobenius $F_Y$ is then induced by the absolute Frobenius morphisms $F_X$ and $F_{A_1[\widetilde{Q}]}.$
Recall that the morphism $\pi:X' \ra X$ \eqref{diag51} is underlying a morphism of framed logarithmic schemes \eqref{PFrob} and satisfies $\pi \circ F_1=F_X.$ The morphism $\pi$ then induces a morphism $\rho:Y' \ra Y$ such that $F_Y=\rho \circ G.$ Let $\Delta':X' \ra Y'$ and $\Delta:X\ra Y$ be the exact diagonal immersions and $p_1,p_2:Y \ra X$ the canonical projections. The compositions
\begin{alignat*}{1}
X' \xrightarrow{\Delta'} Y' \xrightarrow{\rho} Y \xrightarrow{p_i} X \\
X' \xrightarrow{\pi} X \xrightarrow{\Delta} Y \xrightarrow{p_i} X
\end{alignat*}
are both equal to $\pi$ for $i\in \{1,2\}.$ This implies that $\rho \circ \Delta'=\Delta \circ \pi.$
The result then follows from the commutative diagram
$$
\begin{tikzcd}
Y\times_{Y'}R'_1 \ar{r} \ar{d} & Y \ar{r}{F_Y} \ar{d}{G} & Y \\
R'_1 \ar{r} \ar{dr} & Y' \ar[swap]{ru}{\rho} & X \ar[swap]{u}{\Delta} \\
 & X' \ar{u}{\Delta'} \ar{ur}{\pi}, &
\end{tikzcd}
$$
where the upper left square is cartesian.
\end{proof}

\begin{parag}\label{erafparag121}
Keep the assumptions of \ref{groupprop17}. Let $Q_1$ be the special fiber of $Q_{\frakX}.$ By \ref{erafprop13}, we have a canonical morphism of logarithmic schemes $\omega':R'_1 \ra X'.$ Let $\omega:Q_1 \ra X$ be the map of topological spaces coming from the groupoid structure on $Q_1.$ We have a commutative diagram
$$
\begin{tikzcd}
Q_1\ar{rr} & & Q_{\frakX}\ar{d}{g} \\
X\times_{X'}R'_1 \ar{u}{v} \ar{r}{\delta} \ar{urr}{\widetilde{v}} & \frakY \times_{\frakY'}R_{\frakX',1} \ar{ur}{\chi} \ar{r}{\alpha} \ar{d}{\beta} & \frakY \ar{d}{G} \\
 & R_{\frakX',1} \ar{r}{h} & \frakY',
\end{tikzcd}
$$
where the lower square is cartesian, $g$ and $h$ are the canonical morphisms, $\chi:\frakY \times_{\frakY'}R_{\frakX',1} \ra Q_{\frakX}$ is given in \ref{groupprop17} and $\delta$ is induced by $X \ra \frakX \xrightarrow{\Delta}\frakY,$ $X' \ra \frakX' \xrightarrow{\Delta'} \frakY'$ and $R_1' \ra R_{\frakX',1}.$
Let us provide local descriptions for $\widetilde{v}$ and $v.$ Let $p_1,p_2:\frakY \ra \frakX$ be the canonical projections.
Consider the hypothesis \ref{loccoord}. Then, by (\cite{SF1} (9.4.1)), there exists $b\in \Ox_{\frakX}$ such that
$$G^{\#}\left ( \widetilde{\eta}_i'\right ) = \left (\widetilde{\eta}_i^p+\sum_{k=1}^{p-1}\begin{pmatrix}p\\k \end{pmatrix}\widetilde{\eta}_i^k+1 \right )\frac{1+pp_2^{\#}(b)}{1+pp_1^{\#}(b)}-1.$$
Set
$
c=\frac{p_2^{\#}(b)-p_1^{\#}(b)}{1+pp_1^{\#}(b)},
$
so that
$
1+pc=\frac{1+pp_2^{\#}(b)}{1+pp_1^{\#}(b)}
$
and
$$G^{\#}\left ( \widetilde{\eta}_i'\right ) = \left (\widetilde{\eta}_i^p+\sum_{k=1}^{p-1}\begin{pmatrix}p\\k \end{pmatrix}\widetilde{\eta}_i^k+1 \right )(1+pc)-1.$$
Then
\begin{alignat}{2}\label{eq12301}
\widetilde{\eta}_i^p &= G^{\#}(\widetilde{\eta}_i')-\sum_{k=1}^{p-1}\begin{pmatrix}p\\k \end{pmatrix}\widetilde{\eta}_i^k -pc\widetilde{\eta}_i^p-pc \left (\sum_{k=1}^{p-1}\begin{pmatrix}p\\k \end{pmatrix}\widetilde{\eta}_i^k\right )-pc.
\end{alignat}
Applying $g^{\#}$ to \eqref{eq12301}, since $g^{\#}(\widetilde{\eta}_i)^p\in p\Ox_{Q_{\frakX}},$ we get $g^{\#}G^{\#}(\widetilde{\eta}_i') \in p\Ox_{Q_{\frakX}}.$ Using the flatness of $Q_{\frakX}$ over $\Z_p,$ we get
\begin{equation}\label{eq12302}
\frac{g^{\#}(\widetilde{\eta}_i)^p}{p} = \frac{g^{\#}G^{\#}(\widetilde{\eta}_i')}{p}-\sum_{k=1}^{p-1}\frac{(p-1)!}{k!(p-k)!}g^{\#}(\widetilde{\eta}_i)^k -g^{\#}(c\widetilde{\eta}_i^p)- \left (\sum_{k=1}^{p-1}\begin{pmatrix}p\\k \end{pmatrix}g^{\#}(c\widetilde{\eta}_i^k)\right )-g^{\#}(c).
\end{equation}
Since $\widetilde{\eta}_i$ and $c=\frac{p_2^{\#}(b)-p_1^{\#}(b)}{1+pp_1^{\#}(b)}$ belong to the ideal of $\Delta:\frakX \ra \frakY,$ we get
\begin{alignat}{2}
\widetilde{v}^{\#}g^{\#}(\widetilde{\eta}_i) & =\delta ^{\#}\alpha ^{\#}(\widetilde{\eta}_i)=0, \label{eq12303}\\
\widetilde{v}^{\#}g^{\#}(c) & =\delta ^{\#}\alpha ^{\#}(c)=0. \label{eq12304}
\end{alignat}
We also have
$
\chi^{\#}g^{\#}G^{\#}(\widetilde{\eta}_i') = \beta^{\#}h^{\#}(\widetilde{\eta}_i')=p\beta^{\#}\left (\frac{\widetilde{\eta}_i'}{p} \right )=p\left (1\otimes \frac{\widetilde{\eta}_i'}{p} \right ).
$
The morphism $G:\frakY \ra \frakY'$ is log flat \eqref{Glogflat} and hence so is $\beta.$ Since $R_{\frakX',1}$ is log flat over $\op{Spf}W$ \eqref{Qlogflat}, we deduce that $\frakY \times_{\frakY'}R_{\frakX',1}$ is log flat hence flat over $\op{Spf}W.$ We then get
$$
\chi^{\#}\left ( \frac{g^{\#}G^{\#}(\widetilde{\eta}_i')}{p} \right ) = \beta^{\#}\left (\frac{h^{ \#}(\widetilde{\eta}_i')}{p} \right )=1\otimes \frac{\widetilde{\eta}_i'}{p}.
$$
It follows that
\begin{equation}\label{eq12305}
\widetilde{v}^{\#}\left (\frac{g^{\#}G^{\#}(\widetilde{\eta}_i')}{p}\right )=1\otimes \eta_{i(r)}'.
\end{equation}
By \eqref{eq12302}, \eqref{eq12303}, \eqref{eq12304} and \eqref{eq12305}, dropping the notation $g^{\#},$ we get
$$\widetilde{v}^{\#}\left (\frac{\widetilde{\eta}_i^p}{p} \right )=1\otimes \eta_{i(r)}',\ \widetilde{v}^{\#}(\widetilde{\eta}_i)=0,$$
and so
\begin{equation}\label{eq1211v}
v^{\#}\left (\eta_{i(q)} \right ) = 1\otimes \eta_{i(r)}',\ v^{\#}(\eta_i)=0.
\end{equation}
This proves that $v$ is independant of the choice of the lifting $F,$ so we have a global morphism
$$v:X\times_{X'}R'_1 \ra Q_1.$$
Let $\calR'_1=\omega'_*\Ox_{R'_1}$ and $\calQ_1=\omega_*\Ox_{Q_1}$ be the respective Hopf algebras defined by $R'_1$ and $Q_1$ (\ref{paragomega}). 
Consider the cartesian square
$$
\begin{tikzcd}
X\times_{X'}R'_1 \ar{r}{\alpha'} \ar{d}{\beta'} & X \ar{d}{F_1} \\
R'_1 \ar{r}{\omega'} & X'.
\end{tikzcd}
$$
Since $F_1$ and $\omega'$ are affine, there exists a canonical isomorphism
\begin{equation}\label{eq12307}
F_1^*\calR'_1=F_1^*\omega'_*\Ox_{R'_1} \xrightarrow{\sim} \alpha'_*\Ox_{X\times_{X'}R'_1}.
\end{equation}
For a scheme $T,$ denote by $|T|$ its underlying topological space. We have a diagram of topological spaces :
$$
\begin{tikzcd}
 & & \left|Q_1\right| \ar[swap]{dl}{\omega} \ar{d} \\
\left|X\times_{X'}R'_1\right| \ar{r}{\alpha'} \ar[bend right=-30]{urr}{v} & \left|X\right| \ar{r}{\Delta} & \left|Y\right|.
\end{tikzcd}
$$
Since the outer diagram is commutative and $\Delta$ is injective as a map of sets, we get $\omega \circ v=\alpha'$ as maps of topological spaces.
Applying $\omega_*$ to $v^{\#}:\Ox_{Q_1} \ra v_*\Ox_{X\times_{X'}R'_1},$ by \eqref{eq12307}, we get a morphism that we abusively denote by $v:$
\begin{equation}\label{vHopf}
v:\calQ_1 \ra F_1^*\calR'_1.
\end{equation}
\end{parag}

\begin{proposition}
Keep the notation of \ref{erafparag121}. The morphism \eqref{vHopf} is a morphism of Hopf algebras.
\end{proposition}

\begin{proof}
We prove that $v$ is a morphism of Hopf algebras. Let $(\delta_q,\pi_q,\sigma_q)$ and $(\delta_r,\pi_r,\sigma)$ be the Hopf algebra structures of $\calQ_1$ and $\calR_1'$ respectively. We check that the diagram
\begin{equation}\label{diag12309}
\begin{tikzcd}
\calQ_1 \ar{r}{\delta_q} \ar[swap]{d}{v} & \calQ_1 \otimes_{\Ox_X} \calQ_1 \ar{d}{v\otimes v} \\
F_1^*\calR_1' \ar{r}{F_1^*\delta_r} & F_1^*\calR_1' \otimes_{\Ox_X} F_1^*\calR_1'.
\end{tikzcd}
\end{equation}
By \ref{HopfQ}, we have
\begin{alignat*}{2}
\delta_q\left (\eta_{i(q)} \right ) &= 1\otimes \eta_{i(q)}+\sum_{0<b+c<p} \frac{(-1)^{b+c}}{b+c}\begin{pmatrix}b+c \\ b \end{pmatrix} \eta_i^{b+c}\otimes \eta_i^{p-b} + \eta_{i(q)}\otimes 1, \\
\delta_q(\eta_i) &= 1\otimes \eta_i+\eta_i\otimes 1+\eta_i\otimes \eta_i.
\end{alignat*}
By \eqref{eq1211v}, we get
\begin{alignat*}{2}
(v\otimes v)\circ \delta_q(\eta_{i(q)}) &= 1\otimes F_1^*\eta_{i(r)}' + F_1^*\eta_{i(r)}' \otimes 1, \\
(v\otimes v) \circ \delta_q(\eta_i) &= 0.
\end{alignat*}
By \eqref{eq1211v} and \ref{HopfR}, we have
\begin{alignat*}{2}
(F_1^*\delta_r)\circ v(\eta_{i(q)}) &= 1\otimes F_1^*\eta_{i(r)}' + F_1^*\eta_{i(r)}'\otimes 1, \\
(F_1^*\delta_r)\circ v(\eta_{i}) &= 0.
\end{alignat*}
This proves the commutativity of \eqref{diag12309}. Similarly, we prove the commutativity of the diagrams
$$
\begin{tikzcd}
\calQ_1 \ar{r}{\sigma_q} \ar[swap]{d}{v} & \calQ_1 \ar{d}{v} & \calQ_1 \ar[swap]{d}{v} \ar{r}{\pi_q} & \Ox_X \\
F_1^*\calR_1' \ar{r}{F_1^*\sigma} & F_1^*\calR_1' & F_1^*\calR_1' \ar[swap]{ur}{F_1^*\pi_r} & 
\end{tikzcd}
$$
This proves that $v$ is a morphism of Hopf algebras.
By taking the dual of $v,$ we obtain a morphism of $\Ox_X$-algebras
\begin{equation}\label{checkv}
\check{v}:F_1^*\calR_1'^{\vee} \ra \calQ_1^{\vee}.
\end{equation}
\end{proof}

\begin{lemma}
Consider $\calD_{X/S}$ as a module over $F_1^*S^{\bullet}\calT_{X'/S}$ via the $p$-curvature map \eqref{pcurv}. Suppose that $X\ra S$ has local coordinates $m_1,\hdots,m_d \in \Gamma(X,\calM_X)$ and consider the associated differential operators $\partial_I$ (\cite{SF1} 4.4). Then the $F_1^*S^{\bullet}\calT_{X'/S}$-module $\calD_{X/S}$ has a basis $(\partial_I)_{I\in \{0,\hdots,p-1\}^d}.$
\end{lemma}

\begin{proof}
First note that, by (\cite{SF1} (4.5.7)), for any integers $n,r\ge 0,$
\begin{alignat*}{2}
\partial_{(pn+r)\epsilon_i} &= \prod_{k=0}^{pn+r-1}(\partial_{\epsilon_i}-k) &= (\partial_{p\epsilon_i})^n \circ \partial_{r\epsilon_i}.
\end{alignat*}
We have a basis $(\partial_I)_{I\in \N^d}$ for the $\Ox_X$-module $\calD_{X/S}.$ Let $I=(I_1,\hdots,I_d)\in \N^d$ and $I_i=pq_i+r_i,\ 0\le r_i<p.$ Then
$$\partial_I=\prod_{i=1}^d\partial_{I_i\epsilon_i}=\prod_{i=1}^d (\partial_{p\epsilon_i})^{q_i}\circ \partial_{r_i\epsilon_i}.$$
We deduce that the family $(\partial_I)_{I\in \{0,\hdots,p-1\}^d}$ is generating.

Now let $(b_I)_{I\in \{0,\hdots,p-1 \}^d}$ be local sections of $F_1^*S^{\bullet}\calT_{X'/S}$ such that
$$\sum_{i\in \{0,\hdots,p-1 \}^d }b_I\partial_I=0.$$
By \eqref{eqDoma1}, for any $I\in \{0,\hdots,p-1\}^d,$ there exists a finite subset $\Gamma_I\subset \N^d$ and local sections $a_{I,J}$ of $\Ox_X$ for $J\in \Gamma_I$ such that $b_I=\sum_{J\in \Gamma_I}a_{I,J}\partial_{pJ}.$ Hence
$$
\sum_{I\in \{0,\hdots,p-1 \}^d,J\in \Gamma_I}a_{I,J}\partial_{pJ}\circ \partial_I=0.
$$
By (\cite{SF1} (4.5.8)),
$\partial_{pJ}\circ \partial_I=\partial_{pJ+I}.$
We deduce that
$$\sum_{I\in \{0,\hdots,p-1 \}^d,J\in \Gamma_I}a_{I,J}\partial_{pJ+I}=0.$$
Since $pK+I \neq pL+J$ for $K,L\in \N^d$ and $I,J\in \{0,\hdots,p-1 \}^d$ such that $I\neq J,$ it follows that for any $I\in \{0,\hdots,p-1 \}^d$ and $J\in \Gamma_I,$ $a_{I,J}=0.$
\end{proof}

\begin{proposition}\label{Kokoprop1234}
The morphism
$$
\Xi:\calD_{X/S} \otimes_{S^{\bullet}\calT_{X'/S}}\widehat{S}^{\bullet}\calT_{X'/S}\ra \widehat{\calD}_{X/S},
$$
induced by the canonical morphism $\calD_{X/S} \ra \w{\calD}_{X/S}$ and the morphism $\w{\psi}:\widehat{S}^{\bullet}\calT_{X'/S} \ra \widehat{\calD}_{X/S}$ \eqref{eq13191}, is an isomorphism of $\Ox_X$-algebras.
\end{proposition}

\begin{proof}
We may work étale locally on $X$ and hence suppose that we have local coordinates $m_1,\hdots,m_d$ $\in \calM_X.$ Then $\pi^{\flat}m_1,\hdots,\pi^{\flat}m_d\in \calM_{X'}$ (where $\pi:X'\ra X$ is defined in \ref{diag51}) are local coordinates for $X'/S.$ For every multi-index $I\in\N^d,$ denote by $\partial_I$ the differential operator corresponding to $(m_i)_{1\le i\le d},$ as defined in (\cite{SF1} 4.4). Let $(\partial_i')_{1\le i\le d}$ be the dual basis of $(\op{dlog}\pi^{\flat}m_i)_{1\le i\le d}.$ A local section $x$ of $\calD_{X/S} \otimes_{S^{\bullet}\calT_{X'/S}}\widehat{S}^{\bullet}\calT_{X'/S} = \calD_{X/S} \otimes_{F_1^*S^{\bullet}\calT_{X'/S}}F_1^*\widehat{S}^{\bullet}\calT_{X'/S}$ is uniquely written as an infinite sum
$$
x=\sum_{\substack{\alpha \in \llbracket0,p-1\rrbracket^d \\ \beta\in \N^d}} c_{\alpha\beta} \partial_{\alpha}\otimes \partial'^{\beta},\ c_{\alpha \beta}\in \Ox_X,
$$
where
$
\partial'^{\beta}=\prod_{i=1}^d\partial_i'^{\beta_i} \in S^{\bullet}\calT_{X'/S}.
$
By (\cite{SF1} (4.5.8)), the morphism $\Xi$ sends $x$ to
$$
\sum_{\substack{\alpha \in \llbracket0,p-1\rrbracket^d \\ \beta\in \N^d}} c_{\alpha\beta} \partial_{\alpha}\circ \partial_{p\beta}=
\sum_{\substack{\alpha \in \llbracket0,p-1\rrbracket^d \\ \beta\in \N^d}} c_{\alpha\beta} \partial_{\alpha+p\beta}.
$$
This finishes the proof.
\end{proof}
\begin{proposition}\label{prop1232}
Consider the morphisms $\check{u}:\widehat{\calD}_{X/S} \ra \calQ_1^{\vee}$ \eqref{checku}, $\check{v}:F_1^*\calR_1'^{\vee} \ra \calQ_1^{\vee}$ \eqref{checkv}, the antipode morphism $\sigma:\calR'_1 \ra \calR'_1$ \eqref{HopfR} and its dual $\check{\sigma}:\calR_1'^{\vee} \ra \calR_1'^{\vee}.$ Recall that we identify the small étale sites of $X$ and $X'$ via $F_1:X\ra X'.$
Consider the canonical morphisms $\calR_1'^{\vee}\ra F_1^*\calR_1'^{\vee},\ x\mapsto 1\otimes x$ and $\widehat{S}^{\bullet}\calT_{X'/S} \ra \widehat{\Gamma}^{\bullet}\calT_{X'/S}$ and the morphism $\widehat{\psi}:\widehat{S}^{\bullet}\calT_{X'/S} \ra \widehat{\calD}_{X/S}$ \eqref{eq13191}.
We identify $\widehat{\Gamma}^{\bullet}\calT_{X'/S}$ with $\calR_1'^{\vee}$ via the isomorphism $\xi:\widehat{\Gamma}^{\bullet}\calT_{X'/S} \xrightarrow{\sim} \calR_1'^{\vee}$ \eqref{eq1182}.
Then, the diagram
$$
\begin{tikzcd}
\widehat{S}^{\bullet}\calT_{X'/S} \ar{d} \ar{rrrr}{\widehat{\psi}} & & & & \widehat{\calD}_{X/S} \ar{d}{\check{u}} \\
\widehat{\Gamma}^{\bullet}\calT_{X'/S} \ar{r}{\xi} & \calR_1'^{\vee} \ar{r}{\check{\sigma}} & \calR_1'^{\vee} \ar{r} & F_1^*\calR_1'^{\vee} \ar{r}{\check{v}} & \calQ_1^{\vee}
\end{tikzcd}
$$
is commutative.
\end{proposition}

\begin{proof}
We may suppose that the hypothesis \ref{loccoord} is satisfied. We keep the same notations of \ref{erafparag121}. Let $(\partial_1',\hdots,\partial_d')\in \left (\calT_{X'/S} \right )^d$ be the dual basis of $(\op{dlog}m_1',\hdots,\op{dlog}m_d').$ For $1\le i\le d,$ denote by $\epsilon_i$ the multi-index of $\N^d$ whose all coefficients are zero except for the $i$th which is equal to $1.$
By \ref{loccoord}, $\left (\eta^{[I]}\right )_{I\in \N^d}$ is a basis for the $\Ox_X$-module $\calP_0.$ Denote by $\left (\partial_I \right )\in \widehat{\calD}_{X/S}$ its dual basis. By definition and \eqref{eqDoma1}, $\check{u}(\partial_i')=\partial_{p\epsilon_i}\circ u.$ Then, by \eqref{ueta}, we have
\begin{alignat*}{2}
\left (\check{u}(\partial_i')\right )(\eta_j) &= \partial_{p\epsilon_i}(\eta_j)=0.\\
\left (\check{u}(\partial_i')\right )(\eta_{j(q)}) &= \partial_{p\epsilon_i}\left (-\eta_j^{[p]}\right )=-\delta_{ij}.
\end{alignat*}
By \ref{locdescRQ}, $F_1^*\calR_1' =\Ox_X\left [\eta_{1(r)}',\hdots,\eta_{d(r)}'\right ].$ Let $(\varphi_I)_{I\in \N^d}$ be the dual basis of $(\eta_{(r)}'^I)_{I\in \N^d}.$
By \ref{erafprop115}, the local section $\partial_i'\in F_1^*\widehat{\Gamma}^{\bullet}\calT_{X'/S}$ corresponds, by $F_1^*\xi,$ to $\varphi_{\epsilon_i} \in F_1^*\calR_1'^{\vee}.$ It follows that $(\check{v}\circ\check{\sigma})(\partial_i')=\varphi_{\epsilon_i}\circ (F_1^*\sigma) \circ v.$ Then, by \eqref{eq1211v},
\begin{alignat*}{2}
\left (\check{v}\circ \check{\sigma}(\partial_i')\right )(\eta_j) &= \varphi_{\epsilon_i}\circ (F_1^*\sigma)\circ v(\eta_j)=0.\\
\left (\check{v}\circ \check{\sigma}(\partial_i')\right )(\eta_{j(q)}) &= \varphi_{\epsilon_i}\circ (F_1^*\sigma)\circ v\left (\eta_{j(q)}\right )=\varphi_{\epsilon_i}(-\eta_{j(r)}')=-\delta_{ij}.
\end{alignat*}
This proves that
$$\check{u}(\partial_i')=\check{v}\circ \check{\sigma}(\partial_i'),$$
and hence the proposition.
\end{proof}

\begin{proposition}
Keep the notation of \ref{prop1232}.
\begin{enumerate}
\item For any $\partial \in \widehat{\calD}_{X/S}$ and $D'\in \widehat{\Gamma}^{\bullet}\calT_{X'/S},$ we have the following equality in $\calQ_1^{\vee}:$
$$\check{u}(\partial)\cdot (\check{v}\circ \check{\sigma})(D')=(\check{v}\circ \check{\sigma})(D')\cdot \check{u}(\partial),$$
where the algebra structure on $\calQ_1^{\vee}$ comes from the Hopf algebra structure on $\calQ_1.$
\item The map
\begin{equation}\label{isoDQ}
\xi:\begin{array}[t]{clc}
\widehat{\calD}_{X/S} \otimes_{\widehat{S}^{\bullet}\calT_{X'/S}} \widehat{\Gamma}^{\bullet}\calT_{X'/S} & \ra & \calQ_1^{\vee} \\
x\otimes y & \mapsto & \check{u}(x) \cdot \check{v}(y),
\end{array}
\end{equation}
is a well-defined isomorphism of $\Ox_X$-algebras.
\end{enumerate}
\end{proposition}

\begin{proof}
For the first assertion, it is enough to prove the equality for $\partial=\partial_{\alpha}$ and $D'=\partial'^{[\beta]}:=\prod_{i=1}^d\partial_i'^{[\beta_i]},$ for $\alpha,\beta\in \N^d.$
Just as in the proof of \ref{prop1232}, we have
$$
(\check{v} \circ \check{\sigma})(\partial'^{[\beta]})=\varphi_{\beta} \circ (F_1^*\sigma) \circ v
$$
and
$$
\check{u}(\partial_{\alpha})=\partial_{\alpha}\circ u.
$$
By definition of the ring structure on $\calQ_1^{\vee},$ the product $\check{u}(\partial_{\alpha}) \cdot (\check{v} \circ \check{\sigma})(\partial'^{[\beta]})$ is then equal to the composition
\begin{equation*}
\calQ_1 \xrightarrow{\delta} \calQ_1\otimes_{\Ox_X}\calQ_1 \xrightarrow{\op{Id}\otimes v} \calQ_1\otimes_{\Ox_X}F_1^*\calR_1' \xrightarrow{\op{Id}\otimes F_1^*\sigma} \calQ_1\otimes_{\Ox_X}F_1^*\calR_1' \xrightarrow{\op{Id}\otimes \varphi_{\beta}} \calQ_1 \xrightarrow{u} \calP_0 \xrightarrow{\partial_{\alpha}} \Ox_X,
\end{equation*}
where $\delta$ is given in \ref{HopfQ}.
This is equal to the composition
\begin{equation*}
\calQ_1 \xrightarrow{\delta} \calQ_1\otimes_{\Ox_X}\calQ_1 \xrightarrow{u\otimes v} \calP_0\otimes_{\Ox_X} F_1^*\calR_1' \xrightarrow{\partial_{\alpha}\otimes (\varphi_{\beta}\circ F_1^*\sigma)}  \Ox_X.
\end{equation*}
For $J,K\in \N^d,$ set
$$
\eta_{(q)}^J=\prod_{j=1}^d\eta_{j(q)}^{J_j} \in \calQ_1,\ 
\eta^K=\prod_{k=1}^d\eta_k^{K_k} \in \calQ_1.
$$
By \eqref{isoQ}, the $\Ox_X$-module $\calQ_1$ has a basis
$$\left (\eta^K\eta_{(q)}^J\right )_{K\in \llbracket 0,p-1 \rrbracket^d,\ J\in \N^d}.$$
Let $K\in \llbracket 0,p-1 \rrbracket^d$ and $J\in \N^d.$ By \ref{HopfQ},
\begin{equation}\label{12332}
\begin{alignedat}{2}
\delta(\eta^K\eta_{(q)}^J) =& \prod_{1\le k\le d} \delta(\eta_k)^{K_k} \prod_{1\le j\le d}\delta(\eta_{j(q)})^{J_j} \\
=& \prod_{1\le k\le d} (1\otimes\eta_k+\eta_k\otimes 1+\eta_k\otimes\eta_k)^{K_k} \\
&\times \prod_{1\le j\le d} \left (1\otimes \eta_{j(q)}+\sum_{0<b+c<p} \frac{(-1)^{b+c}}{b+c}\begin{pmatrix}b+c \\ b \end{pmatrix} \eta_j^{b+c}\otimes \eta_j^{p-b} + \eta_{j(q)}\otimes 1 \right )^{J_j}.
\end{alignedat}
\end{equation}
Recall the notation (\cite{SF1} 2.2). Developing the right-hand side of \eqref{12332} with the binomial formula and by \eqref{eq1211v} and \eqref{ueta}, we get
\begin{alignat*}{2}
(u\otimes v)\circ \delta(\eta^K\eta_{(q)}^J) =& \left (u \left (\eta^K \right )\otimes 1\right ) \left (\sum_{L\le J} \begin{pmatrix}J\\L \end{pmatrix} u \left (\eta_{(q)}^L\right )\otimes v\left (\eta_{(q)}^{J-L} \right ) \right ) \\
=& \left (\eta^K\otimes 1\right )\left (\sum_{L\le J} \begin{pmatrix} J\\ L \end{pmatrix} \left (-\eta^{[p]} \right )^L\otimes \eta_{(r)}'^{J-L}\right ) \\
=& \sum_{L\le J} \begin{pmatrix} J\\ L \end{pmatrix} \left (\eta^K\left (-\eta^{[p]} \right )^L\right )\otimes \eta_{(r)}'^{J-L}.
\end{alignat*}
We then get
$$
(\op{Id}_{\calP_0}\otimes F_1^*\sigma) \circ (u\otimes v)\circ \delta(\eta^K\eta_{(q)}^J) = (-1)^{|J|}\sum_{L\le J} \begin{pmatrix} J\\ L \end{pmatrix} \left (\eta^K\left (\eta^{[p]} \right )^L\right )\otimes \eta_{(r)}'^{J-L},
$$
and so
\begin{equation}
\left ((\check{v} \circ \check{\sigma})(\partial'^{[\beta]}) \cdot \check{u}(\partial_{\alpha}) \right ) (\eta^K\eta_{(q)}^J) = (-1)^{|J|} \begin{pmatrix} J \\ \beta\end{pmatrix} \partial_{\alpha}\left (\eta^K\left (\eta^{[p]} \right )^{J-\beta} \right ).
\end{equation}
By basic PD structure properties,
$$\left (\eta^{[p]} \right )^{J-\beta}=\frac{(p(J-\beta))!}{p!^{J-\beta}}\eta^{[p(J-\beta)]}.$$
Note that, for any integer $n\ge 0,$
$$v_p\left ( (pn)! \right )=\sum_{k\ge 1} \left \lfloor \frac{pn}{p^k} \right \rfloor \ge n,$$
so $\frac{(p(J-\beta))!}{p!^{J-\beta}}\in \N.$
It follows that
\begin{equation}\label{1222v}
 (\check{v} \circ \check{\sigma})(\partial'^{[\beta]}) \cdot \left ( \check{u}(\partial_{\alpha}) \right ) (\eta^K\eta_{(q)}^J) = (-1)^{|J|} \begin{pmatrix} J \\ \beta\end{pmatrix} \frac{(K+p(J-\beta))!}{p!^{J-\beta}} \partial_{\alpha}\left (\eta^{[K+p(J-\beta)]} \right ).
\end{equation}
The product $\left ((\check{v} \circ \check{\sigma})(\partial'^{[\beta]})\right ) \cdot  \check{u}(\partial_{\alpha})$ is equal to the composition
$$
\calQ_1 \xrightarrow{\delta} \calQ_1\otimes_{\Ox_X}\calQ_1 \xrightarrow{\op{Id} \otimes u} \calQ_1\otimes_{\Ox_X}\calP_0 \xrightarrow{\op{Id} \otimes \partial_{\alpha}} \calQ_1 \xrightarrow{v} F_1^*\calR_1' \xrightarrow{F_1^*\sigma} F_1^*\calR_1' \xrightarrow{\varphi_{\beta}} \Ox_X.
$$
This is equal to the composition
\begin{equation*}
\calQ_1 \xrightarrow{\delta} \calQ_1\otimes_{\Ox_X}\calQ_1 \xrightarrow{v\otimes u} \calP_0\otimes_{\Ox_X} F_1^*\calR_1' \xrightarrow{(\varphi_{\beta}\circ F_1^*\sigma)\otimes \partial_{\alpha}}  \Ox_X.
\end{equation*}
A similar computation for $\left ((\check{v} \circ \check{\sigma})(\partial'^{[\beta]}) \right ) \circ \check{u}(\partial_{\alpha})(\eta^K\eta_{(q)}^J)$ proves that
$$(\check{v} \circ \check{\sigma})(\partial'^{[\beta]}) \cdot \check{u}(\partial_{\alpha})=\check{u}(\partial_{\alpha}) \cdot (\check{v} \circ \check{\sigma})(\partial'^{[\beta]}).$$
This concludes the proof of the first assertion, which in turn proves that $\xi$ \eqref{isoDQ} is well-defined. It remains to prove that it is an isomorphism. A local section of $\widehat{\calD}_{X/S} \otimes_{\widehat{S}^{\bullet}\calT_{X'/S}} \widehat{\Gamma}^{\bullet}\calT_{X'/S}$ can be uniquely written as an infinite sum
$$
\sum_{\substack{\alpha \in \llbracket0,p-1\rrbracket^d \\ \beta\in \N^d}} c_{\alpha\beta} \partial_{\alpha}\otimes \partial'^{[\beta]},\ c_{\alpha \beta}\in \Ox_X.
$$
For $\alpha\in \llbracket 0,p-1 \rrbracket^d$ and $\beta,J,K\in \N^d,$ we have, by \eqref{1222v},
\begin{equation}\label{12315}
\xi(\partial_{\alpha}\otimes \partial'^{[\beta]})(\eta^K\eta_{(q)}^J)=(-1)^{|\beta|}\alpha! \delta_{\alpha K}\delta_{\beta J}.
\end{equation}
This finishes the proof.
\end{proof}

\begin{corollaire}
Let
\begin{equation}\label{eq12331bis}
\calD^{\gamma}_{X/S}=\calD_{X/S}\otimes_{S^{\bullet} \calT_{X'/S}} \widehat{\Gamma}^{\bullet}\calT_{X'/S}.
\end{equation}
There exists an isomorphism of $\Ox_X$-algebras
\begin{equation}\label{eq12331}
\calD^{\gamma}_{X/S} \xrightarrow{\sim} \calQ_1^{\vee}.
\end{equation}
In addition, under the assumption \ref{loccoord}, the local section $(-1)^{|J|} I! \partial_I \otimes \partial'^{[J]}$ of $\calD_{X/S}^{\gamma}$ corresponds to $\varphi_{I,J},$ for $I\in \llbracket 0,p-1 \rrbracket^d$ and $J\in \N^d.$
\end{corollaire}

\begin{proof}
This results from \eqref{isoDQ}, \ref{Kokoprop1234} and \eqref{12315}.
\end{proof}

\begin{definition}\label{locPDnil}
Let $\mathfrak{I}$ be the PD-ideal of $\Gamma^{\bullet}\calT_{X/S}.$
\begin{enumerate}
\item A $\widehat{\Gamma}^{\bullet}\calT_{X/S}$-module $\calE$ is said to be \emph{locally PD-nilpotent} if any local section $x$ of $\calE$ is locally annihilated by $\mathfrak{I}^{[m]}$ for some positive integer $m.$
\item A $\calD_{X/S}^{\gamma}$-module $\calE$ is said to be \emph{locally PD-nilpotent} if it is so when considered as a $\w{\Gamma}^{\bullet}\calT_{X/S}$-module via the canonical morphism
$$\w{\Gamma}^{\bullet}\calT_{X/S} \ra \calD_{X/S}^{\gamma}.$$
\end{enumerate}
\end{definition}

\begin{theorem}\label{thm1237}
Let $\calR_1=\calR_{\frakX,1}/p\calR_{\frakX,1},$ $\calQ_1=\calQ_{\frakX}/p\calQ_{\frakX}$ and
$$\calD^{\gamma}_{X/S}=\calD_{X/S}\otimes_{S^{\bullet} \calT_{X'/S}} \widehat{\Gamma}^{\bullet}\calT_{X'/S}.$$
The following tensor categories are canonically equivalent:
\begin{enumerate}
\item The category of $\Ox_X$-modules equipped with a stratification relative to $\calR_1$ (resp. $\calQ_1$).
\item The category of locally PD-nilpotent $\widehat{\Gamma}^{\bullet}\calT_{X/S}$-modules (resp. locally PD-nilpotent $\calD^{\gamma}_{X/S}$-modules).
\end{enumerate}
\end{theorem}

\begin{proof}
The proof being similar to (\cite{Oyama} 1.2.10), we just outline the idea of the equivalences. 

Let $\calE$ be an $\Ox_X$-module equipped with an $\calR_1$-stratification
$$\epsilon:\calR_1\otimes_{\Ox_X}\calE \ra \calE \otimes_{\Ox_X}\calR_1.$$
Since we have an isomorphism of $\Ox_X$-algebras \eqref{eq1182}
$$\calR^{\vee}_1 \xrightarrow{\sim} \widehat{\Gamma}^{\bullet}\calT_{X/S},$$
it is sufficient to define a $\calR^{\vee}_1$-module structure on $\calE.$
This structure is obtained as follows: for local sections $x$ and $\varphi$ of $\calE$ and $\calR^{\vee}_1$ respectively, the action $\varphi \cdot x$ of $\varphi$ on $x$ is the image of $x$ by the composition
$$\begin{array}[t]{clclc}
\calE & \ra & \calE\otimes_{\Ox_X}\calR_1 & \xrightarrow{\op{Id}_{\calE}\otimes \varphi }  & \calE \\
x & \mapsto & \epsilon(1\otimes x) & &
\end{array}
$$
We check that the $\widehat{\Gamma}^{\bullet}\calT_{X/S}$-module structure obtained is PD-nilpotent by reducing to local coordinates: suppoe that the hypothesis \ref{loccoord} is satisfied. We have a basis $(\op{dlog}m_i)_{1\le i\le d}$ of $\omega^1_{X/S}.$ Let $(D_i)_{1\le i\le d}$ be its dual basis and consider $D_i$ as a section of $\widehat{\Gamma}^{\bullet}\calT_{X/S}$ via the canonical morphism $\calT_{X/S} \ra \widehat{\Gamma}^{\bullet}\calT_{X/S}.$ For $I=(I_1,\hdots,I_d)\in \N^d,$ set
$$\eta_{(r)}^I=\prod_{i=1}^d\eta_{i(r)}^{I_i} \in \calR_1.$$
By \ref{HopfR}, the family $\left (\eta_{(r)}^I\right )_{I\in \N^d}$ is a basis for the $\Ox_X$-module $\calR_1.$ Denote by $\left (\varphi_I \right )_{I\in \N^d}$ its dual base. By \ref{erafprop115}, the section $D_i\in \widehat{\Gamma}^{\bullet}\calT_{X/S}$ corresponds to $\varphi_{\epsilon_i} \in \calR_1^{\vee},$ where $\epsilon_i\in \N^d$ is such that all its coefficients are zero except for the $i$th which is equal to $1.$ Let $x$ be a local section of $\calE.$ There exists a positive integer $n$ and local sections $x_I\in \calE$ such that
$$
\epsilon(1\otimes x)=\sum_{\substack{I\in \N^d \\ |I|\le n}} x_I\otimes \eta_{(r)}^I \in \calE \otimes_{\Ox_X}\calR_1.
$$
It follows that, for $I\in \N^d$ such that $|I|>n,$
$$
\varphi_I\cdot x=0.
$$
Conversly, let $\calE$ be a locally PD-nilpotent $\widehat{\Gamma}^{\bullet}\calT_{X/S}$-module and
$$\alpha:\calR^{\vee}_1\otimes_{\Ox_X}\calE \xrightarrow{\sim} \widehat{\Gamma}^{\bullet}\calT_{X/S}\otimes_{\Ox_X}\calE \ra \calE$$
the structure morphism.
Consider the morphism
$$
\Lambda:\mathscr{Hom}_{\Ox_X}\left (\calE , \calE\otimes_{\Ox_X}\calR_1 \right ) \ra \mathscr{Hom}_{\Ox_X}\left (\calR^{\vee}_1\otimes_{\Ox_X}\calE , \calE \right )
$$
sending a morphism $f:\calE \ra \calE \otimes_{\Ox_X}\calR_1$ to
$$
\begin{array}[t]{clc}
\calR^{\vee}_1 \otimes_{\Ox_X} \calE & \ra & \calE \\
\varphi \otimes x & \mapsto & (\op{Id}_{\calE} \otimes \varphi)\circ f(x).
\end{array}
$$
The $\Ox_X$-module $\calR_1$ is locally free so $\Lambda$ is injective. Suppose that there exists an étale covering $(U_i \ra X)_{i\in I}$ and morphisms $\theta_i:\calE_{|U_i} \ra \calE_{|U_i}\otimes_{\Ox_{U_i}}\calR_{1|U_i}$ such that $\Lambda(\theta_i)=\alpha_{|U_i}$ for all $i\in I.$ The injectivity of $\Lambda$ implies that the morphisms $\theta_i$ glue together into a morphism $\theta:\calE \ra \calE\otimes_{\Ox_X}\calR_1$ such that $\Lambda(\theta)=\alpha.$ It is hence sufficient to construct, étale locally on $X,$ a morphism
$$\theta:\calE \ra \calE \otimes_{\Ox_X}\calR_1$$
sent, by $\Lambda,$ to $\alpha.$ We can thus suppose that the hypothesis \ref{loccoord} is satisfied. We have a basis $\left (\eta_{(r)}^I\right )_{I\in \N^d}$ of the $\Ox_X$-module $\calR_1.$ Let $(\alpha_I)_{i\in \N^d}$ be the dual basis. We set
$$
\theta:\begin{array}[t]{clc}
\calE & \ra & \calE\otimes_{\Ox_X}\calR_1 \\
x & \mapsto & \sum_{I\in \N^d}(\alpha_I\cdot x)\otimes \eta_{(r)}^I.
\end{array}
$$
Note that the sum is finite since $\calE$ is locally PD-nilpotent.

Let us now construct the second equivalence. Let $\calE$ be an $\Ox_X$-module equipped with a $\calQ_1$-stratification $\epsilon.$ By \eqref{eq12331}, it is sufficient to define a $\calQ_1^{\vee}$-module structure on $\calE.$ We define the action $\varphi\cdot x$ of $\varphi \in \calQ_1^{\vee}$ on $x\in \calE$ as the image of $x$ by the composition
$$
\begin{array}[t]{clclc}
\calE & \ra & \calE\otimes_{\Ox_X}\calQ_1 & \xrightarrow{\op{Id}_{\calE}\otimes \varphi }  & \calE \\
x & \mapsto & \epsilon(1\otimes x) & &
\end{array}
$$
To check that this $\calD_{X/S}^{\gamma}$-module structure is PD-nilpotent, we suppose that the hypothesis \ref{loccoord} is satisfied. For $I=(I_1,\hdots,I_d)\in \N^d,$ set
$$\eta_{(q)}^I=\prod_{i=1}^d\eta_{i(q)}^{I_i} \in \calQ_1,\ \eta^I=\prod_{i=1}^d\eta_{i}^{I_i} \in \calQ_1.$$
Then $\left (\eta^I\eta_{(q)}^J\right )_{\substack{I\in \llbracket 0,p-1 \rrbracket^d \\ J\in \N^d}}$ is a basis for the $\Ox_X$-module $\calQ_1.$ Denote by $\left (\varphi_{I,J}\right )$ its dual basis. Let $(\partial'_1,\hdots,\partial'_d) \in \left (\calT_{X'/S}\right )^d$ be the dual basis of $(\op{dlog}m_1',\hdots,\op{dlog}m_d').$ For $\beta\in  \N^d,$ set
$$
\partial'^{[\beta]}:=\prod_{i=1}^d\partial_i'^{[\beta_i]} \in \Gamma^{\bullet}\calT_{X'/S}.
$$
By \eqref{12315}, the section $\partial_I\otimes \partial'^{[\beta]}\in \calD_{X/S}^{\gamma}$ corresponds, by \eqref{eq12331}, to $(-1)^{|\beta|}I!\varphi_{I,\beta}\in \calQ_1^{\vee}.$
For $x\in \calE,$ there exist $x_{I,J}\in \calE$ such that
$$
\epsilon(1 \otimes x)=\sum_{\substack{I\in \llbracket 0,p-1 \rrbracket^d \\ J\in \N^d}} x_{I,J}\otimes \eta^I\eta_{(q)}^J
$$
and $x_{I,J}=0$ except for finitely many multi-indices $I$ and $J.$
It follows that $\varphi_{I,\beta}\cdot x=0$ for sufficiently large $|\beta|.$ The rest of the equivalence is similar to the previous case.
\end{proof}

\section{Logarithmic Oyama topoi and the Cartier transform}

Let $\frakS$ be a logarithmic formal scheme, log flat and locally of finite type over $\op{Spf}W,$ and $X\ra S$ a log smooth morphism of logarithmic schemes.

\begin{definition}\label{defforintro}
We define the categories $\calE(X/\frakS)$ and $\underline{\calE}(X/\frakS)$ as follows:
\begin{enumerate}
\item An object of $\calE(X/\frakS)$ (resp. $\underline{\calE}(X/\frakS)$) is a triple $(U,\frakT,u)$ consisting of an étale strict morphism of logarithmic schemes $U \ra X,$ a log flat fs $p$-adic formal logarithmic $\frakS$-scheme $\frakT$ (\cite{SF1} 7.18) and an $S$-morphism of logarithmic schemes $u:T \ra U$ (resp. $u:\underline{T}\ra U$) which is affine as a morphism of schemes, where $\underline{T}$ is defined in \ref{era3logimagedef}.
\item A morphism $(U_1,\frakT_1,u_1) \ra (U_2,\frakT_2,u_2)$ in $\calE(X/\frakS)$ (resp. $\underline{\calE}(X/\frakS)$) is a pair $(f,g)$ consisting of an $\frakS$-morphism $f:\frakT_1 \ra \frakT_2$ and an $X$-morphism $g:U_1 \ra U_2$ such that $u_2\circ f_1=g\circ u_1$ (resp. $u_2 \circ \underline{f_1}=g\circ u_1$), where $f_1:T_1 \ra T_2$ is the morphism induced by $f$ by reduction modulo $p$ and $\underline{f_1}:\underline{T_1} \ra \underline{T_2}$ is obtained from $f_1$ by functoriality (\ref{era3functoriality}).
\end{enumerate}
A morphism $(U_1,\frakT_1,u_1) \ra (U_2,\frakT_2,u_2)$ in $\calE(X/\frakS)$ or $\underline{\calE}(X/\frakS)$ is said to be \emph{log flat} if $\frakT_1 \ra \frakT_2$ is log flat.
\end{definition}

\begin{definition}
We denote by $\calE_{\mathrm{strict}}(X/\frakS)$ (resp. $\underline{\calE}_{\mathrm{strict}}(X/\frakS)$) the subcategory of $\calE(X/\frakS)$ (resp. $\underline{\calE}(X/\frakS)$) consisting of objects $(U,\frakT,u)$ such that $u:T\ra U$ (resp. $u:\underline{T} \ra U$) is strict.
\end{definition}

\begin{parag}\label{era3fiberparag}
Fiber products by log flat morphisms in the categories $\calE(X/\frakS),$ $\underline{\calE}(X/\frakS),$ $\calE_{\mathrm{strict}}(X/\frakS)$ and $\underline{\calE}_{\mathrm{strict}}(X/\frakS)$ are representable. More precisely, given morphisms $(U_1,\frakT_1,u_1) \ra (U,\frakT,u)$ and $(U_2,\frakT_2,u_2) \ra (U,\frakT,u)$ in $\calE(X/\frakS)$ (resp. $\underline{\calE}(X/\frakS)$) with one of them being log flat, the fiber product $(U_1,\frakT_1,u_1)\times_{(U,\frakT,u)}(U_2,\frakT_2,u_2)$ is representable by $(U_1\times_UU_2,\frakT_1\times_{\frakT}^{\op{log}}\frakT_2,v)$ where $v$ is the morphism $T_1\times_T^{\op{log}}T_2 \ra U_1\times_UU_2$ induced by $u_1$ and $u_2$ (resp. is the composition of the mophism $\underline{T_1}\times_{\underline{T}}^{\op{log}}\underline{T_2} \ra U_1\times_UU_2$ induced by $u_1$ and $u_2$ with the canonical morphism $\underline{T_1\times_T^{\op{log}}T_2} \ra \underline{T_1}\times_{\underline{T}}^{\op{log}}\underline{T_2}$ (\ref{era3functoriality})), which is affine. Indeed, since one of the morphisms $\frakT_1 \ra \frakT$ and $\frakT_2 \ra \frakT$ is log flat, $\frakT_1\times_{\frakT}^{\op{log}}\frakT_2 \ra \frakS$ is also log flat (\cite{Ogus2018} IV 4.1.2 (3)). Note that $U_1\times_UU_2=U_1\times_U^{\op{log}}U_2$ since $U\ra X,$ $U_1\ra X$ and $U_2 \ra X$ are all strict. In addition, if $u_1$ and $u_2$ are strict, then so is $v.$
\end{parag}

\begin{definition}\label{era3defcart}~ 
\begin{enumerate}
\item A morphism $(U_1,\frakT_1,u_1) \ra (U_2,\frakT_2,u_2)$ in $\calE(X/\frakS)$ is said to be \emph{cartesian} if the canonical morphism
$$T_1 \ra T_2\times_{U_2}U_1$$
is an isomorphism.
\item A morphism $(U_1,\frakT_1,u_1) \ra (U_2,\frakT_2,u_2)$ in $\underline{\calE}(X/\frakS)$ is said to be \emph{cartesian} if the morphism $\frakT_1 \ra \frakT_2$ is strict and étale and the square
$$
\begin{tikzcd}
T_1 \ar{r} \ar[swap]{d}{\Delta_{u_1}} & T_2 \ar{d}{\Delta_{u_2}} \\
U_1' \ar{r} & U_2',
\end{tikzcd}
$$
where the horizontal arrows are induced by $\frakT_1 \ra \frakT_2$ and $U_1 \ra U_2$ and the vertical arrows are given in \eqref{morphismShiho1}, is cartesian.
\end{enumerate}
\end{definition}

\begin{proposition}\label{era3stablebc}
The class of cartesian morphisms in $\calE(X/\frakS)$ (resp. $\underline{\calE}(X/\frakS),$ resp. $\calE_{\mathrm{strict}}(X/\frakS),$ resp. $\underline{\calE}_{\mathrm{strict}}(X/\frakS)$) is stable under composition and base change by any morphism.
\end{proposition}

\begin{proof}
Stability under composition in all categories is clear and so is stability by base change for $\calE(X/\frakS)$ and $\calE_{\mathrm{strict}}(X/\frakS).$
Let $(U_1,\frakT_1,u_1) \ra (U,\frakT,u)$ and $(U_2,\frakT_2,u_2) \ra (U,\frakT,u)$ be morphisms in $\underline{\calE}(X/\frakS)$ such that $(U_1,\frakT_1,u_1) \ra (U,\frakT,u)$ is cartesian. By \ref{era3fiberparag}, the fiber product is given by
$$(U_1,\frakT_1,u_1)\times_{(U,\frakT,u)}(U_2,\frakT_2,u_2)=(V,\frakZ,v),$$
where $V=U_1\times_UU_2,$ $\frakZ=\frakT_1\times_{\frakT}^{\op{log}}\frakT_2$ and $v:\underline{T_1\times_T^{\op{log}}T_2}\ra U_1\times_UU_2$ is the composition of the canonical morphism $\underline{T_1\times^{\op{log}}_TT_2} \ra \underline{T_1}\times_{\underline{T}}^{\op{log}}\underline{T_2}$ with the morphism $\underline{T_1}\times_{\underline{T}}^{\op{log}}\underline{T_2} \ra U_1\times_UU_2$ induced by $u_1$ and $u_2.$ Considering the following commutative diagram
$$
\begin{tikzcd}
 & T_1\times_T^{\op{log}}T_2 \ar{rr} \ar[dashed]{ddd}{\Delta_{v}} \ar{dl} & & T_2 \ar{dl} \ar{ddd}{\Delta_{u_2}} \\
T_1 \ar{rr} \ar[swap]{ddd}{\Delta_{u_1}} & & T \ar[swap]{ddd}{\Delta_{u}} & \\
 &  & &  \\
 & U_1'\times_{U'} U_2' \ar[dashed]{rr} \ar[dashed]{dl} & & U_2' \ar{dl} \\
 U_1' \ar{rr} & & U', &
\end{tikzcd}
$$
we get that the diagram
$$
\begin{tikzcd}
T_1\times_T^{\op{log}}T_2 \ar{r} \ar[swap]{d}{\Delta_{v}} & T_2 \ar{d}{\Delta_{u_2}} \\
U_1'\times_{U'}U_2' \ar{r} & U_2'
\end{tikzcd}
$$
is cartesian, hence the stability for $\underline{\calE}(X/\frakS)$. The stability for $\underline{\calE}_{\mathrm{strict}}(X/\frakS)$ follows immediately.
\end{proof}

\begin{parag}\label{era3paralpha}
Let $(U,\frakT,u)$ be an object of $\calE(X/\frakS)$ (resp. $\underline{\calE}(X/\frakS)$). We define a functor
\begin{equation}\label{eqKoko1861}
\alpha_{(U,\frakT,u)}:\begin{array}[t]{cll}
\text{ét}_{/U} & \ra & \calE(X/\frakS)\ (\op{resp.} \underline{\calE}(X/\frakS))\\ 
V & \mapsto & (V,\frakT_V,u_V)
\end{array}
\end{equation}
as follows: let $V$ be an étale $U$-scheme. We equip $V$ with the logarithmic structure pullback of that of $U,$ so that $V\ra U$ becomes a strict morphism. First, we construct $\alpha_{(U,\frakT,u)}$ for $\calE(X/\frakS).$ Let $T_V=T\times_UV$ and $u_v:T_V \ra V$ the canonical projection, which is affine as a base change of the affine morphism $u:T \ra U.$ Since $V\ra U$ is an étale strict morphism, so is $T_V \ra T.$ Then there exists a unique $p$-adic formal scheme $\frakT_V$ that fits into a cartesian square of formal schemes
$$
\begin{tikzcd}
T_V \ar{r} \ar{d} & \frakT_V \ar{d} \\
T \ar{r} & \frakT
\end{tikzcd}
$$
We equip $\frakT_V$ with the logarithmic structure pullback of that of $\frakT.$ This makes the previous square cartesian in the category of logarithmic formal schemes. In addition, $\frakT_V \ra \frakT$ is étale and strict so $\frakT_V$ is log flat over $\frakS.$ We get an object $(V,\frakT_V,u_V)$ of $\calE(X/\frakS)$ and we set
\begin{equation}\label{eqKoko1862}
\alpha_{(U,\frakT,u)}(V)=(V,\frakT_V,u_V).
\end{equation}
We now prove the result for $\underline{\calE}(X/\frakS).$ Let $T_V=T\times_{U'}V',$ where $T \ra U'$ is the morphism $\Delta_{u}$ \eqref{morphismShiho1}. Since $V' \ra U'$ is strict, $T_V\ra T$ is also strict and so $T_V$ is fs. The canonical projection $T_V\ra T$ being étale, we construct $\frakT_V$ the same way we did in the previous case. 
Let $T \xrightarrow{G_T} \underline{T} \ra T$ and $T_V \xrightarrow{G_{T_V}} \underline{T_V} \ra T_V$ be the factorizations of $F_T$ and $F_{T_V}$ through the canonical immersions $\underline{T} \ra T$ and $\underline{T_V} \ra T_V.$ The morphism $T_V \ra \underline{T_V}$ is inseparable. Consider the morphism $T \ra U''$ defined in \ref{Shihoparag1}. By \ref{propShiho185} and \eqref{morphismShiho2}, the diagram
$$
\begin{tikzcd}
T \ar{r}{\Delta_{u}} \ar[swap]{dd}{G_T}  \ar{dr} & U' \ar{d} \\
 & U'' \ar{d} \\
\underline{T} \ar{r}{u} & U
\end{tikzcd}
$$ 
is commutative. In addition, so are the diagrams
$$
\begin{tikzcd}
V' \ar{r} \ar{d} & V \ar{d} & T_V \ar{r} \ar{d} & \underline{T_V} \ar{d} \\
U' \ar{r} & U & T \ar{r} & \underline{T}.
\end{tikzcd}
$$
It follows that the decomposition $T_V \xrightarrow{G_{T_V}} \underline{T_V} \ra \underline{T} \xrightarrow{u} U$ is equal to $T_V \ra T \xrightarrow{G_T} \underline{T} \xrightarrow{u} U.$ This is then equal to $T_V \ra T \xrightarrow{\Delta_u} U' \ra U$ and then to $T_V \ra V' \ra U' \ra U$ and finally to $T_V \ra V' \ra V \ra U.$
It follows that the diagram
$$
\begin{tikzcd}
V' \ar{rrr} & & & V \ar{d} \\
T_V \ar{u} \ar{r}{G_{T_V}}  & \underline{T_V} \ar[dashed]{urr}{u_V} \ar{r} & \underline{T} \ar{r}{u} & U
\end{tikzcd}
$$
is commutative. The existence of the dashed arrow $u_V$ follows then from the fact that $V \ra U$ is étale and strict, the fact that $T_V \ra \underline{T_V}$ is inseparable and (\cite{Ogus2018} IV 3.3.7).
If $U$ and $V$ are affine then so is $u_V.$ Note that, if $u$ is strict, then so is $u_V.$
We set
\begin{equation}\label{eqKoko1863}
\alpha_{(U,\frakT,u)}(V)=(V,\frakT_V,u_V).
\end{equation}
If $V_1 \ra V_2$ is a morphism of étale and strict $U$-schemes then we have cartesian squares
$$
\begin{tikzcd}
T_{V_1} \ar{r} \ar{d} & T_{V_2} \ar{r} \ar{d} & T \ar{d}{\Delta_u} \\
V_1' \ar{r} & V_2' \ar{r} & U'.
\end{tikzcd}
$$
By the étaleness and strictness of $\frakT_{V_2} \ra \frakT,$ we have a unique morphism $\frakT_{V_1} \ra \frakT_{V_2}$ fitting into the commutative diagram
$$
\begin{tikzcd}
T_{V_2} \ar{rr} & & \frakT_{V_2} \ar{d} \\
T_{V_1} \ar{u} \ar{r} & \frakT_{V_1} \ar{r} \ar{ur} & \frakT.
\end{tikzcd}
$$
By construction of $u_{V_1},$ we have a commutative diagram
$$
\begin{tikzcd}
V_1' \ar{rrr} & &  & V_1 \ar{r} \ar{d} & V_2 \ar{dl} \\
T_{V_1} \ar{u} \ar{r}{G_{T_{V_1}}} & \underline{T_{V_1}} \ar{r} \ar{urr}{u_{V_1}} & \underline{T} \ar{r}{u} & U. &
\end{tikzcd}
$$
We also have the commutative diagram
$$
\begin{tikzcd}
 & V_2' \ar{rrr} & & & V_2 \ar{dd} \\
V_1' \ar{ur} & T_{V_2} \ar{u} \ar{r} & \underline{T_{V_2}} \ar{urr}{u_{V_2}} \ar{dr} & &  \\
 & T_{V_1} \ar{ul} \ar{u} \ar{r} & \underline{T_{V_1}} \ar{u} \ar{r} & \underline{T} \ar{r}{u} & U.
\end{tikzcd}
$$
The morphism $V_2 \ra U$ is étale and strict so the diagram
$$
\begin{tikzcd}
\underline{T_{V_1}} \ar[swap]{d}{u_{V_1}} \ar{r} & \underline{T_{V_2}} \ar{d}{u_{V_2}} \\
V_1 \ar{r} & V_2
\end{tikzcd}
$$
is commutative.
We hence get a morphism
$$
\alpha_{(U,\frakT,u)}(V_1)=\left ( V_1,\frakT_{V_1},u_{V_1} \right ) \ra \left ( V_2,\frakT_{V_2},u_{V_2} \right )=\alpha_{(U,\frakT,u)}(V_2).
$$
This defines the functor $\alpha_{(U,\frakT,u)}.$
\end{parag}

\begin{remark}\label{remShiho1}~ 
\begin{enumerate}
\item If $(U,\frakT,u)$ is an object of $\calE(X/\frakS)$ (resp. $\underline{\calE}(X/\frakS)$) and if $V$ is an étale and strict $U$-scheme and $(V,\frakT_V,u_V)=\alpha_{(U,\frakT,u)}(V),$ then the construction of $u_V$ shows that $u_V$ is strict if $u$ is strict. It follows that, if $(U,\frakT,u)$ is an object of $\calE_{\mathrm{strict}}(X/\frakS)$ (resp. $\underline{\calE}_{\mathrm{strict}}(X/\frakS)$), then $\alpha_{(U,\frakT,u)}$ factors through $\calE_{\mathrm{strict}}(X/\frakS)$ (resp. $\underline{\calE}_{\mathrm{strict}}(X/\frakS)$).
\item Suppose that $\frakS=\op{Spf}W$ equipped with the trivial log structure. Let $(U,\frakT,u)$ be an object of $\underline{\calE}(X/\frakS).$ Here, we provide a simpler definition of the functor
$$
\alpha_{(U,\frakT,u)}:\text{ét}_{/U} \ra \underline{\calE}(X/\frakS).
$$
Recall that, in this case, $\Delta_{u}=u' \circ f_{T/S}.$ Let $V$ be an étale $U$-scheme. Set $T_V=T\times_{U'}V',$ where $T \ra U'$ is the composition $u' \circ f_{T/S}.$ Since $V' \ra U'$ is strict, $T_V$ is fs. The canonical projection $T_V\ra T$ being étale, we construct $\frakT_V$ the same way we did in \ref{era3paralpha}. The morphism $f_{T_V/S}$ is inseparable \eqref{propfT/S} and $V' \ra U'$ is étale and strict. It follows, by (\cite{Ogus2018} IV 3.3.7), that there exists a unique morphism $v:\underline{T_V}' \ra V'$ fitting into the following commutative diagram
$$
\begin{tikzcd}
T_V \ar[swap, bend right=70]{dd} \ar{r} \ar[swap]{d}{f_{T_V/S}} & T  \ar{d}{f_{T/S}} \\
 \underline{T_V}' \ar{r} \ar[dashed,swap]{d}{v} & \underline{T}' \ar{d}{u'} \\
V' \ar{r} & U'.
\end{tikzcd}
$$
If $V$ is affine then $T_V=V'\times_{U',u'\circ f_{T/S}}T$ is affine and so so is $\underline{T_V}'.$ It follows that $v$ is affine.
There exists a unique morphism of logarithmic schemes $u_V:\underline{T_V} \ra V$ such that $v=u_V',$ which is affine since $v$ is affine.
We set
\begin{equation}
\alpha_{(U,\frakT,u)}(V)=(V,\frakT_V,u_V).
\end{equation}
\end{enumerate}
\end{remark}

\begin{proposition}\label{era3paralphaKoko}
Let $(U,\frakT,u)$ be an object of $\calE(X/\frakS)$ (resp. $\underline{\calE}(X/\frakS)$). The functor $\alpha_{(U,\frakT,u)}$ \eqref{eqKoko1861} satisfies the following properties
\begin{enumerate}
\item For every étale $U$-scheme $V,$ the canonical morphism $\alpha_{(U,\frakT,u)}(V) \ra (U,\frakT,u)$ is cartesian.
\item For every étale $U$-scheme $V$ and any morphism $a:(W,\frakZ,w) \ra \alpha_{(U,\frakT,u)}(V)$ in $\calE(X/\frakS)$ (resp. $\underline{\calE}(X/\frakS)$), there exists a unique morphism $b:(W,\frakZ,w) \ra \alpha_{(U,\frakT,u)}(W)$ such that the diagram
$$
\begin{tikzcd}
(W,\frakZ,w) \ar{r}{a} \ar{d}{b} & \alpha_{(U,\frakT,u)}(V) \\
\alpha_{(U,\frakT,u)}(W) \ar{ur} & 
\end{tikzcd}
$$
is commutative. In addition, $a$ is cartesian if and only if $b$ is cartesian.
\item For every $U$-morphism of étale $U$-schemes $V_1 \ra V_2,$ the morphism $\alpha_{(U,\frakT,u)}(V_1) \ra \alpha_{(U,\frakT,u)}(V_2)$ is cartesian.
\end{enumerate}
\end{proposition}

\begin{proof}
The first assertion is clear from the construction of $\alpha_{(U,\frakT,u)}.$ The existence of $b$ in the second assertion follows, in the case $\underline{\calE}(X/\frakS),$ from the following fact: let $(V,\frakT_V,u_V)=\alpha_{(U,\frakT,u)}(V)$ and $(W,\frakT_W,u_W)=\alpha_{(U,\frakT,u)}(W).$ By definition, we have cartesian squares
\begin{equation}\label{diagKoko1871}
\begin{tikzcd}
Z \ar[swap]{dr}{\Delta_{(W,\frakZ,w)}} \ar{r} \ar[bend right=-30]{rrr} & T_W \ar{r} \ar{d} & T_V \ar{d} \ar{r} & T \ar{d}{\Delta_{u}} \\
 & W' \ar{r} & V' \ar{r} & U',
\end{tikzcd}
\end{equation}
$$
\begin{tikzcd}
Z\ar{d} \ar[bend right=30]{ddd} \ar{r} & \frakZ \ar{d} \\
T_W \ar{r} \ar{d} & \frakT_W \ar{d} \\
T_V \ar{r} \ar{d} & \frakT_V \ar{d} \\
T \ar{r} & \frakT.
\end{tikzcd}
$$
The case $\calE(X/\frakS)$ is simpler.
The commutativity of \eqref{diagKoko1871} proves that $a$ is cartesian if and only if $b$ is cartesian. The third assertion follows from the first.
\end{proof}

\begin{parag}\label{era3parbeta}
Let $(f,g):(U_1,\frakT_1,u_1) \ra (U_2,\frakT_2,u_2)$ be a morphism in $\calE(X/\frakS)$ (resp. $\underline{\calE}(X/\frakS)$) and
\begin{equation}\label{jg}
j_g:\begin{array}[t]{clc}
\text{ét}_{/U_1} & \ra & \text{ét}_{/U_2} \\
(V\ra U_1) & \mapsto & (V \ra U_1 \xrightarrow[g]{} U_2).
\end{array}
\end{equation}
The morphism $(f,g)$ induces a morphism
\begin{equation}\label{defbeta115}
\beta_{(f,g)}:\alpha_{(U_1,\frakT_1,u_1)} \ra \alpha_{(U_2,\frakT_2,u_2)}\circ j_g
\end{equation}
as follows (we will only give the construction for $\underline{\calE}(X/\frakS)$): let $V$ be an étale $U_1$-scheme and $(V,\frakT_{i_V},v_i)=\alpha_{(U_i,\frakT_i,u_i)}(V).$
The morphism $T_1 \ra T_2$ induces
$$T_{1_V}=T_1\times_{U_1',\Delta_{u_1}}V' \ra T_2 \times_{U_2',\Delta_{u_2}}V'=T_{2_V}$$
fitting into the commutative diagram
$$
\begin{tikzcd}
T_{1_V} \ar{rr} \ar{dd} \ar{dr} & & T_{2_V} \ar{dr} \ar[dashed]{dd} & \\
 & T_1 \ar{rr}\ar{dd} & & T_2\ar{dd} \\
V' \ar{dr} \ar[dashed, equal]{rr} & & V' \ar{dr} & \\
 & U_1' \ar{rr} & & U_2'.
\end{tikzcd}
$$
This morphism extends, by étaleness and strictness of $\frakT_{2_V} \ra \frakT_2,$ to a morphism of $p$-adic logarithmic formal schemes $h:\frakT_{1_V} \ra \frakT_{2_V}$ over $\frakT_2.$
We hence obtain a morphism $(h,\op{Id}_V):(V,\frakT_{1_V},v_1) \ra (V,\frakT_{2_V},v_2)$ over $(f,g).$
We easily prove that the morphisms $\beta_{(f,g)}$ satisfy the following properties:
\begin{enumerate}
\item For every object $(U,\frakT,u)$ of $\calE(X/\frakS)$ (resp. $\underline{\calE}(X/\frakS)$),
$$\beta_{\op{Id}_{(U,\frakT,u)}}=\op{Id}_{\alpha_{(U,\frakT,u)}}.$$
\item For all composable morphisms $(f_1,g_1):(U_1,\frakT_1,u_1) \ra (U_2,\frakT_2,u_2)$ and $(f_2,g_2):(U_2,\frakT_2,u_2) \ra (U_3,\frakT_3,u_3)$ in $\calE(X/\frakS)$ (resp. $\underline{\calE}(X/\frakS)$), the following diagram
$$
\begin{tikzcd}
\alpha_{(U_1,\frakT_1,u_1)} \ar{r}{\beta_{(f_1,g_1)}} \ar[swap]{dr}{\beta_{(f_2,g_2) \circ (f_1,g_1)}} & \alpha_{(U_2,\frakT_2,u_2)} \circ j_{g_1} \ar{d}{j_{g_1}^*\beta_{(f_2,g_2)}} \\
& \alpha_{(U_3,\frakT_3,u_3)} \circ j_{g_2\circ g_1},
\end{tikzcd}
$$
where $j_{g_1}^*\beta_{(f_2,g_2)}$ is the morphism induced by $\beta_{(f_2,g_2)},$ is commutative up to a canonical isomorphism
\item If $(f,g)$ is a cartesian morphism, then $\beta_{(f,g)}$ is an isomorphism.
\end{enumerate}
We just check the third point: if $(f,g)$ is a cartesian morphism in $\underline{\calE}(X/\frakS)$ then
$$T_{1_V}=T_1\times_{U_1'}V'=T_2\times_{U_2'}V'=T_{2_V}.$$
It follows that $\beta_{(f,g)}$ is an isomorphism. This is also true in the category $\calE(X/\frakS).$
\end{parag}

\begin{proposition}\label{propalphacomfiber}
Let $(U,\frakT,u)$ be an object of $\calE(X/\frakS)$ (resp. $\underline{\calE}(X/\frakS)$). The functor $\alpha_{(U,\frakT,u)}$ \eqref{eqKoko1861} commutes with fiber products.
\end{proposition}

\begin{proof}
For any étale $U$-scheme $V,$ we set $(V,\frakT_V,u_V)=\alpha_{(U,\frakT,u)}(V).$ 
Let $V_1 \ra V$ and $V_2\ra V$ be étale morphisms of $U$-schemes. We want to prove that
$$\alpha_{(U,\frakT,u)}(V_1\times_VV_2)=\alpha_{(U,\frakT,u)}(V_1)\times_{\alpha_{(U,\frakT,u)}(V)}\alpha_{(U,\frakT,u)}(V_2).$$
Since
$$(V,\frakT_V,u_V)=\alpha_{(V,\frakT_V,u_V)}(V),$$
we can suppose that $U=V$ and hence $(V,\frakT_V,u_V)=(U,\frakT,u).$
We prove the result for $\underline{\calE}(X/\frakS)$ since the other case is simpler.
Consider the following commutative diagram with cartesian squares
$$
\begin{tikzcd}
T_{V_1\times_UV_2} \ar{rr} \ar{dd} \ar{dr} & & T_{V_2} \ar{dr} \ar[dashed]{dd}& \\
 & T_{V_1} \ar{rr} \ar{dd} & & T\ar{dd}{\Delta_u} \\
V_1'\times_{U'}V_2' \ar[dashed]{rr} \ar{dr} & & V_2' \ar[dashed]{dr} & \\
 & V_1' \ar{rr} & & U'.
\end{tikzcd}
$$
It follows that
\begin{equation}\label{17181}
T_{V_1\times_UV_2}=T_{V_1}\times_T^{\op{log}}T_{V_2}
\end{equation}
and so, by definition of $\alpha_{(U,\frakT,u)},$
$$\frakT_{V_1\times_UV_2}=\frakT_{V_1}\times_{\frakT}^{\op{log}}\frakT_{V_2}.$$
By definition, $u_{V_1\times_UV_2}$ is the unique morphism fitting in the following commutative diagram
\begin{equation}\label{diagS1}
\begin{tikzcd}
\left ( V_1\times_UV_2 \right )' \ar{rrr} & & & V_1\times_{U}V_2\ar{d} \\
T_{V_1\times_UV_2} \ar{u}{\Delta_{\left ( V_1\times_UV_2,T_{V_1\times_UV_2},u_{V_1\times_UV_2}\right )}}  \ar{r} & \underline{T_{V_1\times_UV_2}} \ar{urr}{u_{V_1\times_UV_2}} \ar{r} & \underline{T} \ar{r}{u} & U,
\end{tikzcd}
\end{equation}
where $T_{V_1\times_UV_2} \ra \underline{T_{V_1\times_UV_2}}$ is obtained from the factorization of $F_{T_{V_1\times_UV_2}}$ through $\underline{T_{V_1\times_UV_2}}.$ Similarly, $u_{V_i}$ fits into the commutative diagram
$$
\begin{tikzcd}
V'_i \ar{rrr} & & & V_i \ar{d} \\
T_{V_i} \ar{u} \ar{r}  & \underline{T_{V_i}} \ar[dashed]{urr}{u_{V_i}} \ar{r} & \underline{T} \ar{r}{u} & U.
\end{tikzcd}
$$
By \eqref{17181},
$$\underline{T_{V_1\times_UV_2}}=\underline{T_{V_1}\times_T^{\op{log}}T_{V_2}}.$$
We deduce that $u_{V_1} \times u_{V_2}$ fits also into the commutative diagram \eqref{diagS1}. Hence
$$
u_{V_1\times_UV_2}=u_{V_1} \times u_{V_2}.
$$
This finishes the proof.
\end{proof}

\begin{proposition}\label{era3propdescentdata}
A presheaf $\F$ on $\calE(X/\frakS)$ (resp. $\underline{\calE}(X/\frakS)$) is equivalent to the following data:
\begin{enumerate}
\item For every object $(U,\frakT,u)$ of $\calE(X/\frakS)$ (resp. $\underline{\calE}(X/\frakS)$), a presheaf $\F_{(U,\frakT,u)}$ on $\text{ét}_{/U}.$
\item For every morphism $(f,g):(U_1,\frakT_1,u_1) \ra (U_2,\frakT_2,u_2)$ in $\calE(X/\frakS)$ (resp. $\underline{\calE}(X/\frakS)$), a morphism of presheaves on $\text{ét}_{/U_1},$
$$\gamma_{\F,(f,g)}:g^{-1}\F_{(U_2,\frakT_2,u_2)} \ra \F_{(U_1,\frakT_1,u_1)},$$
\end{enumerate}
such that
\begin{enumerate}[(i)]
\item For every object $(U,\frakT,u)$ of $\calE(X/\frakS)$ (resp. $\underline{\calE}(X/\frakS)$),
$$\gamma_{\F,\op{Id}_{(U,\frakT,u)}}=\op{Id}_{\F_{(U,\frakT,u)}}.$$
\item For all composable morphisms $(f_1,g_1):(U_1,\frakT_1,u_1) \ra (U_2,\frakT_2,u_2)$ and $(f_2,g_2):(U_2,\frakT_2,u_2) \ra (U_3,\frakT_3,u_3)$ in $\calE(X/\frakS)$ (resp. $\underline{\calE}(X/\frakS)$),
$$\gamma_{\F,(f_1,g_1)} \circ g_1^{-1}\gamma_{\F,(f_2,g_2)}=\gamma_{\F,(f_2,g_2)\circ (f_1,g_1)}.$$
\item If $(f,g)$ is a cartesian morphism, then $\gamma_{\F,(f,g)}$ is an isomorphism.
\end{enumerate}
This equivalence is given as follows: for a presheaf $\F$ on $\calE(X/\frakS)$ (resp. $\underline{\calE}(X/\frakS)$),
$$\F_{(U,\frakT,u)}=\F \circ \alpha_{(U,\frakT,u)}$$
and $\gamma_{\F,(f,g)}$ is induced by $\beta_{(f,g)}$ \eqref{defbeta115}. Conversely, given a data $(\F_{(U,\frakT,u)},\gamma_{\F,(f,g)})$ as above, the presheaf is defined by
$$\F(U,\frakT,u)=\F_{(U,\frakT,u)}(U)$$
and, for a morphism $(f,g):(U_1,\frakT_1,u_1) \ra (U_2,\frakT_2,u_2),$ the morphism
$$\F(U_2,\frakT_2,u_2) \ra \F(U_1,\frakT_1,u_1)$$
is equal to $\gamma_{\F,(f,g)}(U_1).$
\end{proposition}

\begin{proof}
Suppose we are given the data $(\F_{(U,\frakT,u)},\gamma_{\F,(f,g)}).$ Conditions $(i)$ and $(ii)$ imply that the correspondance
$$(U,\frakT,u) \mapsto \F_{(U,\frakT,u)}(U)$$
is a presheaf.
Conversely, if $\F$ is a presheaf on $\calE(X/\frakS)$ (resp. $\underline{\calE}(X/\frakS)$) then, for any object $(U,\frakT,u)$ of $\calE(X/\frakS)$ (resp. $\underline{\calE}(X/\frakS)$), $\F_{(U,\frakT,u)}=\F \circ \alpha_{(U,\frakT,u)}$ is clearly a presheaf on $\text{ét}_{/U}.$ Conditons $(i),$ $(ii)$ and $(iii)$ are satisfied by the fact that $\F$ is a presheaf and \ref{era3parbeta}.

We check that the constructions are inverse to each other: let $\F$ be a presheaf on $\calE(X/\frakS)$ and $\calG$ the presheaf corresponding to $\left (\F_{(U,\frakT,u)}\right ).$ For any object $(U,\frakT,u)$ of $\calE(X/\frakS),$ we have
$$\calG(U,\frakT,u)=\F_{(U,\frakT,u)}(U)=\F(U,\frakT,u)$$
so $\calG=\F.$

Conversely, let $(\calG_{(U,\frakT,u)})$ be data satisfying conditions $(i),$ $(ii)$ and $(iii)$ and let $\F$ be the corresponding presheaf. Let $(U,\frakT,u)$ be an object of $\calE(X/\frakS),$ $V$ an étale $U$-scheme and $(V,\frakT_V,u_V)=\alpha_{(U,\frakT,u)}(V).$ Then
$$\F_{(U,\frakT,u)}(V)=\F(V,\frakT_V,u_V)=\calG_{(V,\frakT_V,u_V)}(V).$$
By condition $(iii),$ we have
$$\calG_{(V,\frakT_V,u_V)}(V)=\calG_{(U,\frakT,u)}(V)$$
so $\F_{(U,\frakT,u)}=\calG_{(U,\frakT,u)}.$
\end{proof}

\begin{remark}\label{rem1813}
Proposition \ref{era3propdescentdata} remains true if we replace $\calE(X/\frakS)$ and $\underline{\calE}(X/\frakS)$ with $\calE_{\mathrm{strict}}(X/\frakS)$ and $\underline{\calE}_{\mathrm{strict}}(X/\frakS)$ respectively.
\end{remark}

\begin{definition}\label{era3defdescentdata}
For a presheaf of sets $\F$ on $\calE(X/\frakS)$ (resp. $\underline{\calE}(X/\frakS),$ resp. $\calE_{\mathrm{strict}}(X/\frakS),$ resp. $\underline{\calE}_{\mathrm{strict}}(X/\frakS)$), we call \emph{descent data associated with $\F$} the data $(\F_{(U,\frakT,u)},\gamma_{\F,(f,g)})$ given in \ref{era3propdescentdata} and \ref{rem1813}.
\end{definition}

\begin{parag}\label{parettop}
For any object $(U,\frakT,u)$ of $\calE(X/\frakS)$ (resp. $\underline{\calE}(X/\frakS),$ resp. $\calE_{\mathrm{strict}}(X/\frakS),$ resp. $\underline{\calE}_{\mathrm{strict}}(X/\frakS)$), we denote by $\op{Cov}((U,\frakT,u))$ the collection of families of cartesian morphisms \eqref{era3defcart} $((U_i\frakT_i,u_i) \ra (U,\frakT,u))$ such that $(U_i\ra U)$ is an étale covering of $U.$ By \ref{era3stablebc}, these collections define a pretopology on $\calE(X/\frakS)$ (resp. $\underline{\calE}(X/\frakS),$ $\calE_{\mathrm{strict}}(X/\frakS)$ and $\underline{\calE}_{\mathrm{strict}}(X/\frakS)$). We call the associated topology the \emph{étale topology of $\calE(X/\frakS)$} (resp. $\underline{\calE}(X/\frakS),$ resp. $\calE_{\mathrm{strict}}(X/\frakS),$ resp. $\underline{\calE}_{\mathrm{strict}}(X/\frakS)$) and we denote by $\widetilde{\calE}(X/\frakS)$ (resp. $\widetilde{\underline{\calE}}(X/\frakS),$ resp. $\widetilde{\calE}_{\mathrm{strict}}(X/\frakS),$ resp. $\widetilde{\underline{\calE}}_{\mathrm{strict}}(X/\frakS)$) the category of sheaves of sets on $\calE(X/\frakS)$ (resp. $\underline{\calE}(X/\frakS),$ resp. $\calE_{\mathrm{strict}}(X/\frakS),$ resp. $\underline{\calE}_{\mathrm{strict}}(X/\frakS)$).
\end{parag}

\begin{proposition}\label{era4prop1723}
Let $\F$ be a presheaf of sets on $\calE(X/\frakS)$ (resp. $\underline{\calE}(X/\frakS)$) and
$$(\F_{(U,\frakT,u)},\gamma_{\F,(f,g)})$$
the associated descent data \eqref{era3defdescentdata}. Then $\F$ is a sheaf for the étale topology \eqref{parettop} if and only if $\F_{(U,\frakT,u)}$ is a sheaf of $U_{\text{ét}}$ for every object $(U,\frakT,u)$ of $\calE(X/\frakS)$ (resp. $\underline{\calE}(X/\frakS)$).
\end{proposition}

\begin{proof}
Suppose that $\F$ is a sheaf for the étale topology and let $(U,\frakT,u)$ be an object of $\calE(X/\frakS)$ (resp. $\underline{\calE}(X/\frakS)$) and $(V_i\ra V)_{i\in I}$ an étale covering of $U.$ We have seen that the functor $\alpha_{(U,\frakT,u)}$ \eqref{eqKoko1861} sends morphisms to cartesian morphisms. It follows that
$$\alpha_{(U,\frakT,u)}(V_i) \ra \alpha_{(U,\frakT,u)}(V)$$
is a covering for the étale topology and so the sequence
$$0 \ra \F \circ \alpha_{(U,\frakT,u)}(V) \ra \prod_{i\in I}\F \circ \alpha_{(U,\frakT,u)}(V_i) \rightrightarrows \prod_{i,j\in I} \F \circ \alpha_{(U,\frakT,u)}(V_i\times_VV_j)$$
is exact. By \ref{propalphacomfiber}, we get the exactness of the sequence
$$0 \ra \F_{(U,\frakT,u)}(V) \ra \prod_{i\in I}\F_{(U,\frakT,u)}(V_i) \rightrightarrows \prod_{i,j\in I} \F_{(U,\frakT,u)}(V_i\times_VV_j),$$
and so $F_{(U,\frakT,u)}$ is a sheaf.

Conversely, suppose every $\F_{(U,\frakT,u)}$ is a sheaf and let $((U_i,\frakT_i,u_i) \xrightarrow{(f_i,g_i)} (U,\frakT,u))_{i\in I}$ be a covering for the étale topology and $(U_{ij},\frakT_{ij},u_{ij})=(U_{i},\frakT_{i},u_{i})\times_{(U,\frakT,u)}(U_{j},\frakT_{j},u_{j}).$ Since $\F_{(U,\frakT,u)}$ is a sheaf, we get the exact sequence
$$0 \ra \F_{(U,\frakT,u)}(U) \ra \prod_{i\in I} \F_{(U,\frakT,u)}(U_i) \rightrightarrows \prod_{i,j\in I} \F_{(U,\frakT,u)}(U_i\times_UU_j).$$
By \ref{era3propdescentdata} (iii), we get the exact sequence
$$0 \ra \F_{(U,\frakT,u)}(U) \ra \prod_{i\in I} \F_{(U_i,\frakT_i,u_i)}(U_i) \rightrightarrows \prod_{i,j\in I} \F_{(U_{ij},\frakT_{ij},u_{ij})}(U_i\times_UU_j),$$
which is equal to
$$0 \ra \F(U,\frakT,u) \ra \prod_{i\in I} \F(U_i,\frakT_i,u_i) \rightrightarrows \prod_{i,j\in I} \F(U_{ij},\frakT_{ij},u_{ij}).$$
\end{proof}

\begin{remark}
Proposition \ref{era4prop1723} remains true if we replace $\calE(X/\frakS)$ and $\underline{\calE}(X/\frakS)$ with $\calE_{\mathrm{strict}}(X/\frakS)$ and $\underline{\calE}_{\mathrm{strict}}(X/\frakS)$ respectively.
\end{remark}

\begin{proposition}\label{propalphacocont}
Let $(U,\frakT,u)$ be an object of $\calE(X/\frakS)$ (resp. $\underline{\calE}(X/\frakS)$) and $V$ an étale $U$-scheme. Consider the functor $\alpha_{(U,\frakT,u)}$ \eqref{eqKoko1861} and set $\left (V,\frakT_V,u_V\right )=\alpha_{(U,\frakT,u)}(V).$ For any covering family $\left ((U_i,\frakT_i,u_i) \xrightarrow{(f_i,g_i)} \left (V,\frakT_V,u_V \right ) \right )_{i\in I}$ for the étale topology \eqref{parettop}, there exists an étale covering $\left (U_i \ra V\right )_{i\in I},$ sent by $\alpha_{(U,\frakT,u)}$ to $(f_i,g_i).$
\end{proposition}

\begin{proof}
We first prove the proposition for $\calE(X/\frakS).$ For any $i\in I,$ the morphism $(U_i,\frakT_i,u_i) \ra \left (V,\frakT_V,u_V\right )$ is cartesian \eqref{era3defcart}. It follows that the left square in the diagram
$$
\begin{tikzcd}
T_i \ar{r} \ar[swap]{d}{u_i} & T_V \ar{r} \ar{d}{u_V} & T\ar{d}{u} \\
U_i \ar{r} & V \ar{r} & U
\end{tikzcd}
$$
is cartesian. The right square is cartesian by definition.
Set
$$
\left (U_i,\frakT_{U_i},u_{U_i}\right )=\alpha_{(U,\frakT,u)}(U_i).
$$
By definition, the square
$$
\begin{tikzcd}
T_{U_i} \ar{r} \ar[swap]{d}{u_{U_i}} & T \ar{d}{u} \\
U_i \ar{r} & U
\end{tikzcd}
$$
is cartesian. We deduce a canonical isomorphism
$
\varphi_i:T_i \xrightarrow{\sim} T_{U_i}
$
such that
$
u_{U_i}\circ \varphi_i=u_i.
$
By definition, we have cartesian squares
$$
\begin{tikzcd}
T_{U_i} \ar{r} \ar{d} & \frakT_{U_i} \ar{d} & T_i \ar{r} \ar{d} & \frakT_i \ar{d} \\
T \ar{r} & \frakT & T \ar{r} & \frakT
\end{tikzcd}
$$
where the left vertical arrow in each square is étale and strict.
By the isomorphism $\varphi$ and the étaleness of $T_{U_i} \ra T,$ hence the uniqueness of $\frakT_{U_i},$ we deduce an isomorphism
$
\widetilde{\varphi}_i:\frakT_i \ra \frakT_{U_i}
$
fitting into the commutative diagram
$$
\begin{tikzcd}
T_i \ar{rr} \ar{dr}{\varphi_i} \ar{dd} & & \frakT_i \ar[dashed]{dd} \ar{dr}{\widetilde{\varphi}_i} & \\
 & T_{U_i} \ar{rr}\ar{dd} & & \frakT_{U_i} \ar{dd} \\
T \ar[dashed]{rr} \ar[equal]{dr} & & \frakT \ar[dashed]{dr} & \\
 & T \ar{rr} &  & \frakT. 
\end{tikzcd}
$$
We deduce an isomorphism
$
(U_i,\frakT_i,u_i) \xrightarrow{\sim} \left (U_i,\frakT_{U_i},u_{U_i}\right )=\alpha_{(U,\frakT,u)}(U_i)
$
and the result follows. For the case $\underline{\calE}(X/\frakS),$ we have the following cartesian rectangles:
$$
\begin{tikzcd}
T_i \ar{r} \ar[swap]{d}{\Delta_{(U_i,\frakT_i,u_i)}} & T_V \ar{d}{\Delta_{(V,\frakT_V,u_V)}} \ar{rr} & & T \ar{d}{\Delta_{u}} \\
U_i' \ar{r} & V' \ar{rr} & & U'.
\end{tikzcd}
$$
We deduce an isomorphism
$
T_i \xrightarrow{\sim} T_{U_i}.
$
The rest of the proof is similar to the previous case.
\end{proof}

\begin{lemma}\label{Kokolemcocont}
Let $u:\calC \ra \calD$ be a functor between sites whose topologies are defined by pretopologies. Suppose that $u$ satisfies the following properties:
\begin{enumerate}
\item $u$ commutes with fiber products.
\item If $(X_i \ra X)_{i\in I}$ is a covering family for the pretopology of $\calC$ then $\left (u(X_i) \ra u(X) \right )_{i\in I}$ is a covering family in $\calD.$
\item If $\left (Y_i \xrightarrow{g_i} u(X) \right )_{i\in I}$ is a covering family in $\calD$ then there exists a covering family $(f_i:X_i \ra X)_{i\in I}$ in $\calC,$ such that $u(f_i)=g_i.$
\end{enumerate}
Then $u$ is continuous and cocontinuous and induces a morphism of topoi
$$
u:\widetilde{\calC} \ra \widetilde{\calD}
$$
such that $u^{-1}$ is the composition by $u.$
\end{lemma}

\begin{proof}
Since $u$ satisfies properties 1 and 2, its continuity results from (\cite{SGA43} III 1.6). The cocontinuity of $u$ results from property 3 and (\cite{SGA43} III 2.1 and II 1.4). The induced morphism of topoi results then from (\cite{SGA43} IV 4.7).
\end{proof}

\begin{proposition}
Let $(U,\frakT,u)$ be an object of $\calE(X/\frakS)$ (resp. $\underline{\calE}(X/S)$). The functor $\alpha_{(U,\frakT,u)}$ \eqref{eqKoko1861} induces a morphism of topoi
\begin{equation}\label{Kokoalphatopos}
\alpha_{(U,\frakT,u)}:U_{\text{ét}} \ra \widetilde{\calE}(X/\frakS)\ \left( \mathrm{resp. }\ \widetilde{\underline{\calE}}(X/\frakS)\right )
\end{equation}
such that the inverse image functor is the composition by $\alpha_{(U,\frakT,u)}.$
\end{proposition}

\begin{proof}
It is sufficient to apply \ref{Kokolemcocont}. By \ref{era3paralphaKoko} (3), the functor $\alpha_{(U,\frakT,u)}$ sends covering families to covering families for the étale topology. Then, by \ref{propalphacomfiber} and \ref{propalphacocont}, the functor $\alpha_{(U,\frakT,u)}$ satisfies the conditions of \ref{Kokolemcocont}. The result follows.
\end{proof}

\begin{parag}\label{era4logflattop}
We say that a family of morphisms $((U_i,\frakT_i,u_i) \ra (U,\frakT,u))_{i\in I}$ of $\calE(X/\frakS)$ (resp. $\underline{\calE}(X/\frakS)$) is a \emph{covering for the log flat topology} if $(U_i \ra U)_{i\in I}$ is an étale covering and, for all positive integers $n,$ $(\frakT_{i,n} \ra \frakT_n)_{i\in I}$ is a log flat covering (\cite{Kat19} 2.3). By \ref{era3fiberparag}, this defines a pretopology on $\calE(X/\frakS)$ (resp. $\underline{\calE}(X/\frakS)$). We call the corresponding topology the \emph{log flat topology} and we denote by $\widetilde{\calE}_{lf}(X/\frakS)$ (resp. $\widetilde{\underline{\calE}}_{lf}(X/\frakS)$) the corresponding topos.
\end{parag}

\begin{remark}
The log flat topology on $\calE(X/\frakS)$ (resp. $\underline{\calE}(X/\frakS)$) is finer then the étale topology (\ref{parettop}). Indeed, if $\left ((U_i,\frakT_i,u_i) \ra (U,\frakT,u) \right )_{i\in I}$ is an étale covering then $(\frakT_i \ra \frakT)_{i\in I}$ is an étale and strict covering.
\end{remark}

\subsection*{Crystals}

\begin{definition}
For all positive integers $n,$ we define a presheaf of rings $\Ox_{\calE(X/\frakS),n}$ (resp. $\Ox_{\underline{\calE}(X/\frakS),n}$) on $\calE(X/\frakS)$ (resp. $\underline{\calE}(X/\frakS)$) by
$$
\Ox_{\calE(X/\frakS),n}:(U,\frakT,u) \mapsto \Gamma(\frakT_n,\Ox_{\frakT_n})\quad
(\op{resp.}\ \Ox_{\underline{\calE}(X/\frakS),n}: (U,\frakT,u)\mapsto \Gamma(\frakT_n,\Ox_{\frakT_n})).
$$
We denote by $\Ox_{\calE_{\mathrm{strict}}(X/\frakS),n}$ and $\Ox_{\underline{\calE}_{\mathrm{strict}}(X/\frakS),n}$ the restrictions of the rings $\Ox_{\calE(X/\frakS),n}$ and $\Ox_{\underline{\calE}(X/\frakS),n}$ to $\calE_{\mathrm{strict}}(X/\frakS)$ and $\underline{\calE}_{\mathrm{strict}}(X/\frakS)$) respectively.
For $n=1,$ we denote $\Ox_{\calE(X/\frakS),1}$ (resp. $\Ox_{\underline{\calE}(X/\frakS),1},$ resp. $\Ox_{\calE_{\mathrm{strict}}(X/\frakS),1}$ resp. $\Ox_{\underline{\calE}_{\mathrm{strict}}(X/\frakS),1}$) simply by $\Ox_{\calE(X/\frakS)}$ (resp. $\Ox_{\underline{\calE}(X/\frakS)},$ resp. $\Ox_{\calE_{\mathrm{strict}}(X/\frakS)},$ resp. $\Ox_{\underline{\calE}_{\mathrm{strict}}(X/\frakS)}$).
\end{definition}

\begin{proposition}
For every positive integer $n,$ the presheaves $\Ox_{\calE(X/\frakS),n}$ and $\Ox_{\underline{\calE}(X/\frakS),n}$ are sheaves for the étale and log flat topologies.
\end{proposition}

\begin{proof}
Since the log flat topology if finer then the étale topology, it is sufficient to prove the result for the log flat topology. Let $n$ be a positive integer and 
$((U_i,\frakT_i,u_i) \xrightarrow{(f_i,g_i)} (U,\frakT,u))_{i\in I}$
a covering for the log flat topology. For any $i,j\in I,$ let $f_{ij}:\frakT_{ij}=\frakT_i\times_{\frakT}^{\op{log}}\frakT_j \ra \frakT$ be the canonical morphism. Since $(\frakT_{i,n} \xrightarrow{f_i}\frakT_n)_{i\in I}$ is a covering for the log flat topology (\ref{era4logflattop}), the sequence
$$0 \ra \Ox_{\frakT_n} \ra \prod_{i\in I}f_{i*}f_i^*\Ox_{\frakT_n} \rightrightarrows \prod_{i,j\in I}f_{ij*}f_{ij}^*\Ox_{\frakT_n}$$
is exact by (\cite{SF1} 13.10). We deduce the exactness of the sequence
$$0 \ra \Gamma \left (\frakT_n,\Ox_{\frakT_n} \right ) \ra \prod_{i\in I}\Gamma \left (\frakT_{i,n},\Ox_{\frakT_{i,n}} \right ) \rightrightarrows \prod_{i,j\in I} \Gamma\left (\frakT_{ij,n},\Ox_{\frakT_{ij,n}}\right ).$$
\end{proof}

\begin{parag}\label{lindescentdata}
Let $n$ be a positive integer, $(U,\frakT,u)$ an object of $\calE(X/\frakS)$ (resp. $\underline{\calE}(X/\frakS)$) and $V$ an étale $U$-scheme. Let
$(V,\frakT_V,u_V)=\alpha_{(U,\frakT,u)}(V).$
Then
\begin{alignat*}{2}
\Gamma \left (V,\left (\Ox_{\calE(X/\frakS),n} \right )_{(U,\frakT,u)} \right ) = \Gamma \left (\frakT_{V,n},\Ox_{\frakT_{V,n}} \right ) 
= \Gamma \left (V, u_{V*}\Ox_{\frakT_{V,n}} \right ).
\end{alignat*}
It follows that
\begin{equation}\label{eqKoko18171}
\left (\Ox_{\calE(X/\frakS),n} \right )_{(U,\frakT,u)}=u_*\Ox_{\frakT_n}.
\end{equation}
The same is true for $\underline{\calE}(X/\frakS).$

Let $\F$ be a module of $\left (\widetilde{\calE}(X/\frakS),\Ox_{\calE(X/\frakS)}\right )$ or $\left (\widetilde{\underline{\calE}}(X/\frakS),\Ox_{\underline{\calE}(X/\frakS)}\right )$ and $\left (\F_{(U,\frakT,u)},\gamma_{\F,(f,g)} \right )$ the associated descent data (\ref{era3propdescentdata}). A morphism $(f,g):(U_1,\frakT_1,u_1) \ra (U_2,\frakT_2,u_2)$ induces a morphism of ringed topoi
\begin{equation}\label{Kokoftilda}
\widetilde{f}:\left (U_{1,\text{ét}},u_{1*}\Ox_{T_{1}} \right ) \ra \left ( U_{2,\text{ét}}, u_{2*}\Ox_{T_{2}} \right ).
\end{equation}
Note that this morphism $\widetilde{f}$ does not depend on $f:\frakT_1 \ra \frakT_2$ but rather on its special fiber.
The morphism
$\gamma_{\F,(f,g)}:g^{-1}\F_{(U_2,\frakT_2,u_2)} \ra \F_{(U_1,\frakT_1,u_1)}$
then induces a $u_{1*}\Ox_{T_1}$-linear morphism
\begin{equation}\label{cflindata}
c_{\F,(f,g)}: \widetilde{f}^*\F_{(U_2,\frakT_2,u_2)} \ra \F_{(U_1,\frakT_1,u_1)}.
\end{equation}
By \ref{era3propdescentdata}, if $(f,g)$ is cartesian, then
$$
\gamma_{\Ox_{\calE(X/\frakS)},(f,g)}:g^{-1}\Ox_{\calE(X/\frakS),(U_2,\frakT_2,u_2)} \ra \Ox_{\calE(X/\frakS),(U_1,\frakT_1,u_1)}
$$
is an isomorphism. This means that the morphism
$$
\gamma_{\Ox_{\calE(X/\frakS)},(f,g)}:g^{-1}u_{2*}\Ox_{U_2} \ra u_{1*}\Ox_{U_1}
$$
is an isomorphism. It follows that, if $(f,g)$ is cartesian,
$
\widetilde{f}^{-1}=\widetilde{f}^*.
$
By \ref{era3propdescentdata} and \ref{era4prop1723}, an $\Ox_{\calE(X/\frakS)}$-module (resp. $\Ox_{\underline{\calE}(X/\frakS)}$-module) $\F$ is equivalent to the data:
\begin{enumerate}
\item For every object $(U,\frakT,u)$ of $\calE(X/\frakS)$ (resp. $\underline{\calE}(X/\frakS)$), a module $\F_{(U,\frakT,u)}$ of $(U_{\text{ét}},u_*\Ox_T).$
\item For every morphism $(f,g):(U_1,\frakT_1,u_1) \ra (U_2,\frakT_2,u_2),$ a $u_{1*}\Ox_{T_1}$-linear morphism
$$c_{\F,f,g)}:\widetilde{f}^*\F_{(U_2,\frakT_2,u_2)} \ra \F_{(U_1,\frakT_1,u_1)},$$
\end{enumerate}
satisfying the following conditions:
\begin{enumerate}[(i)]
\item If $(f,g)$ is the identity then so is $c_{\F,(f,g)}.$
\item For all composable morphisms $(f_1,g_1):(U_1,\frakT_1,u_1) \ra (U_2,\frakT_2,u_2)$ and $(f_2,g_2):(U_2,\frakT_2,u_2) \ra (U_3,\frakT_3,u_3)$ in $\calE(X/\frakS)$ (resp. $\underline{\calE}(X/\frakS)$),
$$c_{\F,(f_1,g_1)} \circ \widetilde{f_1}^*c_{\F,(f_2,g_2)}=c_{\F,(f_2,g_2)\circ (f_1,g_1)}.$$
\item If $(f,g)$ is a cartesian morphism, then $c_{\F,(f,g)}$ is an isomorphism.
\end{enumerate}
This equivalence remains true if we replace $\calE(X/\frakS)$ and $\underline{\calE}(X/\frakS)$ with $\calE_{\mathrm{strict}}(X/\frakS)$ and $\underline{\calE}_{\mathrm{strict}}(X/\frakS)$ respectively.
We call $\left (\F_{(U,\frakT,u)},c_{(f,g)} \right )$ the \emph{linearized descent data associated with $\F$}.
\end{parag}

\begin{definition}\label{defcrys}
Let $\F$ be a module of $\widetilde{\calE}(X/\frakS),$ $\widetilde{\underline{\calE}}(X/\frakS),$
$\widetilde{\calE}_{lf}(X/\frakS),$ $\widetilde{\underline{\calE}}_{lf}(X/\frakS),$ $\widetilde{\calE}_{\mathrm{strict}}(X/\frakS)$ or $\widetilde{\underline{\calE}}_{\mathrm{strict}}(X/\frakS)$ and $\left (\F_{(U,\frakT,u)},c_{\F,(f,g)} \right )$ the associated linearized descent data (\ref{lindescentdata}).
\begin{enumerate}
\item We say that $\F$ is \emph{quasi-coherent} if $\F_{(U,\frakT,u)}$ is a quasi-coherent $u_*\Ox_{T}$-module for every object $(U,\frakT,u).$
\item We say that $\F$ is a crystal if $c_{\F,(f,g)}$ is an isomorphism for every morphism $(f,g).$
\end{enumerate}
We denote by $\calC(X/\frakS),$ $\underline{\calC}(X/\frakS),$ $\calC_{lf}(X/\frakS),$ $\underline{\calC}_{lf}(X/\frakS),$ $\calC_{\mathrm{strict}}(X/\frakS)$ and $\underline{\calC}_{\mathrm{strict}}(X/\frakS)$ the full subcategories of crystals of the categories of modules of $\widetilde{\calE}(X/\frakS),$ $\widetilde{\underline{\calE}}(X/\frakS),$
$\widetilde{\calE}_{lf}(X/\frakS),$ $\widetilde{\underline{\calE}}_{lf}(X/\frakS),$ $\widetilde{\calE}_{\mathrm{strict}}(X/\frakS)$ or $\widetilde{\underline{\calE}}_{\mathrm{strict}}(X/\frakS)$ respectively.
We also denote by $\calC^{\text{qcoh}}(X/\frakS),$ $\underline{\calC}^{\text{qcoh}}(X/\frakS),$ $\calC_{lf}^{\text{qcoh}}(X/\frakS),$ $\underline{\calC}_{lf}^{\text{qcoh}}(X/\frakS),$ $\calC_{\mathrm{strict}}^{\text{qcoh}}(X/\frakS)$ and $\underline{\calC}_{\mathrm{strict}}^{\text{qcoh}}(X/\frakS)$ the full subcategories of quasi-coherent crystals respectively.
\end{definition}

\begin{proposition}\label{equivcrysShiho}
The restriction functor
$$
R:\widetilde{\calE}(X/\frakS) \ra \widetilde{\calE}_{\mathrm{strict}}(X/\frakS) \ \left (\text{resp.}\ \underline{\widetilde{\calE}}(X/\frakS) \ra \underline{\widetilde{\calE}}_{\mathrm{strict}}(X/\frakS) \right )
$$
induces an equivalence of categories
$$
R_{\mathrm{crys}}:\begin{Bmatrix}\mathrm{Crystals\ of\ }\widetilde{\calE}(X/\frakS) \end{Bmatrix} \xrightarrow{\sim} \begin{Bmatrix}\mathrm{Crystals\ of\ }\widetilde{\calE}_{\mathrm{strict}}(X/\frakS) \end{Bmatrix}$$
$$\left ( \mathrm{resp.}\ \begin{Bmatrix}\mathrm{Crystals\ of\ }\widetilde{\underline{\calE}}(X/\frakS) \end{Bmatrix} \xrightarrow{\sim} \begin{Bmatrix}\mathrm{Crystals\ of\ }\widetilde{\underline{\calE}}_{\mathrm{strict}}(X/\frakS) \end{Bmatrix} \right ).
$$
\end{proposition}

\begin{proof}
We just prove the result for $\calE(X/\frakS)$ as the other case is similar. Let $f_1,f_2:\F \ra \calG$ be morphisms of crystals of $\widetilde{\calE}(X/\frakS)$ such that $R(f_1)=R(f_2).$ We want to prove that $f_1=f_2.$ Let $(U,\frakT,u)$ be an object of $\calE(X/\frakS).$ To prove that $f_{1,(U,\frakT,u)}=f_{2,(U,\frakT,u)},$ we can suppose that $U$ is affine. It follows that $T$ is also affine. Since $U\ra S$ is log smooth, there exists a log smooth logarithmic formal scheme $\frakU$ over $\frakS,$ fitting into the cartesian square
$$
\begin{tikzcd}
U \ar{r} \ar{d} & \frakU \ar{d} \\
S\ar{r} & \frakS.
\end{tikzcd}
$$
There exists then a morphism $\varphi:\frakT \ra \frakU$ fitting into the commutative diagram
$$
\begin{tikzcd}
U\ar{rr} & & \frakU \ar{d} \\
T\ar{r} \ar{u} & \frakT \ar{r} \ar{ur}{\varphi} & \frakS.
\end{tikzcd}
$$
We then get a morphism
$$
\left ( \op{Id}_U,\varphi \right ):(U,\frakT,u) \ra (U,\frakU,\op{Id}_U)
$$
in $\calE(X/\frakS)$ and we obtain a commutative diagram
$$
\begin{tikzcd}
\widetilde{\varphi}^*\F_{(U,\frakU,\op{Id}_U)} \ar{r} \ar[swap]{d}{\widetilde{\varphi}^*f_{i,(U,\frakU,\op{Id}_U)}} & \F_{(U,\frakT,u)} \ar{d}{f_{i,(U,\frakT,u)}} \\
\widetilde{\varphi}^*\calG_{(U,\frakU,\op{Id}_U)} \ar{r} & \calG_{(U,\frakT,u)},
\end{tikzcd}
$$
where $\widetilde{\varphi}$ is given in \eqref{Kokoftilda}. Since $\F$ and $\calG$ are crystals of $\widetilde{\calE}(X/\frakS),$ the horizontal morphisms are isomorphisms. Since $R(f_1)=R(f_2)$ and $(U,\frakU,\op{Id}_U)$ is an object of $\calE_{\mathrm{strict}}(X/\frakS),$ we have $f_{1,(U,\frakU,\op{Id}_U)}=f_{2,(U,\frakU,\op{Id}_U)}.$ It follows that
$
f_{1,(U,\frakT,u)}=f_{2,(U,\frakT,u)}.
$
We deduce that $f_1=f_2$ and $R_{\mathrm{crys}}$ is faithful.

Now let $\F$ and $\calG$ be crystals of $\widetilde{\calE}(X/\frakS)$ and $g:R(\F) \ra R(\calG)$ a morphism. We want to prove that there exists a morphism $f:\F \ra \calG$ in $\widetilde{\calE}(X/\frakS)$ such that $R(f)=g.$ Let $(U,\frakT,u)$ and $\varphi:\frakT \ra \frakU$ be as defined above. Just as above, we suppose that $U$ is affine. Note that the morphism of topoi
$$
\widetilde{\varphi}:\left (U_{\text{ét}},u_*\Ox_T \right ) \ra \left (U_{\text{ét}},\Ox_U \right )
$$
depends only on $u:T\ra U.$
We set
$$
g_{(U,\frakU,\op{Id}_U)}=f_{(U,\frakU,\op{Id}_U)},\ g_{(U,\frakT,u)}=\widetilde{\varphi}^*g_{(U,\frakU,\op{Id}_U)}.
$$
This is independant of the lifting $\varphi.$ We also prove that it is independant of the lifting $\frakU.$ Let $\frakU_1$ and $\frakU_2$ be log smooth logarithmic formal schemes over $\frakS$ fitting into cartesian squares
$$
\begin{tikzcd}
U \ar{r} \ar{d} & \frakU_i \ar{d} \\
S\ar{r} & \frakS.
\end{tikzcd}
$$
There exists morphisms $\varphi_1:\frakT\ra \frakU_1$ and $\varphi_{12}:\frakU_1 \ra \frakU_2$ fitting into the commutative diagram
$$
\begin{tikzcd}
 & & \frakU_2 \ar{d} \\
U \ar{r} \ar[bend right=-30]{urr} & \frakU_1 \ar{r}  \ar{ur}{\varphi_{12}} & \frakS \\
T \ar{u} \ar{r} & \frakT \ar{u}{\varphi_1} \ar{ur} & 
\end{tikzcd}
$$
Set $\varphi_2 =\varphi_{12} \circ \varphi_1.$
We hence obtain morphisms
$$
\left ( \op{Id}_U,\varphi_2 \right ):(U,\frakT,u) \xrightarrow{(\op{Id}_U,\varphi_1)} (U,\frakU_1,\op{Id}_U) \xrightarrow{(\op{Id}_U,\varphi_{12})} (U,\frakU_2,\op{Id}_U).
$$
Since $\F$ and $\calG$ are crystals of $\widetilde{\calE}(X/\frakS),$ we get a canonical isomorphism
$
\widetilde{\varphi}_{12}^*g_{(U,\frakU_2,\op{Id}_U)} \xrightarrow{\sim} g_{(U,\frakU_1,\op{Id}_U)}
$
and then, by applying $\widetilde{\varphi}_1^*,$ a canonical isomorphism
\begin{equation}\label{Shihoeq10}
\widetilde{\varphi}_2^*g_{(U,\frakU_2,\op{Id}_U)} \xrightarrow{\sim} \widetilde{\varphi}_1^*g_{(U,\frakU_1,\op{Id}_U)}.
\end{equation}
Now let $(U,\frakT,u)$ an object of $\calE(X/\frakS)$ where $U$ is not necessarily affine. Let $U=\bigcup_{i\in I}U_i$ an affine open cover of $U.$ By \ref{era3paralphaKoko} (1), we have an étale covering $\left (\alpha_{(U,\frakT,u)}\left (U_i \right ) \ra (U,\frakT,u) \right )_{i\in I}.$ The isomorphisms \eqref{Shihoeq10} satisfy the cocycle condition and hence prove that the morphisms $g_{\alpha_{(U,\frakT,u)}(U_i)}$ glue together into a morphism $g_{(U,\frakT,u)}.$ We then obtain the desired morphism $g:\F\ra \calG.$ The functor $R_{\mathrm{crys}}$ is hence full.

To prove the essential surjectivity, we consider a crystal $\calG$ of $\widetilde{\calE}_{\mathrm{strict}}(X/\frakS).$ For an object $(U,\frakT,u)$ of $\calE(X/\frakS)$ such that $U$ is affine, we consider $\varphi:\frakT \ra \frakU$ as defined above and set
$$
\F_{(U,\frakT,u)}=\widetilde{\varphi}^*\calG_{(U,\frakU,\op{Id}_U)}.
$$
By an argument similar to the one used for fullness, we define $\F_{(U,\frakT,u)}$ for any object $(U,\frakT,u)$ of $\calE(X/\frakS)$ and then obtain a crystal $\F$ of $\widetilde{\calE}(X/\frakS)$ extending $\calG.$
\end{proof}

\begin{parag}\label{equivRQDef}
Suppose that $X\ra S$ lifts to a log smooth morphism of framed logarithmic formal schemes $(\frakX,Q) \ra (\frakS,P).$ Consider the logarithmic formal schemes $R_{\frakX,1}$ and $Q_{\frakX}$ defined in \ref{parag86}. 
Let $R_1$ and $Q_1$ be the special fibers of $R_{\frakX,1}$ and $Q_{\frakX}$ and $\calR_1$ and $\calQ_1$ the corresponding Hopf algebras respectively, $q_1,q_2:R_{\frakX,1} \ra \frakX$ and $q_1',q_2':Q_{\frakX} \ra \frakX$ the canonical projections and $\lambda_R:R_1 \ra X$ and $\lambda_Q:\underline{Q_1} \ra X$ the morphisms given in (\cite{SF1} 11.4) and \ref{erafprop13}.
By \ref{Xreduced} and \ref{Qlogflat}, $(X,\frakX,\op{Id}_X)$ and $(X,R_{\frakX,1},\lambda_R)$ are objects of $\calE(X/\frakS)$ (resp. $(X,\frakX,\op{Id}_X)$ and $(X,Q_{\frakX},\lambda_Q)$ are objects of $\underline{\calE}(X/\frakS)$).
Let $\mathfrak{E}$ be a crystal of $\calC(X/\frakS)$ (resp. $\underline{\calC}(X/\frakS)$). We set $\calE=\mathfrak{E}_{(X,\frakX,\op{Id}_X)}.$ Since $\mathfrak{E}$ is a crystal, the morphisms $(\op{Id_X},q_i):(X,R_{\frakX,1},\lambda_R) \ra (X,\frakX,\op{Id}_X)$ (resp. $(\op{Id_X},q_i'):(X,Q_{\frakX},\lambda_Q) \ra (X,\frakX,\op{Id}_X)$) induce isomorphisms of $\Ox_{R_1}$-modules (resp. $\Ox_{Q_1}$-modules)
$$c_{(\op{Id}_X,q_i)}:q_i^*\calE \xrightarrow{\sim} \mathfrak{E}_{(X,R_{\frakX,1},\lambda_R)}\quad
\left (\text{resp.}\ c_{(\op{Id}_X,q_i')}:q_i'^*\calE \xrightarrow{\sim} \mathfrak{E}_{(X,Q_{\frakX},\lambda_Q)}\right ).
$$
Let $\epsilon$ be the composition
$$\epsilon = \left (c_{\op{Id}_X,q_1} \right )^{-1} \circ c_{\op{Id}_X,q_2}:q_2^*\calE \xrightarrow{\sim} \mathfrak{E}_{(X,R_{\frakX,1},\lambda_R)} \xrightarrow{\sim} q_1^*\calE$$
$$
\left (\text{resp.}\ \epsilon = \left (c_{\op{Id}_X,q_1'} \right )^{-1} \circ c_{\op{Id}_X,q_2'}:q_2'^*\calE \xrightarrow{\sim} \mathfrak{E}_{(X,Q_{\frakX},\lambda_Q)} \xrightarrow{\sim} q_1'^*\calE \right ).
$$
Then $\epsilon$ is an $\calR_1$-stratification (resp. $\calQ_1$-stratification) on $\calE$ and so we obtain a functor
\begin{equation}\label{funct1} 
\begin{array}[t]{clc c}
\calC(X/\frakS) \left ( \text{resp.\ } \underline{\calC}(X/\frakS) \right )& \ra & \begin{Bmatrix}\Ox_{X}\text{-}\mathrm{modules\ with\ an}\\\calR_1\text{-}\mathrm{stratification} \end{Bmatrix} & \left (\mathrm{resp.} \begin{Bmatrix}\Ox_{X}\text{-}\mathrm{modules\ with\ an}\\\calQ_1\text{-}\mathrm{stratification} \end{Bmatrix} \right )\\
\mathfrak{E} & \mapsto & (\calE,\epsilon). &
\end{array}
\end{equation}
\end{parag}

\begin{proposition}\label{equivRQ}
Suppose that $X\ra S$ lifts to a log smooth morphism of framed logarithmic formal schemes $(\frakX,Q) \ra (\frakS,P).$ Consider the logarithmic formal schemes $R_{\frakX,1}$ and $Q_{\frakX}$ defined in \ref{parag86}. The functor \eqref{funct1} is an equivalence of categories.
\end{proposition}

\begin{proof}
This is similar to (\cite{DXU19} 8.10) so we just outline the proof for $\calC(X/\frakS)$ (the case $\underline{\calC}(X/\frakS)$ is similar).
Let $(\calE,\epsilon)$ be an $\Ox_X$-module equipped with an $R_{\frakX,1}$-stratification. Set
$$\calE = \mathfrak{E}_{(X,\frakX,\op{Id}_X)}.$$
Let $(U,\frakT,u)$ be an object of $\calE(X/\frakS)$ such that $U$ is affine. Since $u:T\ra U$ is by definition affine, so is $T.$ Since $\frakX \ra \frakS$ is log smooth and $T$ is affine, there exists a morphism $\varphi:\frakT \ra \frakX$ fitting into the commutative diagram
$$
\begin{tikzcd}
U \ar{rr} & & \frakX\ar{d} \\
T\ar{u}{u} \ar{r} & \frakT\ar{ur}{\varphi} \ar{r} & \frakS.
\end{tikzcd}
$$
We obtain a morphism $\varphi:(U,\frakT,u) \ra (X,\frakX,\op{Id}_X).$ We set
$$\mathfrak{E}_{(U,\frakT,u)}=\widetilde{\varphi}^*\mathfrak{E}_{(X,\frakX,\op{Id}_X)},$$
where $\widetilde{\varphi}$ is given in \eqref{Kokoftilda}.
This definition of $\mathfrak{E}_{(U,\frakT,u)}$ is independant of the choice of the lifting $\varphi$ up to a canonical isomorphism induced by the stratification $\epsilon.$

Let $(f,g):(U_1,\frakT_1,u_1) \ra (U_2,\frakT_2,u_2)$ be a morphism of $\calE(X/\frakS)$ such that $U_1$ and $U_2$ are affine. We have to define a $u_{1*}\Ox_{T_1}$-linear morphism
$$
\widetilde{f}^*\frakE_{(U_2,\frakT_2,u_2)} \ra \frakE_{(U_1,\frakT_1,u_1)}.
$$
By the argument above, there exists a morphism $\varphi_2$ fitting into the commutative diagram
$$
\begin{tikzcd}
U_2 \ar{rr} & & \frakX\ar{d} \\
T_2\ar{u}{u_2} \ar{r} & \frakT_2\ar{ur}{\varphi_2} \ar{r} & \frakS.
\end{tikzcd}
$$
Let $\varphi_1$ be the composition
$\varphi_1:\frakT_1 \xrightarrow{f} \frakT_2 \xrightarrow{\varphi_2} \frakX.$
We set
$$c_{\mathfrak{E},(f,g)}:\widetilde{f}^*\mathfrak{E}_{(U_2,\frakT_2,U_2)}=\widetilde{f}^*\widetilde{\varphi_2}^*\calE \xrightarrow{\sim}  \mathfrak{E}_{(U_1,\frakT_1,u_1)},$$
where the isomorphism is the canonical one. The isomorphism $c_{\mathfrak{E},(f,g)}$ is independant of the choices so 
we can glue together the morphisms $c_{\mathfrak{E},(f,g)}$ for $U_1$ and $U_2$ affine. The cocycle condition of $\epsilon$ implies that of the morphisms $c_{\mathfrak{E},(f,g)}.$ We thus obtain linearized descent datum $\left (\mathfrak{E}_{(U,\frakT,u)},c_{\mathfrak{E},(f,g)}\right )$ and then a crystal $\mathfrak{E}$ of $\widetilde{\calE}(X/\frakS).$
The correspondance $(\calE,\epsilon) \mapsto \mathfrak{E}$ is functorial and quasi-inverse to the one given in \ref{equivRQDef}.
\end{proof}

\begin{parag}
We have a functor
\begin{equation}\label{era3rhoShiho}
\rho:\begin{array}[t]{clc}
\underline{\calE}(X/\frakS) & \ra & \calE(X'/\frakS) \\
(U,\frakT,u) & \mapsto & \left ( U',\frakT,\Delta_{u} \right ),
\end{array}
\end{equation}
where $\Delta_u$ is defined in \ref{propShiho185}.
We will show in \ref{thmproof2} that, if $X\ra S$ is of Cartier type (or equivalently saturated, by \cite{Ogus2018} III 2.5.4), the functor $\rho$ is continuous and cocontinuous for the étale and log flat topologies. It follows, by (\cite{SGA43} IV 4.7), that $\rho$ defines two morphisms of topoi
\begin{alignat}{2}
C_{X/\frakS}&:\widetilde{\underline{\calE}}(X/\frakS) \ra \widetilde{\calE}(X'/\frakS), \label{era3C} \\
C_{X/\frakS,lf}&:\widetilde{\underline{\calE}}_{lf}(X/\frakS) \ra \widetilde{\calE}_{lf}(X'/\frakS). \label{era3Clf}
\end{alignat}
For any object $(U,\frakT,u)$ of $\underline{\calE}(X/\frakS),$ we have
\begin{alignat*}{2}
\left (C_{X/\frakS}^{-1}\Ox_{\calE(X'/\frakS)} \right )(U,\frakT,u) &= \Ox_{\calE(X'/\frakS)} \circ \rho (U,\frakT,u) \\
&= \Ox_{\calE(X'/\frakS)}(U',\frakT,\Delta_u) \\
&= \Gamma(T,\Ox_T) \\
&= \Ox_{\underline{\calE}(X/\frakS)}(U,\frakT,u).
\end{alignat*}
Then
$
C_{X/\frakS}^{-1}\Ox_{\calE(X'/\frakS)}=\Ox_{\underline{\calE}(X/\frakS)}.
$
Similarly, we have
$
C_{X/\frakS,lf}^{-1}\Ox_{\calE(X'/\frakS)}=\Ox_{\underline{\calE}(X/\frakS)}.
$
It follows that $C_{X/\frakS}$ and $C_{X/\frakS,lf}$ are morphisms of ringed topoi.
\end{parag}

\begin{lemma}\label{lemfiberprodcom}
Suppose that $X\ra S$ is of Cartier type. Then the functor $\rho:\underline{\calE}(X/\frakS) \ra \calE(X'/\frakS)$ defined in \eqref{era3rho} commutes with fiber products.
\end{lemma}

\begin{proof}
Let $(f_i,g_i):(U_i,\frakT_i,u_i) \ra (U,\frakT,u),$ $i=1,2,$ be two morphisms of $\underline{\calE}(X/\frakS).$ Since $X\ra S$ is of Cartier type,
$$
U_1'=U_1'',\ U_2'=U_2'',\ U'=U''.
$$
By definition,
$$(U_1,\frakT_1,u_1) \times_{(U,\frakT,u)}(U_2,\frakT_2,u_2)=(U_1\times_UU_2,\frakT_1\times_{\frakT}^{\op{log}}\frakT_2,v),$$
where $v$ is the composition
$$v:\underline{T_1\times_T^{\op{log}}T_2} \ra \underline{T_1}\times_{\underline{T}}^{\op{log}}\underline{T_2} \ra U_1\times_UU_2$$
of the canonical morphism $\underline{T_1\times_T^{\op{log}}T_2} \ra \underline{T_1}\times_{\underline{T}}^{\op{log}}\underline{T_2}$ (\ref{era3functoriality}) and the morphism induced by $u_1$ and $u_2.$ It follows that
$$\rho((U_1,\frakT_1,u_1) \times_{(U,\frakT,u)}(U_2,\frakT_2,u_2))=(U_1'\times_{U'}U_2',\frakT_1\times_{\frakT}^{\op{log}}\frakT_2,\Delta_v).$$
On the other hand,
$$\rho(U_1,\frakT_1,u_1) \times_{\rho(U,\frakT,u)}\rho(U_2,\frakT_2,u_2)=(U_1'\times_{U'}U_2',\frakT_1\times_{\frakT}^{\op{log}}\frakT_2,w),$$
where $w:T_1\times_T^{\op{log}}T_2 \ra U_1'\times_{U'}U_2'$ is induced by $\Delta_{u_1}:T_1 \ra U_1',$ $\Delta_{u_2}:T_2 \ra U_2'$ and $\Delta_{u}:T \ra U'.$
The composition
$$
T_1\times_T^{\op{log}}T_2 \xrightarrow{G_{T_1\times_T^{\op{log}}T_2}}  \underline{T_1\times_T^{\op{log}}T_2} \xrightarrow{v} U_1\times_UU_2,
$$
where $G_{T_1\times_T^{\op{log}}T_2}$ is defined in \ref{era4prop15},
is equal to the composition
$$
T_1\times_T^{\op{log}}T_2 \ra \underline{T_1}\times_{\underline{T}}^{\op{log}}\underline{T_2} \ra U_1\times_UU_2
$$
of the morphism induced by $G_{T_i}$ and $G_T$ and the morphism induced by $u_i$ and $u.$
The diagram
$$
\begin{tikzcd}
T_1\times_T^{\op{log}}T_2 \ar{r}{G_{T_1\times_T^{\op{log}}T_2}} \ar{dr}{w} \ar[bend right=30]{ddr} & \underline{T_1\times_T^{\op{log}}T_2} \ar{dr}{v} & \\
 & U_1'\times_{U'}U_2' \ar{r} \ar{d} & U_1\times_UU_2 \ar{d} \\
& S \ar{r}{F_S} & S
\end{tikzcd}
$$
is then commutative. It follows, by definition of $\Delta_v,$ that $\Delta_v=w.$
\end{proof}

\begin{lemma}\label{lemcoco}
Suppose that $X\ra S$ is of Cartier type.
Let $(U,\frakT,u)$ be an object of $\underline{\calE}(X/\frakS)$ and $g:(W,\frakZ,w) \ra \rho(U,\frakT,u)$ a morphism of $\calE(X'/\frakS).$ Then there exists a morphism $f:(V,\frakZ,v) \ra (U,\frakT,u)$ of $\underline{\calE}(X/\frakS)$ such that $\rho(f)=g.$ In addition, if $g$ is cartesian (resp. log flat), then so is $f.$
\end{lemma}

\begin{proof}
We have a commutative diagram
$$
\begin{tikzcd}
Z \ar{r} \ar[swap]{d}{w}  & T \ar{d}{\Delta_u} \\
W \ar{r} & U'.
\end{tikzcd}
$$
By \ref{piunivhomeo}, the morphism $\pi:X' \ra X$ \eqref{diag51} is a universal homeomorphism. It follows that the functor
$$
\text{ét}_{/X} \ra \text{ét}_{/X'},\ Y\mapsto Y\times_XX',
$$
induced by $\pi,$ is an equivalence of categories (\href{https://stacks.math.columbia.edu/tag/04DZ}{Theorem 04DZ}). Let $V$ be the unique étale $X$-scheme such that $V\times_XX'=W.$ We then have cartesian squares
$$
\begin{tikzcd}
W \ar{r} \ar{d} & V \ar{d} \\
X' \ar{r}{\pi} \ar{d} & X \ar{d} \\
S \ar{r}{F_S} & S.
\end{tikzcd}
$$ 
It follows that $W=V'.$ We have an $X'$-morphism $W=V' \ra U'$ and there exists a unique $X$-morphism $V \ra U$ whose base change by $\pi$ is $V' \ra U'.$
Consider the diagram
\begin{equation}\label{diagSH1}
\begin{tikzcd}
V' \ar{rrr} & & & V \ar{d} \\
Z \ar{u}{w} \ar{r}{G_Z} & \underline{Z} \ar[dashed]{urr}{v} \ar{r} & \underline{T} \ar{r}{u} & U,
\end{tikzcd}
\end{equation}
where $G_Z$ is defined in \ref{era4prop15}. The composition
$
Z \xrightarrow{G_Z} \underline{Z} \ra \underline{T} \xrightarrow{u} U
$
is equal to
$
Z \ra T \xrightarrow{G_T} \underline{T} \xrightarrow{u} U.
$
By definition of $\Delta_u$ \eqref{propShiho185} and \eqref{morphismShiho2}, this is equal to
$
Z \ra T \xrightarrow{\Delta_u} U' \ra U.
$
This is the same as
$
Z \ra V' \ra U' \ra U.
$
This is equal to
$
Z \ra V' \ra V \ra U
$
and the diagram \eqref{diagSH1} is commutative. Since $G_Z$ is inseparable, the dashed morphism $v:\underline{Z} \ra V$ exists and is unique by (\cite{Ogus2018} IV 3.3.7). In addition, $v$ is affine since $V' \ra V$ and $w$ are affine. Hence we have an object $(V,\frakZ,v)$ of $\underline{\calE}(X/\frakS)$ and a morphism $g:(V,\frakZ,v) \ra (U,\frakT,u).$ It remains to prove that
$$
\rho(V,\frakZ,v)=(W,\frakZ,w).
$$
This is equivalent to proving that $w=\Delta_v.$
Since $X\ra S$ is of Cartier type, $V'=V''.$ The morphism $\Delta_v$ is hence the unique morphism fitting into the diagram \eqref{morphismShiho2}, which we recall here:
$$
\begin{tikzcd}
Z \ar{r}{G_Z} \ar{dr}{\Delta_v} \ar[bend right=30]{ddr} & \underline{Z} \ar{dr}{v} & \\
 & V' \ar{r} \ar{d} & V \ar{d} \\
 & S \ar{r}{F_S} & S.
\end{tikzcd}
$$
The morphism $w:Z\ra V'$ fits in the same commutative diagram by \eqref{diagSH1} and the fact that $w$ is an $S$-morphism. It is clear that $g$ is cartesian (resp. log flat) if and only if $f$ is.
\end{proof}

\begin{proposition}\label{thmproof2}
Suppose that $X\ra S$ is of Cartier type. The functor $\rho:\underline{\calE}(X/\frakS) \ra \calE(X'/\frakS)$ defined in \eqref{era3rhoShiho} is continuous and cocontinuous for the étale topology \eqref{parettop} and the log flat topology \eqref{era4logflattop}.
\end{proposition}

\begin{proof}
A family $((U_i,\frakT_i,u_i) \ra (U,\frakT,u))_{i\in I}$ of $\underline{\calE}(X/\frakS)$ is a covering for the log flat topology (resp. étale topology) if and only if $(\rho(U_i,\frakT_i,u_i) \ra \rho(U,\frakT,u))_{i\in I}$ is a covering for the log flat topology (resp. étale topology). In addition, $\rho$ commutes with fiber products (\ref{lemfiberprodcom}). It follows, by (\cite{SGA43} III 1.6), that $\rho$ is continuous for both topologies. The cocontinuity follows from (\cite{SGA43} III 2.1) and \ref{lemcoco}.
\end{proof}

\begin{lemma}\label{lem1741}
For every object $(U,\frakT,u)$ of $\underline{\calE}(X/\frakS),$ the diagram
$$
\begin{tikzcd}
\text{ét}_{/U} \ar{rr}{\alpha_{(U,\frakT,u)}} \ar[swap, sloped]{d}{\sim} & & \underline{\calE}(X/\frakS) \ar{d}{\rho} \\
\text{ét}_{/U'} \ar[swap]{rr}{\alpha_{\rho(U,\frakT,u)}} & & \calE(X'/\frakS)
\end{tikzcd}
$$
is commutative, where the left vertical arrow is $V \mapsto V'.$
\end{lemma}

\begin{proof}
This is clear from the definitions of $\alpha_{(U,\frakT,u)}$ and $\alpha_{\rho(U,\frakT,u)}$ \eqref{eqKoko1861}.
\end{proof}

\begin{lemma}\label{Kokolem1833}
Suppose that $X\ra S$ is of Cartier type. Consider the morphisms of topoi
$$C_{X/\frakS}:\left ( \widetilde{\underline{\calE}}(X/\frakS), \Ox_{\underline{\calE}(X/\frakS)}\right ) \ra \left ( \widetilde{\calE}(X'/\frakS), \Ox_{\calE(X'/\frakS)}\right )$$
and
$$C_{X/\frakS,lf}:\left ( \widetilde{\underline{\calE}}_{lf}(X/\frakS), \Ox_{\underline{\calE}(X/\frakS)}\right ) \ra \left ( \widetilde{\calE}_{lf}(X'/\frakS), \Ox_{\calE(X'/\frakS)}\right ),$$
defined in \eqref{era3Clf}. Let $(U,\frakT,u)$ be an object of $\underline{\calE}(X/\frakS)$ and $\F$ a module of $\widetilde{\calE}(X'/\frakS)$ (resp. $\widetilde{\calE}_{lf}(X'/\frakS)$). We identify $\text{ét}_{/U}$ and $\text{ét}_{/U'}$ via the exact relative Frobenius $U\ra U'.$ We denote $C_{X/\frakS}$ (resp. $C_{X/\frakS,lf}$) by $C.$ Then
$$
(C^*\F)_{(U,\frakT,u)} = \F_{\rho(U,\frakT,u)}.
$$
In addition, if $(f,g)$ is a morphism in $\underline{\calE}(X/\frakS),$ then $\rho(f,g)=(f,g')$ and
$$c_{C^*\F,(f,g)}=c_{\F,(f,g')}.$$
\end{lemma}

\begin{proof}
This is a consequence of \ref{era3propdescentdata} and \ref{lem1741} as shown in the following computation:
\begin{alignat*}{2}
(C^*\F)_{(U,\frakT,u)} &= (C^*\F)\circ \alpha_{(U,\frakT,u)} = \F \circ \rho \circ \alpha_{(U,\frakT,u)} \\
&= \F \circ \alpha_{\rho(U,\frakT,u)} = \F_{\rho(U,\frakT,u)}.
\end{alignat*}
\end{proof}

\begin{proposition}\label{era3preserve}
Suppose that $X\ra S$ is of Cartier type. Then the inverse image functors of
$$C_{X/\frakS}:\left ( \widetilde{\underline{\calE}}(X/\frakS), \Ox_{\underline{\calE}(X/\frakS)}\right ) \ra \left ( \widetilde{\calE}(X'/\frakS), \Ox_{\calE(X'/\frakS)}\right )$$
and
$$C_{X/\frakS,lf}:\left ( \widetilde{\underline{\calE}}_{lf}(X/\frakS), \Ox_{\underline{\calE}(X/\frakS)}\right ) \ra \left ( \widetilde{\calE}_{lf}(X'/\frakS), \Ox_{\calE(X'/\frakS)}\right ),$$
defined in \eqref{era3Clf}, preserve crystals (resp. quasi-coherent modules).
\end{proposition}

\begin{proof}
In this proof, and for simplicity, we will denote $C_{X/\frakS}$ (resp. $C_{X/\frakS,lf}$) by $C.$ Let $(U,\frakT,u)$ be an object of $\underline{\calE}(X/\frakS)$ and $\F$ a module of $\widetilde{\calE}(X'/\frakS)$ (resp. $\widetilde{\calE}_{lf}(X'/\frakS)$). We canonically identify $\text{ét}_{/U}$ and $\text{ét}_{/U'}.$ By \ref{Kokolem1833}, we have
\begin{alignat*}{2}
(C^*\F)_{(U,\frakT,u)} &= \F_{\rho(U,\frakT,u)}.
\end{alignat*}
In addition, if $(f,g)$ is a morphism in $\underline{\calE}(X/\frakS),$ then $\rho(f,g)=(f,g')$ and
$$c_{C^*\F,(f,g)}=c_{\F,(f,g')}.$$
It follows that if $\F$ is a quasi-coherent module (resp. crystal), so is $C^*\F.$
\end{proof}

\begin{theorem}\label{thmKoko1926}
Suppose that $X\ra S$ is of Cartier type and that it lifts to a log smooth morphism of framed logarithmic formal schemes $(\frakX,Q) \ra (\frakS,P).$ We also suppose that the exact relative Frobenius $F_1:X\ra X'$ lifts to an $(\frakS,P)$-morphism of framed logarithmic formal schemes $(\frakX,Q) \ra (\frakX',Q'),$ such that $\frakX' \ra \frakS$ is log smooth and where $Q'$ is the monoid defined in (\cite{SF1} 3.1). Consider the logarithmic formal schemes $R_{\frakX',1}$ and $Q_{\frakX}$ defined in \ref{parag86}, the morphism of formal groupoids $\nu:Q_{\frakX}\ra R_{\frakX',1}$ given in \ref{lem92}. Denote by $\nu_1:Q_1\ra R_1'$ the special fiber of $\nu$ and by $\mathcal{V}:\calR_1' \ra \calQ_1$ the corresponding morphism of Hopf algebras. Consider the diagram
\begin{equation}\label{diag185}
\begin{tikzcd}
\calC(X'/\frakS) \ar{r}{C_{X/\frakS}^*} \ar[swap, sloped]{d}{\sim} & \underline{\calC}(X/\frakS)  \ar[sloped]{d}{\sim} \\
\begin{Bmatrix}\Ox_{X'}\text{-modules\ with\ an}\\ \calR_1'\text{-stratification} \end{Bmatrix} \ar{r}{\Psi_0} & \begin{Bmatrix}\Ox_{X}\text{-modules\ with\ a}\\ \calQ_1\text{-stratification} \end{Bmatrix},
\end{tikzcd}
\end{equation}
where the vertical equivalences are given in \ref{equivRQ}, $C_{X/\frakS}$ is given in \eqref{era3C} and \ref{era3preserve} and $\Psi_0$ is defined by
$$\Psi_0:(\calE',\epsilon') \mapsto (F_1^*\calE',\mathcal{V}^*\epsilon').$$
Then the diagram \eqref{diag185} is commutative up to a canonical isomorphism.
\end{theorem}

\begin{proof}
If $\F$ is a module on $\calE(X'/\frakS)$ and $(f,g):(U_1,\frakT_1,u_1) \ra (U_2,\frakT_2,u_2)$ is a morphism of $\calE(X'/\frakS)$ then we have a $u_{1*}\Ox_{T_1}$-linear morphism \eqref{cflindata}
\begin{equation}\label{Oh181}
c_{\F,(f,g)}:\widetilde{f}^*\F_{(U_2,\frakT_2,u_2)} \ra \F_{(U_1,\frakT_1,u_1)},
\end{equation}
where 
$
\widetilde{f}:\left (U_{1,\text{ét}},u_{1*}\Ox_{T_1}\right ) \ra \left (U_{2,\text{ét}},u_{2*}\Ox_{T_2}\right )
$
is the morphism of ringed topoi \eqref{Kokoftilda}. The morphism \eqref{Oh181} is, by definition, equal to
\begin{equation}\label{Oh183}
c_{\F,(f,g)}:u_{1*}\Ox_{T_1} \otimes_{g^{-1}u_{2*}\Ox_{T_2}}g^{-1}\F_{(U_2,\frakT_2,u_2)} \ra \F_{(U_1,\frakT_1,u_1)},
\end{equation}
where the morphism $g^{-1}u_{2*}\Ox_{T_2} \ra u_{1*}\Ox_{T_1}$ is induced by the special fiber $f_1:T_1\ra T_2$ of $f.$
Let $Q_1$ and $R'_1$ be the special fibers of $Q_{\frakX}$ and $R_{\frakX',1}$ respectively, $q_1,q_2:Q_{\frakX} \ra \frakX$ and $r_1',r_2':R_{\frakX',1} \ra \frakX'$ the canonical projections and $\lambda_Q:\underline{Q_1} \ra X$ and $\lambda_{R'}:R'_1 \ra X'$ the canonical morphisms \eqref{erafprop13}.
Consider the object $(X,Q_{\frakX},\lambda_Q)$ of $\underline{\calE}(X/\frakS)$ and the objects $(X',R_{\frakX',1},\lambda_{R'})$ and $(X',\frakX',\op{Id}_{X'})$ of $\calE(X'/\frakS).$ As in \ref{equivRQDef}, $(X,\frakX,i_X)$ is an object of $\underline{\calE}(X/\frakS),$ where $i_X:\underline{X} \ra X$ is the canonical immersion, and the projections $q_1,q_2,r_1'$ and $r_2'$ define morphisms
\begin{alignat*}{2}
q_1,q_2:&(X,Q_{\frakX},\lambda_Q) \ra (X,\frakX,i_X), \\
r_1',r_2':&(X',R_{\frakX',1},\lambda_{R'}) \ra (X',\frakX',\op{Id}_{X'})
\end{alignat*}
of $\underline{\calE}(X/\frakS)$ and $\calE(X'/\frakS)$ respectively.
Let $\mathfrak{E}'$ be a crystal of $\calC(X'/\frakS).$ Denote by $\calR_1'$ the Hopf algebra corresponding to $R_1'.$ Then $\calR_1'=\lambda_{R'*}\Ox_{R_1'}.$ The morphisms \eqref{Oh183} corresponding to $r_1'$ and $r_2'$ are then equal to
$$
c_{\frakE',r_2'}:\calR_1'\otimes_{\Ox_{X'}}\frakE'_{(X',\frakX',\op{Id}_{X'})} \xrightarrow{\sim} \frakE'_{(X',R_{\frakX',1},\lambda_{R'})},
$$
$$
c_{\frakE',r_1'}:\frakE'_{(X',\frakX',\op{Id}_{X'})} \otimes_{\Ox_{X'}} \calR_1'\xrightarrow{\sim} \frakE'_{(X',R_{\frakX',1},\lambda_{R'})}.
$$
On one hand, by \ref{equivRQ}, the corresponding $\Ox_{X'}$-module with an $R_{\frakX',1}$-stratification $(\calE',\epsilon')$ is given as follows:
$$\calE'=\mathfrak{E}'_{(X',\frakX',\op{Id}_{X'})},$$
$$\epsilon':\calR_1'\otimes_{\Ox_{X'}}\calE' \xrightarrow{c_{\frakE',r_2'}} \mathfrak{E}'_{(X',R_{\frakX',1},\lambda_{R'})} \xrightarrow{c_{\frakE',r_1'}^{-1}}\calE' \otimes_{\Ox_{X'}} \calR_1'.$$
Its image by $\Psi_0$ is then the module
$$F_1^*\calE'=F_1^*\mathfrak{E}'_{(X',\frakX',\op{Id}_{X'})},$$
equipped with the stratification
$$\mathcal{V}^*\epsilon':\calQ_1\otimes_{\Ox_{X'}}\calE' \xrightarrow{\mathcal{V}^*c_{\frakE',r_2'}} \mathcal{V}^*\mathfrak{E}'_{(X',R_{\frakX',1},\lambda_{R'})}\xrightarrow{\mathcal{V}^*c_{\frakE',r_1'}^{-1}}\calE' \otimes_{\Ox_{X'}}\calQ_1.$$
This is equal to
$$
\mathcal{V}^*\epsilon':\calQ_1\otimes_{\Ox_{X}}(F_1^*\calE') \xrightarrow{\mathcal{V}^*c_{\frakE',r_2'}} \mathcal{V}^*\mathfrak{E}'_{(X',R_{\frakX',1},\lambda_{R'})}\xrightarrow{\mathcal{V}^*c_{\frakE',r_1'}^{-1}} (F_1^*\calE') \otimes_{\Ox_{X}}\calQ_1.
$$
On the other hand, by \ref{equivRQ}, the $\Ox_X$-module with a $Q_{\frakX}$-stratification $(\calE,\epsilon)$ corresponding to $\frakE:=C_{X/\frakS}^*\mathfrak{E}'$ is given as follows:
$$\calE=\left (C_{X/\frakS}^*\mathfrak{E}' \right )_{(X,\frakX,i_X)},$$
$$\epsilon:\calQ_1 \otimes_{\Ox_X} \calE\xrightarrow{c_{\frakE,q_2}} \left (C_{X/\frakS}^*\mathfrak{E}' \right )_{(X,Q_{\frakX},\lambda_Q)} \xrightarrow{c_{\frakE,q_1}^{-1}} \calE \otimes_{\Ox_X}\calQ_1.$$
By \ref{propShiho185}, we have
$
\rho(X,\frakX,i_X)=(X',\frakX,F_1).
$
By \ref{Kokolem1833}, we have
$$\calE=\left (C_{X/\frakS}^*\mathfrak{E}' \right )_{(X,\frakX,i_X)}=\mathfrak{E}'_{\rho(X,\frakX,i_X)}=\mathfrak{E}'_{(X',\frakX,F_1)}$$
and
$$
\left (C_{X/\frakS}^*\mathfrak{E}' \right )_{(X,Q_{\frakX},\lambda_Q)}=\frakE'_{\rho(X,Q_{\frakX},\lambda_Q)}=\frakE'_{(X',Q_{\frakX},\Delta_{\lambda_Q})}.
$$
The lifting $F:\frakX\ra \frakX'$ induces a morphism $(F,\op{Id}_{X'}):(X',\frakX,F_1) \ra (X',\frakX',\op{Id}_{X'})$ in $\calE(X'/\frakS).$ We then get a morphism of ringed topoi
$$
\widetilde{F}:(X'_{\text{ét}},F_{1*}\Ox_X) \ra (X'_{\text{ét}},\Ox_{X'}).
$$
Note that, since we identify $X_{\text{ét}}$ and $X'_{\text{ét}}$ via $F_1,$ we have $\widetilde{F}^*=F_1^*.$
We then get an isomorphism \eqref{cflindata}
\begin{equation}\label{zzcF}
c_{F_1}:F_1^*\calE' = F_1^*\mathfrak{E}'_{(X',\frakX',\op{Id}_{X'})} \xrightarrow{\sim} \mathfrak{E}'_{(X',\frakX,F_1)}=\calE.
\end{equation}
By \ref{lemKhaminei1}, $\nu:Q_{\frakX} \ra R_{\frakX',1}$ defines a morphism
\begin{equation}\label{zabb222}
\nu:\rho(X,Q_{\frakX},\lambda_Q)=(X',Q_{\frakX},\Delta_{\lambda_Q}) \ra (X',R_{\frakX',1},\lambda_{R'})
\end{equation}
in $\calE(X'/\frakS).$ We then have a morphism of ringed topoi
$$
\widetilde{\nu}: \left (X'_{\text{ét}}, \left (\Delta_{\lambda_Q} \right )_*\Ox_{Q_1} \right ) \ra \left (X'_{\text{ét}},\lambda_{R'*}\Ox_{R_1'} \right )=\left (X'_{\text{ét}},\calR_1' \right ).
$$
By \ref{lemKhaminei1}, we have
$$
\left (\Delta_{\lambda_Q} \right )_*\Ox_{Q_1}=F_{1*}q_{1*}\Ox_{Q_1}=F_{1*}\calQ_1.
$$
Since we identify $X_{\text{ét}}$ and $X'_{\text{ét}}$ via $F_1$ and $\frakE'$ is a crystal, we get an isomorphism
$$c_{\frakE',\nu}:\mathcal{V}^*\mathfrak{E}'_{(X',R_{\frakX',1},\lambda_{R'})} = \calQ_1 \otimes_{\calR_1'} \mathfrak{E}'_{(X',R_{\frakX',1},\lambda_{R'})} \xrightarrow{\sim} \mathfrak{E}'_{\rho(X,Q_{\frakX},\lambda_Q)},$$
fitting into the commutative diagram
$$
\begin{tikzcd}
\calQ_1 \otimes_{\Ox_X} \left (F_1^*\calE' \right ) \ar{rr}{\mathcal{V}^*c_{\frakE',r_2'}} \ar[swap]{d}{\op{Id} \otimes c_{F_1}}  & & \mathcal{V}^*\mathfrak{E}'_{(X',R_{\frakX',1},\lambda_{R'})}\ar{rr}{\mathcal{V}^*c_{\frakE,r_1'}^{-1}} \ar{d}{c_{\frakE',\nu}} & & \left (F_1^*\calE'\right ) \otimes_{\Ox_X}\calQ_1 \ar{d}{c_{F_1} \otimes \op{Id}} \\
\calQ_1 \otimes_{\Ox_X} \calE\ar{rr}{c_{\frakE,q_2}} & & \mathfrak{E}'_{\rho(X,Q_{\frakX},\lambda_Q)} \ar{rr}{c_{\frakE,q_1}^{-1}} & & \calE \otimes_{\Ox_X} \calQ_1.
\end{tikzcd}
$$
This proves that the isomorphism $c_{F_1}:F_1^*\calE' \xrightarrow{\sim} \calE$ \eqref{zzcF} underlies an isomorphism of stratified modules
$
(F_1^*\calE',\nu^*\epsilon') \xrightarrow{\sim} (\calE,\epsilon).
$
This finishes the proof.
\end{proof}

\begin{lemma}\label{lemKhaminei1}
Keep the hypothesis of \ref{thmKoko1926} and let $q_1,q_2:Q_1 \ra X$ be the canonical projections. Then $\Delta_{\lambda_Q}=F_1\circ q_1=F_1\circ q_2$ and
the diagram
\begin{equation}\label{Khaminei4}
\begin{tikzcd}
Q_1 \ar{r} \ar[swap]{d}{\Delta_{\lambda_Q}} & R_1' \ar{dl}{\lambda_{R'}} \\
X' & 
\end{tikzcd}
\end{equation}
is commutative.
\end{lemma}

\begin{proof}
Since $X\ra S$ is of Cartier type, the relative Frobenius $F_{X/S}$ is exact and hence $F_{X/S}=F_1.$
Then, by \eqref{morphismShiho2}, $\Delta_{\lambda_Q}$ is the unique morphism fitting into the commutative diagram
\begin{equation}\label{Khaminei1}
\begin{tikzcd}
Q_1\ar{r}{G_{Q_1}} \ar[swap]{dr}{\Delta_{\lambda_Q}} \ar[bend right=30]{rdd} & \underline{Q_1} \ar{dr}{\lambda_Q} & \\
 & X'\ar{d}\ar{r}{\pi} & X\ar{d} \\
 & S \ar{r}{F_S} & S.
\end{tikzcd}
\end{equation}
The composition $$Q_1 \xrightarrow{q_i} X \xrightarrow{F_1} X' \ra S,$$ where $X' \ra S$ is the canonical projection, and $Q_1 \ra S$ are clearly equal. It is hence sufficient to prove that the compositions
\begin{equation}\label{Khaminei2}
Q_1 \xrightarrow{q_i} X \xrightarrow{F_1} X'\xrightarrow{\pi} X
\end{equation}
and
\begin{equation}\label{Khaminei3}
Q_1 \xrightarrow{G_{Q_1}} \underline{Q_1} \xrightarrow{\lambda_{Q}} X
\end{equation}
are equal.
The composition \eqref{Khaminei2} is equal to
$$
Q_1\xrightarrow{q_i} X \xrightarrow{F_X} X.
$$
This is equal to
$$
Q_1 \xrightarrow{F_{Q_1}} Q_1 \xrightarrow{q_i} X
$$
and then to
$$
Q_1 \xrightarrow{G_{Q_1}} \underline{Q_1} \hookrightarrow Q_1 \xrightarrow{q_i} X.
$$
By definition of $\lambda_Q,$ this is equal to \eqref{Khaminei3}.
The commutativity of \eqref{Khaminei4} follows from the definition of $\nu$ and the fact that $\Delta_{\lambda_Q}=F_1\circ q_i.$
\end{proof}

\section{Logarithmic Oyama topoi and the Cartier transform over a trivial log point}

Let $\frakS=\op{Spf}W$ equipped with the trivial logarithmic structure and $X\ra S$ a log smooth morphism of logarithmic schemes which is of Cartier type.

\begin{parag}
By \ref{rem178} and \ref{propfT/S}, for an object $(U,\frakT,u)$ of $\underline{\calE}(X/\frakS),$ we have the following commutative diagram
$$
\begin{tikzcd}
U \ar[swap]{d}{F_{U/S}} & & \underline{T}\ar[swap]{ll}{u} \ar[swap]{d}{F_{\underline{T}/S}} \ar[hook]{rr} & & T\ar{d}{F_{T/S}} \ar{dll}{f_{T/S}} \\
U' & & \underline{T}' \ar{ll}{u'} \ar[hook]{rr} & & T',
\end{tikzcd}
$$
where the vertical arrows are the exact relative Frobenius morphisms (\ref{PFrob}), $U'=U\times_{S,F_S}S,$ $u':\underline{T}' \ra U'$ is the morphism induced by $u$ and $f_{T/S}$ is defined in \ref{propfT/S}. The commutativity of the right upper triangle follows from the commutativity of the right square, the lower right triangle and the fact that $\underline{T}' \ra T'$ is an immersion, hence a monomorphism.
In this case, by \ref{propShiho185} (3), the functor \eqref{era3rhoShiho} is given by
\begin{equation}\label{era3rho}
\rho:\begin{array}[t]{clc}
\underline{\calE}(X/\frakS) & \ra & \calE(X'/\frakS) \\
(U,\frakT,u) & \mapsto & (U',\frakT,u'\circ f_{T/S}).
\end{array}
\end{equation}
We will see in \ref{era3equivlfcrystals} that the direct image and inverse image functors of $C_{X/\frakS,lf}$ are equivalences of categories quasi-inverse to each other.  We will also see in \ref{lemdirectimage} that $C_{X/\frakS}^*$ induces a fully faithful functor between crystals. We start with some lemmas.
\end{parag}

\begin{lemma}[\cite{DXU19} 8.5]\label{era3lem1}
Let $u:T\ra U$ be an affine morphism of schemes. The functors
$$u_*:(T_{\text{ét}},\Ox_T) \ra (U_{\text{ét}},u_*\Ox_T),\ u^*:(U_{\text{ét}},u_*\Ox_T) \ra (T_{\text{ét}},\Ox_T)$$
induce equivalences of categories quasi-inverse to each other between the category of quasi-coherent $\Ox_T$-modules of $T_{\text{ét}}$ and the category of quasi-coherent $u_*\Ox_T$-modules of $U_{\text{ét}}.$
\end{lemma}

\begin{proof}
Recall that, for every scheme $X,$ $\op{QCoh}(X_{\text{zar}})$ and $\op{QCoh}(X_{\text{ét}})$ are canonically equivalent (\href{https://stacks.math.columbia.edu/tag/03DX}{Proposition 03DX}).

We start by proving that $u_*$ is exact. Let $g:\calM \ra \calN$ be a surjective morphism of quasicoherent modules of $(T_{\text{ét}},\Ox_T).$ Let $V \ra U$ be an étale morphism such that $V$ is affine. The projection $T\times_UV \ra V$ is affine and $\calM$ and $\calN$ are quasi-coherent so the morphism $(u_*g)(V)=g(T\times_UV)$ is surjective. Since $u_*$ is also left exact, it is exact.

Now let $\calM$ be a quasicoherent module of $(T_{\text{ét}},\Ox_T).$ We will prove that the adjunction morphism $u^*u_*\calM \ra \calM$ is an isomorphism. For this, we may suppose that $U$ is affine. The scheme $T$ is then also affine and hence we have an exact sequence
$$\Ox_T^{\oplus I} \ra \Ox_T^{\oplus J} \ra \calM \ra 0.$$
Since $u_*$ is exact, we deduce the exactness of the sequence
$$u_*\left (\Ox_T^{\oplus I} \right ) \ra u_* \left (\Ox_T^{\oplus J} \right ) \ra u_*\calM \ra 0.$$
The quasi-coherent module $\Ox_T^{\oplus I}$ is associated to the $\Gamma(T,\Ox_T)$-module $\bigoplus_{i\in I}\Gamma(T,\Ox_T).$ Similarly, the quasi-coherent module $\left (u_*\Ox_T \right )^{\oplus I}$ is associated to the $\Gamma(U,\Ox_U)$-module $\bigoplus_{i\in I} \Gamma(U,u_*\Ox_T)=\bigoplus_{i\in I}\Gamma(T,\Ox_T).$ It follows that the canonical morphism
$$\left (u_*\Ox_T \right )^{\oplus I} \ra u_*\left ( \Ox_T ^{\oplus I} \right )$$
is clearly an isomorphism.
Hence the quasi-coherence of $u_*\calM.$
Consider the following commutative diagram of $\Ox_T$-modules
$$
\begin{tikzcd}
u^*\left (u_*\left (\Ox_T \right )^{\oplus I}\right )  \ar{r} \ar{d} & u^* \left (u_* \left (\Ox_T\right )^{\oplus J}\right )  \ar{r} \ar{d} & u^*u_*\calM \ar{d}\ar{r} & 0 \\
\Ox_T^{\oplus I} \ar{r} & \Ox_T^{\oplus J} \ar{r}& \calM \ar{r} & 0
\end{tikzcd}
$$
Since $u_*$ and $u^*$ are both right exact, the rows are exact. The middle and left vertical arrows are clearly isomorphisms so the right vertical arrow is also an isomorphism.

Now let $\calN$ be a quasi-coherent $u_*\Ox_T$-module of $U_{\text{ét}}.$ We will prove that the adjunction morphism $\calN \ra u_*u^*\calN$ is an isomorphism. For this, we may work étale locally on $U$ and suppose that we have an exact sequence
$$\left (u_*\Ox_T\right )^{\oplus I} \ra \left (u_*\Ox_T \right )^{\oplus J} \ra \calN \ra 0.$$
Consider the following commutative diagram of $u_*\Ox_T$-modules
$$
\begin{tikzcd}
\left (u_*\Ox_T \right )^{\oplus I} \ar{r} \ar{d} & \left (u_*\Ox_T\right )^{\oplus J} \ar{r} \ar{d} & \calN \ar{r} \ar{d} & 0 \\
u_*u^*\left (\left (u_*\Ox_T\right )^{\oplus I}\right ) \ar{r} & u_*u^*\left (\left (u_*\Ox_T\right )^{\oplus J} \right ) \ar{r} & u_*u^*\calN \ar{r} & 0
\end{tikzcd}
$$
The rows are exact and the left and middle vertical arrows are clearly isomorphisms. The result follows.
\end{proof}

\begin{lemma}\label{era3lem2}
Any crystal $\F$ of $\calC^{\text{qcoh}}(X/\frakS)$ (resp. $\underline{\calC}^{\text{qcoh}}(X/\frakS)$) is a sheaf for the log flat topology of $\calE(X/\frakS)$ (resp. $\underline{\calE}(X/\frakS)$). 
\end{lemma}

\begin{proof}
We just prove the result for a crystal $\F$ of $\calC^{\text{qcoh}}(X/\frakS)$ as the proof for $\underline{\calC}^{\text{qcoh}}(X/\frakS)$ is similar. Let $(\F_{(U,\frakT,u)},c_{(f,g)})$ be the linearized descent data associated with $\F$ and let $((U_i,\frakT_i,u_i) \xrightarrow{(f_i,g_i)} (U,\frakT,u))_{i\in I}$ be a log flat covering, where $f_i:\frakT_i \ra \frakT$ and $g_i:U_i \ra U.$ For any $i,j\in I,$ let
$$(U_{ij},\frakT_{ij},u_{ij})=(U_i,\frakT_i,u_i)\times_{(U,\frakT,u)}(U_j,\frakT_j,u_j)$$
and
$$(f_{ij},g_{ij}):(U_{ij},\frakT_{ij},u_{ij}) \ra (U,\frakT,u)$$
the canonical morphism.
We will abusively denote by $f_i:T_i\ra T$ and $f_{ij}:T_{ij} \ra T$ the special fiber of $f_i$ and $f_{ij}$ respectively.
We have to prove that the sequence
\begin{equation}\label{era3exseq1}
0 \ra \F(U,\frakT,u) \ra \prod_{i\in I}\F(U_i,\frakT_i,u_i) \rightrightarrows \prod_{i,j\in I}\F(U_{ij},\frakT_{ij},u_{ij})
\end{equation}
is exact.
By \ref{era3lem1}, we can consider $\F_{(U_i,\frakT_i,u_i)},$ $\F_{(U,\frakT,u)}$ and $\F_{(U_{ij},\frakT_{ij},u_{ij})}$ as modules of $(T_{i,\text{ét}},\Ox_{T_i}),$ $(T_{\text{ét}},\Ox_T)$ and $(T_{ij,\text{ét}},\Ox_{T_{ij}})$ respectively. The sequence \eqref{era3exseq1} becomes equal to
\begin{equation}\label{era3exseq1c}
0 \ra \F_{(U,\frakT,u)}(T) \ra \prod_{i\in I}\F_{(U_i,\frakT_i,u_i)}(T_i) \rightrightarrows \prod_{i,j\in I}\F_{(U_{ij},\frakT_{ij},u_{ij})}(T_{ij}).
\end{equation}
Since $\F$ is a crystal, the morphisms
$$c_{(f_i,g_i)}:f_i^*\F_{(U,\frakT,u)} \xrightarrow{\sim} \F_{(U_i,\frakT_i,u_i)}$$
$$c_{(f_{ij},g_{ij})}:f_{ij}^*\F_{(U,\frakT,u)} \xrightarrow{\sim} \F_{(U_{ij},\frakT_{ij},u_{ij})}$$
are isomorphisms. Let $\calG=\F_{(U,\frakT,u)}.$ The sequence \eqref{era3exseq1c} becomes equal to
\begin{equation}
0 \ra \calG(T) \ra \prod_{i\in I}(f_i^*\calG)(T_i) \rightrightarrows \prod_{i,j\in I} (f_{ij}^*\calG)(T_{ij}).
\end{equation}
It is thus sufficient to prove the exactness of the sequence
\begin{equation}\label{era3exseq2}
0 \ra \calG \ra \prod_{i\in I}f_{i*}f_i^*\calG \rightrightarrows \prod_{i,j\in I} f_{ij*}f_{ij}^*\calG.
\end{equation}
This follows from (\cite{SF1} 13.10).
\end{proof}

\begin{proposition}\label{era3propequivlfetale}
The canonical functors
$$\widetilde{\calE}_{lf}(X/\frakS) \ra \widetilde{\calE}(X/\frakS),\ \widetilde{\underline{\calE}}_{lf}(X/\frakS) \ra \widetilde{\underline{\calE}}(X/\frakS)$$
induce equivalences of categories
$$\calC_{lf}^{\text{qcoh}}(X/\frakS) \xrightarrow{\sim} \calC^{\text{qcoh}}(X/\frakS),\ \underline{\calC}_{lf}^{\text{qcoh}}(X/\frakS) \xrightarrow{\sim} \underline{\calC}^{\text{qcoh}}(X/\frakS).$$
\end{proposition}

\begin{proof}
We only prove the result for $\calE(X/\frakS)$ as the proof for $\underline{\calE}(X/\frakS)$ is similar. It is clear that $\widetilde{\calE}_{lf}(X/\frakS) \ra \widetilde{\calE}(X/\frakS)$ induces a functor
\begin{equation}\label{eraf17291}
\calC_{lf}^{\text{qcoh}}(X/\frakS) \ra \calC^{\text{qcoh}}(X/\frakS).
\end{equation}
The functor $\widetilde{\calE}_{lf}(X/\frakS) \ra \widetilde{\calE}(X/\frakS)$ is clearly fully faithful and hence so is \eqref{eraf17291}. It is thus sufficient to prove that a crystal of $\calC^{\text{qcoh}}(X/\frakS)$ is a sheaf for the log flat topology. This is proved in \ref{era3lem2}.
\end{proof}

\begin{proposition}[\cite{Oyama} 4.2.1]\label{Oyamathm}
Let $\calC$ and $\calD$ be two sites, such that the topology of $\calD$ is defined by a pretopology, and $u:\calC\ra \calD$ a functor. If $u$ satisfies the following conditions:
\begin{enumerate}
\item $u$ is fully faithful.
\item $u$ is continuous and cocontinuous.
\item For any object $Y$ of $\calD,$ there exists a covering of the form $(u(X_i)\ra Y)_{i\in I}.$
\end{enumerate}
then $u$ induces an equivalence of topoi
$$u=(u^*,u_*):\widetilde{\calC} \ra \widetilde{\calD},$$
such that $u^*$ is composition with $u$ and $u_*$ is a right adjoint of $u^*$ for presheaves.
\end{proposition}

\begin{proposition}\label{thmproof1}
The functor $\rho:\underline{\calE}(X/\frakS) \ra \calE(X'/\frakS)$ defined in \eqref{era3rho} is fully faithful.
\end{proposition}

\begin{proof}
Let $(f_1,g_1),(f_2,g_2):(U_1,\frakT_1,u_1) \ra (U_2,\frakT_2,u_2)$ be two morphisms in $\underline{\calE}(X/\frakS)$ such that $\rho(f_1,g_1)=\rho(f_2,g_2).$ Since
$$\rho(f_i,g_i)=(f_i,g_i')$$
for $i=1,2,$ we get $g_1'=g_2'$ and so $g_1=g_2.$ The functor $\rho$ is then faithful. We now prove its fullness. Let $(U_1,\frakT_1,u_1)$ and $(U_2,\frakT_2,u_2)$ be two objects of $\underline{\calE}(X/\frakS)$ and $(f,h):\rho(U_1,\frakT_1,u_1) \ra \rho(U_2,\frakT_2,u_2)$ a morphism in $\calE(X'/\frakS).$ There exists a unique morphism $g:U_1 \ra U_2$ such that $h=g'.$ In the diagram
$$
\begin{tikzcd}
T_1 \ar{r}{f_1} \ar[swap]{d}{f_{T_1/S}}  & T_2\ar{d}{f_{T_2/S}} \\
\underline{T_1'} \ar{r}{\underline{f_1}'} \ar[swap]{d}{u_1'} & \underline{T_2'}\ar{d}{u_2'} \\
U_1' \ar{r}{g'} & U_2',
\end{tikzcd}
$$
the rectangle and the upper square are commutative.
Since $f_{T_1/S}$ is an epimorphism (\ref{propfT/S}), the lower square is commutative. Since the functor $T \mapsto T'$ is faithful, we get
$$u_2\circ \underline{f_1} = g\circ u_1.$$
\end{proof}

\begin{lemma}\label{thmproof3}
For any object $(W,\frakT,v)$ of $\calE(X'/\frakS),$ there exists a covering of $(W,\frakT,v)$ for the log flat topology \eqref{era4logflattop} of the form
$$\left ( \rho(U_i,\frakT_i,u_i) \ra (W,\frakT,v)\right )_{i\in I}.$$
\end{lemma}

\begin{proof}
Let $F_1:X\ra X'$ be the relative Frobenius morphism of $X$ over $S$ (which is automatically exact since $X\ra S$ is saturated by \cite{Ogus2018} III 2.5.4). 
There exists a unique strict étale morphism $V \ra X$ such that $W=V\times_XX'=V'.$ The assertion of the proposition being étale local on $W,$ we may suppose that there exists a $p$-adic logarithmic formal scheme $\frakV$ and $\frakV'$ log smooth over $\frakS,$ fitting into the cartesian squares
$$
\begin{tikzcd}
V \ar{r} \ar{d} & \frakV \ar{d} & & V' \ar{r} \ar{d} & \frakV' \ar{d} \\
S \ar{r}& \frakS & & S \ar{r} & \frakS
\end{tikzcd}
$$
We may also suppose that the restriction $F_1:V \ra V'$ and $v:T \ra V'$ lift, respectively, to morphisms $F:\frakV \ra \frakV'$ and $\frakT \ra \frakV'$ of $p$-adic logarithmic formal schemes such that the diagrams
$$
\begin{tikzcd}
V' \ar{rr} & & \frakV' \ar{d} & & V' \ar{rr} & & \frakV' \ar{d} \\
V \ar{r} \ar{u}{F_1} & \frakV \ar{r} \ar{ur}{F} & \frakS & & T\ar{u}{v} \ar{r} & \frakT \ar{ur} \ar{r} & \frakS
\end{tikzcd}
$$
are commutative.
Set $\frakZ=\frakT \times_{\frakV'}^{\op{log}} \frakV$ and $u:\underline{Z} \ra Z \ra V$ the composition of the canonical immersion $\underline{Z} \ra Z$ and the morphism $Z \ra V$ induced by the projection $\frakZ \ra \frakV.$ Since $v:T \ra V'=W$ and the canonical morphism $Z \ra T\times_{V'}V$ are affine (\cite{Ogus2018} III 2.1.6), the morphism $u$ is also affine. We obtain an object $(V,\frakZ,u)$ of $\underline{\calE}(X/\frakS)$ and a morphism
$$\rho(V,\frakZ,u)=(V',\frakZ,u'\circ f_{Z/S}) \ra (V',\frakT,v),$$
given by the projection $\frakZ \ra \frakT$ and $\op{Id}_{V'}.$ The result then follows from (\cite{Ogus2018} IV 4.2.2) applied to $F$ and the fact that the projection $Z \ra T$ is the base change of $F_1:V \ra V',$ which is a log flat covering.
\end{proof}

\begin{theorem}\label{era3equivlfcrystals}
The functor $\rho:\underline{\calE}(X/\frakS) \ra \calE(X'/\frakS),$ defined in \eqref{era3rho}, induces an equivalence of topoi
$$
C_{X/S,lf}:\widetilde{\underline{\calE}}_{lf}(X/\frakS) \ra \widetilde{\calE}_{lf}(X'/\frakS).
$$
\end{theorem}

\begin{proof}
This is a consequence of \ref{Oyamathm}, \ref{thmproof1}, \ref{thmproof2} and $\ref{thmproof3}.$
\end{proof}

\begin{theorem}\label{lemdirectimage}
The functor $C_{X/\frakS,lf}^*$ induces a fully faithful functor
$$\calC^{\text{qcoh}}(X'/\frakS) \ra \underline{\calC}^{\text{qcoh}}(X/\frakS).$$
\end{theorem}

\begin{proof}
For a ringed topos $(T,\Ox),$ we denote by $\op{Mod}(T,\Ox)$ the category of $\Ox$-modules of $T.$
By \ref{era3preserve} and \ref{era3propequivlfetale}, the commutative diagram
$$
\begin{tikzcd}
\op{Mod}\left ( \widetilde{\calE}_{lf}(X'/\frakS), \Ox_{\calE(X'/\frakS)}\right ) \ar{r}{C_{X/\frakS,lf}^*} \ar[hook]{d} & \op{Mod}\left ( \widetilde{\underline{\calE}}_{lf}(X/\frakS), \Ox_{\underline{\calE}(X/\frakS)}\right ) \ar[hook]{d} \\
\op{Mod}\left ( \widetilde{\calE}(X'/\frakS), \Ox_{\calE(X'/\frakS)}\right ) \ar{r}{C_{X/\frakS}^*} & \op{Mod}\left ( \widetilde{\underline{\calE}}(X/\frakS), \Ox_{\underline{\calE}(X/\frakS)}\right )
\end{tikzcd}
$$
induces
$$
\begin{tikzcd}
\calC^{\text{qcoh}}_{lf}(X'/\frakS) \ar{r}{C_{X/\frakS,lf}^*} \ar[swap,sloped]{d}{\sim} & \underline{\calC}^{\text{qcoh}}_{lf}(X/\frakS) \ar[sloped]{d}{\sim} \\
\calC^{\text{qcoh}}(X'/\frakS)\ar{r}{C_{X/\frakS}^*} & \underline{\calC}^{\text{qcoh}}(X/\frakS).
\end{tikzcd}
$$
By \ref{era3equivlfcrystals}, $C_{X/\frakS,lf}^*$ and the functor it induces between crystals is fully faithful, hence the result.
\end{proof}

\end{document}